\documentclass[12pt,reqno]{amsart}
\usepackage{amssymb,mathrsfs,centernot,scalerel,diagbox}
\usepackage{hyperref}

\newcommand{\nw}{\newcommand}

\nw{\A}{\mathcal{A}} \nw{\B}{\mathcal{B}} \nw{\D}{\mathcal{D}}
\nw{\E}{\mathcal{E}} \nw{\F}{\mathcal{F}} \nw{\K}{\mathcal{K}}
\nw{\M}{\mathcal{M}} \nw{\T}{\mathcal{T}} \nw{\W}{\mathcal{W}}
\nw{\X}{\mathcal{X}} \nw{\Y}{\mathcal{Y}}

\nw{\C}{\mathbb{C}} \nw{\N}{\mathbb{N}} \nw{\Q}{\mathbb{Q}}
\nw{\R}{\mathbb{R}} \nw{\Z}{\mathbb{Z}}

\nw{\cH}{\mathcal{H}} \nw{\cI}{\mathcal{I}} \nw{\cL}{\mathcal{L}}
\nw{\cN}{\mathcal{N}} \nw{\cP}{\mathcal{P}} \nw{\cQ}{\mathcal{Q}}
\nw{\cR}{\mathcal{R}} \nw{\cS}{\mathcal{S}} \nw{\cZ}{\mathcal{Z}}

\nw{\sB}{\mathscr{B}} \nw{\sC}{\mathscr{C}} \nw{\sD}{\mathscr{D}}
\nw{\sE}{\mathscr{E}} \nw{\sH}{\mathscr{H}} \nw{\sM}{\mathscr{M}}
\nw{\sR}{\mathscr{R}} \nw{\sS}{\mathscr{S}} \nw{\sX}{\mathscr{X}}

\nw{\fh}{\mathfrak{h}} \nw{\fm}{\mathfrak{m}} \nw{\fp}{\mathfrak{p}}
\nw{\fB}{\mathfrak{B}} \nw{\fC}{\mathfrak{C}} \nw{\fP}{\mathfrak{P}}

\nw{\vv}{\mathrm{v}}

\nw{\al}{\alpha} \nw{\be}{\beta} \nw{\ga}{\gamma}
\nw{\de}{\delta} \nw{\e}{\varepsilon}  \nw{\z}{\zeta}
\nw{\y}{\eta} \nw{\te}{\theta} \nw{\io}{\iota} 
\nw{\ka}{\kappa} \nw{\la}{\lambda} \nw{\x}{\xi}
\nw{\ro}{\rho} \nw{\s}{\sigma} \nw{\ta}{\tau} 
\nw{\fy}{\varphi} \nw{\om}{\omega}

\nw{\Ga}{\Gamma} \nw{\De}{\Delta} \nw{\La}{\Lambda}
\nw{\Om}{\Omega}

\nw{\p}{\partial}
\nw{\na}{\nabla}
\nw{\Cu}{\bigcup}
\nw{\Ca}{\bigcap}
\nw{\re}{\mathop{\mathrm{Re}}}
\nw{\im}{\mathop{\mathrm{Im}}}
\nw{\supp}{\operatorname{supp}}
\nw{\sign}{\operatorname{sign}}

\nw{\lec}{\lesssim}
\nw{\gec}{\gtrsim}
\nw{\etc}{,\ldots,}

\nw{\I}{\infty}
\nw{\da}{\dagger}
\nw{\empt}{\varnothing}

\nw{\ti}{\tilde}
\nw{\ck}{\check}
\nw{\ba}{\overline}
\nw{\ha}{\widehat}
\nw{\U}{\underline}

\nw{\BR}[1]{\left[#1\right]}
\nw{\LR}[1]{{\langle #1 \rangle}}
\nw{\tf}{\tfrac}
\nw{\IN}[1]{\text{ in }#1}

\nw{\ndc}{{\scaleto{\nearrow}{4pt}}}
\nw{\nin}{{\scaleto{\searrow}{4pt}}}
\nw{\emb}{\hookrightarrow}
\nw{\Div}{\operatorname{div}}
\nw{\trace}{\operatorname{tr}}
\nw{\tran}{\intercal}

\nw{\CAS}[1]{\begin{cases} #1 \end{cases}}
\nw{\mat}[1]{\begin{pmatrix} #1 \end{pmatrix}}
\nw{\EQ}[1]{\begin{equation}\begin{split} #1 \end{split}\end{equation}}
\nw{\pt}{&}
\nw{\pr}{\\ &}
\nw{\pq}{\quad}
\nw{\pQ}{\qquad}
\nw{\pn}{}
\nw{\prq}{\\ &\quad}
\nw{\prQ}{\\ &\qquad}

\nw{\tagtext}[3]{\begin{description}  \item[\hspace*{-20pt} (#1) #2] \label{#1} \ \\ #3  \end{description}}
\nw{\hyptag}[1]{\hyperref[#1]{\rm (#1)}}

\numberwithin{equation}{section}

\newtheorem{thm}{Theorem}[section]
\newtheorem{cor}[thm]{Corollary}
\newtheorem{lem}[thm]{Lemma}
\newtheorem{prop}[thm]{Proposition}
\theoremstyle{remark}
\newtheorem{rem}{Remark}[section]
\newtheorem{defn}{Definition}[section]

\title[Global wellposedness for distributions on the Fourier half space]{Global wellposedness of \\ general nonlinear evolution equations \\ for distributions on the Fourier half space}

\author{Kenji Nakanishi}
\author{Baoxiang Wang}
\address{Research Institute for Mathematical Sciences, Kyoto University, Kyoto 606-8502, Japan}
\email{kenjinakanishi@gmail.com}
\address{School of Sciences, Jimei University, Xiamen, 361021 and  School of Mathematical Sciences, Peking University, Beijing 100871, China}
\email{wbx@math.pku.edu.cn}
\thanks{The first author (K.~N.) was supported by JSPS KAKENHI Grant Numbers JP20H01814 and JP22H01132, the second author (B.~W.) was supported in part by NSFC Grant Number 12171007.} 



\date{}

\begin{document}

\maketitle
\begin{abstract}
The Cauchy problem is studied for very general systems of evolution equations, 
where the time derivative of solution is written by Fourier multipliers in space and analytic nonlinearity, with no other structural requirement. 
We construct a function space for the Fourier transform embedded in the space of distributions, and establish the global wellposedness with no size restriction. 
The major restriction on the initial data is that the Fourier transform is supported on the half space, decaying at the boundary in the sense of measure. 
We also require uniform integrability for the orthogonal directions in the distribution sense, but no other condition. 
In particular, the initial data may be much more rough than the tempered distributions, and may grow polynomially at the spatial infinity. 
A simpler argument is also presented for the solutions locally integrable in the frequency. 
When the Fourier support is slightly more restricted to a conical region, the generality of equations is extremely wide, including those that are even locally illposed in the standard function spaces, such as the backward heat equations, 
as well as those with infinite derivatives and beyond the natural boundary of the analytic nonlinearity. 
As more classical examples, our results may be applied to the incompressible and compressible Navier-Stokes and Euler equations, the nonlinear diffusion and wave equations, and so on.  
The major drawback of the Fourier support restriction is that the solutions cannot be real valued. 
\end{abstract}

\tableofcontents

\section{Introduction}
We consider general systems of nonlinear evolution equations in the form
\begin{align}
\partial_t u   = L(D)u + N(D) \sH(M(D)u), \ \ u(0,x) =u_0(x),   \label{GEE}
\end{align}
for the unknown function $u(t,x):\R^{1+d}\to\C^{n_1}$, where $L(D)=\F^{-1}\hat L(\x)\F$, $M(D)$ and $N(D)$ are given matrix-valued Fourier multipliers, defined with the Fourier transform $\F$ on $\R^d$ and the symbols $\hat{L}(\x)$, $\hat{M}(\x)$, $\hat{N}(\x)$ of matrix size $n_1\times n_1$, $n_2\times n_1$, $n_1\times n_3$ respectively, 
for some $d,n_1,n_2,n_3\in\N$. $\sH$ is a given analytic function defined on some neighborhood of $0\in\C^{n_2}$ with values in $\C^{n_3}$, satisfying $\sH(z)=o(z)$ as $z\to 0$. 

Each component of $L,M,N$ with smooth symbol may be regarded as a linear and continuous map on $\F^{-1}\sD'(\R^d)$, namely the Fourier (inverse) image of the space of Schwartz distributions $\sD'(\R^d)$, defined as the dual space of $\F\sD(\R^d)\subset\sS(\R^d)$ (for the precise definition, see \cite{Eh} where it is denoted by $\mathbf{D}'$, \cite[Chapter II, Section 1.4]{GeSh} where it is denoted by $Z'$, or the end of this section).

The goal of this paper is to establish global wellposedness of the Cauchy problem \eqref{GEE} for such a wide class of equations in a certain function space embedded in $\F^{-1}\sD'(\R^d)$, with no size restriction, but restricting the Fourier support to the half space $[0,\I)\times\R^{d-1}$. Such a kind of Fourier support conditions have definite meanings for the solutions of certain physical models. For instance, the solutions of the Constantin--Lax--Majda equation can be reduced to those of the complex heat equation which are supported in the Fourier half space, cf. \cite{ChWaWa2022}.     
The form of equations \eqref{GEE} covers major nonlinear evolution equations, such as the Euler equations and the Navier-Stokes equations, both in the incompressible case and of the compressible models, and many other evolutional PDEs of nonlinear hyperbolic, parabolic, and dispersive types, though the complex analyticity of the nonlinear part $\sH$ may well require some extension or reformulation of the equations. 

The advantage of the Fourier support restriction is to make sense of convolution between two distributions in the Fourier space, corresponding to the multiplication in the physical space.  
Without that restriction, it would be highly non-trivial even to make sense of the equations in such  generality, and of course, one should not expect the global wellposedness in any way compatible with the standard theories. 
Nevertheless, there have been many attempts for general existence of weak solutions to nonlinear PDEs. In particular, there is an extensive literature of studies based on algebras of generalized functions, those introduced by Colombeau \cite{Co} and those by Rosinger \cite{Ro}. 
Without trying to be exhaustive, we just mention two results for arbitrary types of equations: the global Cauchy-Kovalevskaya theory for analytic equations and data \cite[Chapter 2]{Ro}, and local existence for general equations and data \cite{CoHeOb}. 
Their settings are more general than ours, as the equations may contain variable coefficients and the initial data may contain arbitrary frequencies, but they also require considerable departure from the linear distribution theory, including the meanings of the generalized functions, the derivatives and the equations.

This paper is a generalization of the previous ones by the second author and others, where unique global existence of solutions was proven for the Navier-Stokes equations \cite{FeGrLiWa2021}, the complex-valued heat equations \cite{ChWaWa2022}, the nonlinear Schr\"odinger equations \cite{ChLuWa} and the nonlinear Klein-Gordon equation \cite{Wa2023}, in very rough function spaces defined by negative exponential weight in the Fourier space, into which all the standard Sobolev spaces are continuously embedded. 
Roughly speaking, besides the time decaying behavior and the lower regularity in Sobolev spaces,  the results in \cite{FeGrLiWa2021,ChWaWa2022,ChLuWa,Wa2023} may be summarized as follows: If the initial data belongs to the negative exponential Sobolev (or modulation) space and the Fourier transform is supported in the first octant $[0,\I)^d$, then those equations have unique global solutions in similar function spaces with less regularity. 

The generalization in this paper is threefold. First, the equations are very general and not specialized, with no structural condition than those mentioned above (except some regularity and boundedness of the symbols $\hat L,\hat M,\hat N$). 
Second, the function space for the solutions is much wider. 
In the previous papers, the Fourier transform was locally $L^2$ and growing at most exponentially. 
In this paper, it may be a distribution of arbitrary growth. 
In particular, the solutions in this paper may decay, grow, or be periodic in space, while the previous ones decay in average sense. 
Another benefit of the larger space is the wellposedness: solutions stay in the same space as the initial data, and are continuous for all time. 
That was not the case in the previous papers, due to the regularity loss in time.  
Third, the Fourier support is on the bigger set, namely the half space instead of the octant. 
We will also consider subsets of cone shape with possibly growing aperture, for which the conditions on $\hat L,\hat M,\hat N$ become much less restrictive, and the octant is a special case. 

Since the space $\X$ is rather complicated, we will also consider smaller function spaces where the Fourier transform is locally integrable, with a much simpler proof with more explicit information on the solutions.

\subsection{Global wellposedness results} \label{ss:GWPX}
Now let us describe the main results in this paper about the global wellposedness for \eqref{GEE}. 
For any $\vv\in\R^d\setminus\{0\}$ and $a,b\in\R$, we denote
\EQ{ \label{def half}
 \R^d_{<a}(\vv):=\{\x\in\R^d \mid \vv\cdot\x<a\}, \pq \R^d_{(a,b)}(\vv):=\{\x\in\R^d \mid a<\vv\cdot\x<b\},}
and similarly define $\R^d_{\le a}(\vv)$, $\R^d_{\ge a}(\vv)$, etc. 
One may fix $\vv=(1,0\etc 0)$ without losing generality.  
The function space $\X(\vv)\subset\sD'(\R^d)$ is defined, to be used for the Fourier transform of solutions, as the set of distributions 
$F\in\sD'(\R^d)$ with the following two properties \hyptag{X-i}--\hyptag{X-ii}:

\tagtext{X-i}{Decay at the boundary as measure}
{There exist $a\in(0,\I)$ and a complex Borel measure $\fm$ on $\R^d$, such that $F$ and $\fm$ coincide when restricted to $\R^d_{<a}(\vv)$, and the total variation $|\fm|$ satisfies $|\fm|(\R^d_{\le 0}(\vv))=0$. (It implies $\supp F\subset \R^d_{\ge 0}$.)}
 
More concretely, there exist $a\in(0,\I)$ and finite Borel measures $\fm_k$ on $\R^d$ satisfying $\fm_k(\R^d_{\le 0}(\vv))= 0$ for $k=0,1,2,3$, such that for all $\fy\in\sD(\R^d)$ with $\supp\fy\subset\R^d_{<a}(\vv)$, 
\EQ{
 F(\fy) = \sum_{k=0}^3 \int_{\R^d} i^k\fy(x)d\fm_k(x).}

\tagtext{X-ii}{Integrability for the orthogonal directions as distribution} 
{For each $a\in(0,\I)$, there exists $m\in\N$ such that $F|_{\R^d_{<a}(\vv)} \in W^{-m,1}(\R^d_{<a}(\vv))$, namely the $L^1$-based Sobolev space of the negative regularity.} 

More concretely, for each $a\in(0,\I)$, there exist finite subsets $A\subset\N_0^d$ and 
$\{f_\al\}_{\al\in A}\subset L^1(\R^d)$ such that for all $\fy\in\sD(\R^d)$ with $\supp\fy\subset\R^d_{<a}(\vv)$, 
\EQ{ \label{F iii}
 F(\fy) = \sum_{\al\in A}\int_{\R^d} \p_x^\al\fy(x) f_\al(x)dx.}

The space $\X(\vv)$ is equipped with a locally convex Hausdorff topology such that $\X(\vv)\subset\sD'(\R^d)$ is continuous. See Section \ref{s:X} for the definition and the properties of $\X(\vv)$ (and Lemma \ref{lem:defX} for consistency with the above description). 
Examples of $u_0\in\F^{-1}\X(\vv)$ may be found in \eqref{u0 Hs many periods}.  

The condition \hyptag{X-i} implies that $u_0\in\F^{-1}\X(\vv)\subset\F^{-1}\sD'(\R^d)$ can not be real-valued unless it is $0$, because $u_0$ is real-valued if and only if $\F u_0(-\x)=\ba{\F u_0(\x)}$. 
So we can not deal with any non-trivial real-valued solutions under that condition. 
This is the major and huge disadvantage for applications, and it is the same feature as in the previous papers \cite{FeGrLiWa2021,ChWaWa2022,ChLuWa,Wa2023}. 
All the surprising generality rests on this price, which also precludes potential contradictions with the standard results.

For the symbols of the Fourier multipliers in \eqref{GEE}, we assume

\tagtext{C-s}{Coefficients of semi-linear type}
{$\hat L,\hat M,\hat N$ are smooth on $\R^d_{>0}(\vv)$ such that
$\p_\x^\al e^{t\hat L(\x)}(t^\te \hat N(\x))^j$ and $\p_\x^\al \hat M(\x)$ 
are bounded for $j\in\{0,1\}$, $0<t<1$ and $\x\in\R^d_{(1/n,n)}(\vv)$, 
for some $\te\in[0,1)$, every $n\in\N$ and every $\al\in\N_0^d$. 
Without the derivative ($\al=0$), they are required to be bounded also on $\R^d_{(0,1)}(\vv)$.} 

Note that the symbols are allowed to be discontinuous at the boundary $\vv\cdot\x=0$. 
\begin{thm}
Let $d\in\N$, $\vv\in\R^d\setminus\{0\}$ and assume \hyptag{C-s}. Then the Cauchy problem of \eqref{GEE} is globally wellposed in $(\F^{-1}\X(\vv))^{n_1}$ on $t\ge 0$. 
\end{thm}
See Corollary \ref{GWP-XR} and Theorem \ref{thm:gex unif} for the more precise versions. 

The above result is general but also restrictive about the equation when $d\ge 2$, since the symbols are required to be uniformly bounded in the orthogonal directions. 
Roughly speaking, the nonlinear part of the equation is required to have less derivatives in the orthogonal directions  than the linear smoothing effect of $L$. 
We may remove this limitation by imposing a stronger condition on the Fourier support than the half space. For any closed subset $\cR\subset\R^d$, let 
\EQ{
 \X_\cR(\vv):=\{F\in\X(\vv) \mid \supp F\subset\cR\}. }
Since $\X(\vv)\subset\sD'(\R^d)$ is continuous, it is a closed subspace of $\X(\vv)$. 
In this case, we may deal with much more general symbols, namely

\medskip

\tagtext{C-0}{Coefficients bounded at $0$}
{$\hat L,\hat M,\hat N$ are smooth on a neighborhood of $\cR\cap\R^d_{>0}(\vv)$ and bounded as $\x\to 0$.} 

\begin{thm} \label{thm:gex supp cond}
Let $d\in\N$, $\vv\in\R^d\setminus\{0\}$, and $\cR\subset\R^d_{\ge 0}(\vv)$ be a closed subset such that $\cR+\cR\subset\cR$ and that $\cR\cap\R^d_{<a}(\vv)$ is bounded for each $a\in(0,\I)$. Assume \hyptag{C-0}. Then the Cauchy problem of \eqref{GEE} is globally wellposed in $(\F^{-1}\X_\cR(\vv))^{n_1}$ on $t\in\R$. 
\end{thm}
It is proven after Corollary \ref{GWP-XR} as an immediate consequence. 
The simplest choice of $\cR$ is the cone $\{|\x|\le \la(\vv\cdot\x)\}$ with a constant $\la>0$, which approaches $\R^d_{\ge 0}(\vv)$ as $\la\to\I$. 
The left side $|\x|$ may be replaced with any norm on $\R^d$. 
Another example is the octant $[0,\I)^d$ with $\vv\in(0,\I)^d$. 
More generally, $\cR$ may grow arbitrarily fast as $\vv\cdot\x\to\I$: For any $C^1$ function $R:[0,\I)\to[0,\I)$ with $R(0)=0$, there exists $\cR\supset\{|\x|\le R(\vv\cdot\x)\}$ with the require properties (cf.~Lemma \ref{lem:super-additive}). 

In the above results, the symbols may grow arbitrarily fast as $|\x|\to\I$ on $\cR$ (which is possible only as $\vv\cdot\x\to\I$), meaning that the equation may contain arbitrary order of derivatives, including infinity. 
Hence the global wellposedness in $\F^{-1}\X_\cR(\vv)$  holds even for extremely bad PDEs such as 
\EQ{
 u_t = -\De u + u^3, \pq 
 u_t = e^{e^{-\De}}[e^{(e^{-\De}u)^2}-1],}
where the sign in front of $\De$ is the opposite one from the usual settings. 
Such equations are not even locally wellposed in the standard function spaces, and the ODE versions are not globally wellposed either. 
See Section \ref{ss:beyond nb} for another example, where the solutions may live outside of the natural boundary of the analytic nonlinearity $\sH$. 

The above generality indicates how strong is the assumption of the Fourier support, which essentially eliminates all possible nonlinear feedback. 
In fact, the core of the proof is to construct weight functions in the Fourier space, depending on time, which control the nonlinear flow of frequencies to $\vv\cdot\x\to\I$, so that the standard contraction argument can work on the subspaces forming $\X$ characterized by the weights. 
The main novelty compared with the previous work \cite{FeGrLiWa2021,ChWaWa2022,ChLuWa,Wa2023} is that the procedure works without using any explicit form of the equations or the weights. 
This flexibility is crucial to get the wellposedness in the entire function space $\F^{-1}\X$. 

\subsection{Solutions locally integrable in frequency}
Because of the vast generality of the equation and the initial data, it is not easy to extract concrete information on the behavior of solutions and the weights from the proof in $\X$. 
Restricting the initial data to the Fourier transform of locally integrable functions, we may simplify the construction of solutions considerably, with more explicit information on the behavior. 
This version is still more general than the previous work on the equations and the initial data. 

In this case, we first fix a closed $\cR\subset\R^d_{\ge 0}(\vv)$ satisfying $\cR+\cR\subset\cR$, together with a continuous and super-additive function $\fp:\cR\to[0,\I)$ (namely $\fp(\x)+\fp(\y)\le\fp(\x+\y)$). Then for the Fourier multipliers in \eqref{GEE}, we assume

\tagtext{C-m}{Measurable coefficients}
{$\hat L,\hat M,\hat N$ are measurable on $\cR$ satisfying
\EQ{ \label{bd LNMp}
 \s_{\max}(\hat L+\hat L^*) \le 2\te\fp, \pq \max_{j,k}|\hat N_{j,k}|+|\hat M_{l,m}| \lec 1+\fp,}
for all $\x\in\cR$ and some constant $\te\in(0,1)$, where $\hat L^*$ is the adjoint of $\hat L$ and 
$\s_{\max}$ is the maximal eigenvalue of the self-adjoint matrix.} 

If $\fp:\cR\to[0,\I)$ satisfies \eqref{bd LNMp} with locally bounded $|\x|^{-1}\fp(\x)$, then there exists $\ti\fp:\cR\to[0,\I)$ that is continuous and super-additive satisfying $\ti\fp\ge\fp$, by Lemma \ref{lem:super-additive}. Then we may replace $\fp$ with $\ti\fp$. 

Next, for any $s<0$, we consider the weighted $L^1$ space
\EQ{
 \pt Z^s_0:=\{\fy\in L^1_{loc}(\R^d) \mid \supp\fy\subset\cR,\ \|\fy\|_{Z^s_0}:= \|(\fp^{-1}+\fp)e^{s\fp}\fy\|_{L^1(\cR)}<\I\}.}
More generally, we may choose any Banach space $Z^s\subset L^1_{loc}(\R^d)$ satisfying 
\EQ{ \label{cond Zs}
 \supp\fy\subset\cR, \pq \sup_{c\in\cR}\|\fy(\x-c)\|_{Z^s}+\|\fy\|_{Z^s_0} \lec \|\fy\|_{Z^s}=\||\fy|\|_{Z^s}}
for all $\fy\in Z^s$. $Z^s_0$ is the maximal choice for $Z^s$. Another example is the anisotropic space for $\x:=(\bar\x,\ti\x)\in\R^{d_0}\times\R^{d-d_0}$ with 
\EQ{
 \pt \fp(\x):=\la(\vv\cdot\x+(\vv\cdot\x)^m), \pq \cR:=\{\x\in\R^d \mid |\bar\x| \le \la\fp(\x)\},
 \pr Z^s:=\{\fy\in Z^s_0 \mid \|\fy\|_{Z^s}:=\|(\fp^{-\s}+1)\fy\|_{L^r_{\bar\x}L^1_{\ti\x}}<\I\},}
for some $d_0\in\{1\etc d\}$, $\la>0$, $m\ge 1$, $r\in[1,\I]$ and $\s>1-d_0/r'$.

\begin{thm}
Let $d\in\N$, $\vv\in\R^d\setminus\{0\}$, $\cR\subset\R^d_{\ge 0}(\vv)$ be closed satisfying $\cR+\cR\subset\cR$, and let $\fp:\cR\to[0,\I)$ be continuous and super-additive. 
Assume \hyptag{C-m}. Let $s<0$ and let $Z^s\subset L^1_{loc}(\R^d)$ be a Banach space satisfying \eqref{cond Zs}. 
Then for any $B\in(0,\I)$, there exists $b\in[1,\I)$ such that 
for any $u_0\in (\F^{-1}L^1_{loc}(\R^d))^{n_1}$ satisfying $\supp\F u_0\subset\cR$ and $\|e^{s\fp^2}\F u_0\|_{Z^s}\le B$, 
the Cauchy problem of \eqref{GEE} has a unique global solution $u$ on $t\ge 0$ satisfying 
$e^{b(s\fp^2-t\fp)}\F u(t) \in Z^s(L^\I_t(0,\I))^{n_1}$. 
\end{thm}

See Theorem \ref{thm:global L1} for the more precise version, including the case where nonlinearity $\sH$ is decaying at $0$ in higher order. 
Note that the above is not a wellposedness result, because the solution is in a much rougher space than the initial data, due to the loss of regularity. 
In the previous case of $\X$, the wellposedness holds thanks to its topological structure, namely the inductive limit, but huge loss of regularity is going on inside $\X$. 

\subsection{Outline of the paper}
In Section \ref{s:X}, we give the precise definition and some basic properties of the space $\X$. 
It is a locally convex topological vector space, defined as the inductive limit of an uncountable directed system of Banach spaces in $\sD'(\R^d)$. 
We give a local basis at $0\in\X$, as well as characterization of bounded subsets and compact subsets of $\X$. The latter is crucial for the proof of wellposedness in $\F^{-1}\X$. 

In Section \ref{s:conv}, we investigate the convolution in those Banach spaces forming $\X$. 
Lemma \ref{lem:gen weight} is crucial for the contraction argument to solve the Cauchy problem, where the point is that arbitrarily small factors may be gained from the convolution on the half space with the boundary decay, if the weights are appropriately chosen. 

In Section \ref{s:GWP}, we prove the global wellposedness in $\F^{-1}\X$, including the cases of stringer support conditions. Solutions periodic in space are also considered as the special case where the support is a lattice. 

In Section \ref{s:comp}, we compare those solutions in $\F^{-1}\X$ with others, observing for the simplest  ODE and for the Burgers equation that they go beyond the classical blow-up. 
The example shows also that the global wellposedness can not hold within the tempered distributions. 
Another example shows that the global solutions exist even if the initial data are completely outside of the natural boundary of the nonlinearity $\sH$. 

In Section \ref{s:opt}, we investigate optimality of the space $\X$ or the conditions \hyptag{X-i}--\hyptag{X-ii}, as well as the stronger support condition, for the very general wellposedness to hold, using again the simple ODE. 

In Section \ref{s:L1}, we consider the solutions locally integrable in the Fourier space, and give a much simpler argument for the unique global existence of solutions. 

Finally in Section \ref{s:Ex}, our results are applied to several examples of equations.

\subsection{Notations}
The sets of non-negative integers are denoted by $\N_0=\{0,1,2\etc\}$ and $\N=\{1,2\etc\}$. 
$1_A$ denotes the characteristic function of any set $A$. 

In any seminormed space $V$, open and closed balls are denoted by 
\EQ{
 B_V(c,r):=\{y\in V\mid \|c-y\|_V<r\}, \pq \bar B_V(c,r):=\{y\in V\mid \|c-y\|_V\le r\},}
with center $c\in V$ and radius $r>0$. 
For vectors $x\in V^n$ of finite dimension $n\in\N$, the seminorm is defined by $\|x\|_V:=[\sum_{k=1\etc n}\|x_k\|_V^2]^{1/2}$. 
For two normed spaces $X,Y$, the Banach space of linear bounded operators from $X$ to $Y$ is denoted by $\B(X\to Y)$. 
The normed space of $m\times n$ complex matrices is denoted by $M(m,n)=\B(\C^n\to\C^m)$. 
That of square matrices is abbreviated $M(n)=M(n,n)$. The adjoint and the transpose of $A\in M(m,n)$ are respectively denoted by $A^*,A^\tran\in M(n,m)$. For any vector $\x\in\R^d$, the components are denoted by 
\EQ{
 \x=(\x_1\etc\x_d)=(\x_1,\x^\perp)\in\R\times\R^{d-1}.} 
The $L^2$ inner product on $\R^d$ is denoted by 
\EQ{
 \LR{f|g} := \int_{\R^d} f(x)\ba{g(x)}dx.}
$W^{m,p}(\Om)$ denotes the $L^p$-Sobolev space of regularity $m\in\R$ on an open set $\Om\subset\R^d$, and $H^s(\R^d)=W^{s,2}(\R^d)$. 
$B^s_{p,q}(\R^d)$ denotes the inhomogeneous $L^p$-Besov space of regularity $s\in\R$ and interpolation index $q\in[1,\I]$. 
$\sD(\Om)$ is the space of smooth test functions compactly supported on $\Om$, and $\sS(\R^d)$ denotes the Schwartz class. $\sD'(\Om)$ is the dual space or the space of distributions, and $\sS'(\R^d)$ is the space of tempered distributions. 
All the function spaces are considered with complex values.

For any $\vv\in\R^d$ and $a,b\in\R$, we define  
$\R^d_{<a}(\vv)$ and $\R^d_{(a,b)}(\vv)$ by \eqref{def half}, 
and similarly $\R^d_{\le a}(\vv)$, $\R^d_{>a}(\vv)$, and $\R^d_{\ge a}(\vv)$. 
When $\vv=e_1:=(1,0\etc 0)$, it is omitted as $\R^d_{<a}:=\R^d_{<a}(e_1)$, etc. 
Similarly, for any $Z\subset\sD'(\R^d)$, we denote
\EQ{
 \pt Z_{<a}:=\{f\in Z \mid \supp f\subset\R^d_{<a}\}, \pq Z_{(a,b)}:=\{f\in Z \mid \supp f\subset\R^d_{(a,b)}\}}
as well as for the other subscripts. 
More generally, for any $O\subset\R^d$, we denote 
\EQ{
 Z_O:=\{f\in Z \mid \supp f\subset O\}.}
When $O$ is an open set in $\R^d$, $\sD_O$ is often identified with $\sD(O)$. 
Strictly speaking, they are homeomorphic by the restriction to $O$ and the zero extension to the outside. 

For cut-off arguments, we fix a $C^\I$ function $\chi:\R\to[0,1]$ satisfying
\EQ{
 t\le 1 \implies \chi(t)=1, \pq t\ge 2 \implies \chi(t)=0.}
For any $l>0$, define $\chi^l(t):= \chi(t/l)$ and $\chi_l(t):=1-\chi(t/l)$, 
so that we have
\EQ{ 
 \pt \supp\chi^l \subset (-\I,2l), \pq
 \supp\chi_l \subset (l,\I). }
The smooth cut-off operators in $\x_1$ are defined by 
\EQ{ \label{def X-cut}
 (X^l\fy)(\x):=\chi^l(\x_1)\fy(\x), \pq (X_l\fy)(\x):=\chi_l(\x_1)\fy(\x).}
For mollifiers, we fix $\z\in C^\I(\R^d)$ satisfying 
\EQ{
 0\le \z(\x)=\z(|\x|), \pq \supp\z\subset B_{\R^d}(0,1), \pq \|\z\|_{L^1(\R^d)}=1,}
and let $\z_n(\x):=n^d\z(n\x)\to\de(\x)$ in $\sD'(\R^d)$ as $n\to\I$. 

For any $\de\ge 0$ and any interval $I\subset[0,\I)$, let 
\EQ{
 \pt C^\ndc:=\{f:[0,\I)\to[0,\I) \mid \text{continuous and non-decreasing}\},
 \pr C^\nin:=\{f:[0,\I)\to[0,\I) \mid \text{continuous and non-increasing}\},
 \pr C^\ndc_I:=\{f\in C^\ndc \mid f([0,\I))\subset I\},
 \pq C^\nin_I:=\{f\in C^\nin \mid f([0,\I))\subset I\},
 \pr C^\ndc_\de:=\{f\in C^\ndc \mid f(0)=0,\ f([0,2\de])\subset[0,\tf12]\},
 \pr C^\ndc_+:=\{f\in C^\ndc \mid f(0)=0,\ t>0\implies f(t)>0\}.}
$C^\ndc_\de$ is closed for $\max$, so it is a directed set for the pointwise order. 
Note that $C^\ndc_0=\Cu_{0<\de\le 1}C^\ndc_\de\supset C^\ndc_+$. We will use $C^\ndc_0$ as an index set for the inductive limit to construct function spaces of distributions.

For topological vector spaces $X,Y$, the continuous embedding is denoted by $X\emb Y$. 
For any $m\in[0,\I)$, the Banach space of bounded $C^m$ functions on $\R^d$ is denoted by 
\EQ{
 \B^m:=\{\fy\in C^m(\R^d;\C) \mid \|\fy\|_m:=\max_{|\al|\le m}\sup_{\x\in\R^d}|\p^\al\fy(\x)|<\I\}.}
Note that $\B^m=\B^{m_0}$ for the largest $m_0\in\N$ satisfying $m_0\le m$. 
The trivial extension to non-integer $m$ is for notational convenience. 
Let $\B^0_0$ be the closure of $\sD(\R^d)$ in $\B^0$ and let $\M:=(\B^0_0)'$ be its dual space. 
Each $\fm\in\M$ may be regarded as a (bounded) complex-valued Borel measure on $\R^d$, whose total variation is a finite Borel measure on $\R^d$, denoted by $|\fm|$, satisfying 
\EQ{
 |\fm|(\R^d)/2 \le \|\fm\|_{(\B^0_0)'} =: \|\fm\|_\M \le 2|\fm|(\R^d).}
We also introduce the subspace vanishing on the left:
\EQ{
 \M_+ :=\{ \fm\in\M \mid |\fm|(\R^d_{\le 0})=0\},}
then the Dirichlet condition \hyptag{X-i} with $\vv=e_1$ is equivalent to existence of $l>0$ such that $X^lF\in\M_+$. 
See also Definitions \ref{def:Bm+}--\ref{def:X} for some other function spaces. 

We define the Fourier transform for $\fy,\psi\in\sS(\R^d)$ by
\EQ{
 \hat\fy(\x)=\F\fy(\x):=\tf{1}{(2\pi)^d}\int_{\R^d}\fy(x)e^{-ix\x}dx,
 \pq \F^{-1}\psi(x)=\int_{\R^d}\psi(\x)e^{i\x x}d\x}
which is extended as usual to the tempered distributions $F\in\sS'(\R^d)$ by duality 
\EQ{ \label{def F}
 (\F F)(\fy)=\hat F(\fy):=F(\hat\fy),}
for $\fy\in\sS(\R^d)$. 
The above definition of $\F$ may be less conventional, but it simplifies the constant in the convolution, namely, for any $f\in\sS$ and any $g\in\sS'$, 
\EQ{ \label{F conv}
 \F(fg)=(\F f)*(\F g).}

The Fourier inverse image of the entire space of distributions 
$\F^{-1}\sD'(\R^d)$ is defined as the dual space of 
\EQ{
 \F\sD(\R^d):=\{\F\fy\mid \fy\in\sD(\R^d)\} \subset\sS(\R^d),}
which has the natural topology induced by $\F$ from $\sD(\R^d)$, so that $\F:\sD(\R^d)\to\F\sD(\R^d)$ is homeomorphic. 
These spaces have been studied from the early time of distributions; 
$\F^{-1}\sD'(\R^d)$ is the same as $\mathbf{D}'$ in \cite{Eh} and $Z'$ in \cite[Chapter II]{GeSh}. 
 
The Fourier transform $\F:\F^{-1}\sD'(\R^d)\to\sD'(\R^d)$ is defined by \eqref{def F} for $\fy\in\sD(\R^d)$, satisfying the same algebraic formulas as in $\sS'$. 
In particular, \eqref{F conv} holds for $f\in\F\sD$ and $g\in\F^{-1}\sD'$. 
It is immediate from the definition that $\F:\F^{-1}\sD'(\R^d)\to\sD'(\R^d)$ is homeomorphic, which justifies the notation $\F^{-1}\sD'(\R^d)$. 
More generally, for any topological linear space $X\emb\sD'(\R^d)$, its Fourier inverse image is denoted by
\EQ{
 \F^{-1}X:=\{\F^{-1}f\mid f\in X\}\emb\F^{-1}\sD'(\R^d),} 
with the natural topology induced by $\F^{-1}$, or the initial topology for $\F$, so that $\F:\F^{-1}X\to X$ is homeomorphic.

Since $\sD\emb\sS$ is dense, so is $\F\sD\emb\sS$, hence we have 
$\sS'(\R^d)\emb\F^{-1}\sD'(\R^d)$, which is also dense. 
Indeed, for the standard Littlewood-Paley decomposition using a dyadic partition in the Fourier space $\{\fy_j\}_{j\in\N}\subset\F\sD(\R^d)$, each piece $\F^{-1}\sD'(\R^d)\ni F\mapsto \fy_j*F\in\sS'(\R^d)$ is continuous, and $F=\sum_{j\in\N}\fy_j*F$ on $\F^{-1}\sD'(\R^d)$. 
In such a way, $\F^{-1}\sD'(\R^d)$ is a natural space for the analysis of PDE on $\R^d$, in contrast with $\sD'(\R^d)$. In the latter, the Littlewood-Paley decomposition is not well-defined, 
and both the heat and Schr\"odinger equations are illposed in both time directions, whereas they are wellposed in $\F^{-1}\sD'$ in both time directions.

Note that if $\F$ is replaced with $\F^{-1}$ in the above definition, then we have 
$\F\sD=\F^{-1}\sD\emb\sS$ including the topology, so we could naturally write $\F^{-1}\sD'=\F\sD'$, but we refrain from using this identification, to avoid possible ambiguity or confusion. 
It is also worth noting that $\F^{-1}\X\cap\F\X=\{0\}$ because of the anisotropy.

\section{Spaces of distributions for the Fourier transform of solutions} \label{s:X}
Here we introduce some function spaces consisting of distributions in $\sD'(\R^d)$, for the Fourier transform of the solutions to the Cauchy problem \eqref{GEE}. 
In the following, we fix $\vv:=e_1=(1,0\etc 0)$, without losing generality, since the other case follows just by a linear transform of the spatial coordinates in $\R^d$. 
The main purpose of this section is to define and characterize the function space $\X\emb\sD'(\R^d)$. 

\begin{defn} \label{def:Bm+}
Let $d\in\N$. For $\de,m\in[0,\I)$, we define
\EQ{
 \pt \|\fy\|_m^+:=\max_{|\al|\le m}\sup_{\x\in\R^d_{\ge 0}}|\p^\al\fy(\x)|,
 \pq  \|\fy\|_{m,\de}:=\|X^\de\fy\|_0^+ + \|X_\de\fy\|_m^+, 
 \pr \B^m_+:=\{\fy \in C^m(\R^d;\C) \mid \|\fy\|_m^+<\I\},
 \pq \B^\I_+ := \Ca_{m\in\N}\B^m_+. }
For any $l,m\in[0,\I)$, and $F\in\sD'(\R^d)$, we define
\EQ{ \label{def rolm}
 \ro_l^{m}(F):= \sup\{\re F(\fy) \mid \fy\in\sD_{<l},\ \|\fy\|_m^+ \le 1\}. }
\end{defn}

The functional $F\in\sD'$ may be extended to $\B^m_{+<l}$ if $\ro^m_l(F)<\I$, using that $\sD_{<l}\subset C^m_{<l}$ is dense for the $C^m$ topology (namely for the local uniform convergence including the partial derivatives up to order $m$). 
This extension is needed to define the convolution. 
\begin{lem} \label{lem:ext F}
Let $F\in\sD'(\R^d)$, $l>0$, $m\ge 0$ and $\ro_l^m(F)<\I$. 
Then $F\in\sD'_{\ge 0}$ and 
\EQ{ \label{unif dec F}
 \lim_{R\to\I}\ro_l^m(\chi_R(|\x|)F)=0.} 
Hence $F$ is extended to a unique linear functional $\B^m_{+<l}\to\C$ that is continuous with respect to the $C^m$ topology. 
Moreover, for all $\fy\in \B^m_{+<l}$, 
\EQ{ \label{bd F}
 |F(\fy)| \le \ro_l^m(F)\|\fy\|_m^+.}
\end{lem}
Note that the space $\B^m_+$ does not assume any decay or uniform continuity for $|\x|\to\I$. 
\begin{proof}
For $\fy\in\sD_{<l}$, \eqref{bd F} is immediate from the definition. 
If $\fy\in\sD_{<0}$, then $\|\fy\|_m^+=0$ implies $F(\fy)=0$, hence $\supp F\subset\R^d_{\ge 0}$. 
Next we prove \eqref{unif dec F}. 
If it is false, then there exist sequences $\fy_n\in\sD_{<l}$ and $R_n\to\I$ such that $\|\fy_n\|_m^+\le 1$, $F(\chi_{R_n}(|\x|)\fy_n)=1$, and $\supp\fy_n\subset B_{\R^d}(0,R_{n+1})$. Let 
\EQ{
 \psi_N:=\sum_{1\le n \le N} \chi_{R_n}(|\x|)\fy_n \in\sD_{<l}. }
Since the summand is supported on $R_n<|\x|<R_{n+1}$, which is mutually disjoint, we have $\|\psi_N\|_m^+ \lec 1$, whereas $F(\psi_N)=N$, contradicting \eqref{bd F} on $\sD_{<l}$.
Hence \eqref{unif dec F} is true. 
It implies that $F|_{\sD_{<l}}$ is continuous with respect to the $C^m$ topology. 
Then its unique extension to $\B^m_{+<l}$ is given by the density of $\sD_{<l}$. 
In fact, it may be defined by using mollifiers, 
$F(\fy):=\lim_{n\to\I}F(\z_n*\chi^n(|\x|)\fy)$. 
\end{proof}

\begin{defn} \label{def:X}
For $\mu,\nu\in C^\ndc$, $l\in(0,\I)$, and $F\in\sD'(\R^d)$, define
\EQ{
 \pt \ro^{\mu,\nu}_l(F):=\inf\{B\ge 0 \mid 0<t\le l\implies \ro^{\mu(t)}_t(F)\le B\nu(t)\}, \pq(\inf\empt=\I),  
 \pr \ro^{\mu,\nu}(F):=\sup_{l>0}\ro^{\mu,\nu}_l(F),
 \pq \X^{\mu,\nu}:=\{F\in\sD'(\R^d)\mid \ro^{\mu,\nu}(F)<\I\}.}
$\X^{\mu,\nu}\emb\sD'(\R^d)$ is a Banach space. 
For any $\mu_0,\mu_1,\nu_0,\nu_1\in C^\ndc$ satisfying $(\mu_0,\nu_0)\le(\mu_1,\nu_1)$ (pointwisely), we have 
$\ro^{\mu_1,\nu_1}(F) \le \ro^{\mu_0,\nu_0}(F)$ and $\X^{\mu_0,\nu_0}\emb\X^{\mu_1,\nu_1}$, so that we may consider their union as the inductive limit in the category of locally convex topological vector spaces, denoted by
\EQ{
 \X:=\Cu_{\mu,\nu\in C^\ndc_0} \X^{\mu,\nu} \emb \sD'(\R^d).}
The subscript $0$ in the index set $C^\ndc_0$ requires $\mu(0)=\nu(0)=0$, which is corresponding to the boundary condition \hyptag{X-i} at $\x_1=0$, and crucial for our analysis of the PDE \eqref{GEE}. 
When the direction $e_1$ is replaced with a general $\vv\in\R^d\setminus\{0\}$, it is defined by 
\EQ{
 \X(\vv) := \{F(U_\vv \x) \mid F\in\X\},}
where $U_\vv$ is any invertible linear transform of $\R^d$ such that $U_\vv \vv=e_1$. 
\end{defn}
In this paper, we consider the inductive limits only in the category of locally convex topological vector spaces. See e.g., \cite{Bour} about basic properties of them, though we will use only a few of them. 
$\X\emb\sD'(\R^d)$ ensures that $\X$ is Hausdorff. 
Roughly speaking, $-\mu(l),\nu(l)$ are respectively the regularity order and the size of $F$ on $\R^d_{<l}$. 
By definition, we have for any $F\in\X^{\mu,\nu}$, $l>0$ and $\fy\in\sD_{<l}$,
\EQ{ \label{bd on Ffy}
 |F(\fy)| \le \nu(l)  \ro^{\mu,\nu}_l(F) \|\fy\|_{\mu(l)}^+,}
where $\ro^{\mu,\nu}_l(F)$ is the optimal non-negative number satisfying this inequality. Lemma \ref{lem:ext F} extends it to $\fy\in\B^{\mu(l)}_{+<l}$. 
If $\mu\in C^\ndc_0$ then $\mu\in C^\ndc_\de$ for some $\de>0$, $\mu(l)\le1/2<1$ for $l\le2\de$, so that we may weaken the norm for $\fy$ as 
\EQ{ \label{bd on Ffy-de}
 |F(\fy)| \pt\le |F(X^\de\fy)|+|F(X_\de\fy)|
  \le \nu(l)\ro^{\mu,\nu}_l(F)\|\fy\|_{\mu(l),\de}.}

To change the size weight $\nu$, we have 
\begin{lem} \label{lem:weight change}
For any $\ka\in C^\ndc$, $\mu,\nu\in C^\ndc$, $F\in\sD'(\R^d)$ and $l\in(0,\I)$, we have 
\EQ{ \label{size weight mult}
 \ro^{\mu,\nu}_l(F) \le \ka(l)\ro^{\mu,\ka\nu}_l(F).}
Let $\mu_j,\nu_j\in C^\ndc$ ($j=0,1$) and $T\in\B(\X^{\mu_1,\nu_1}\to\X^{\mu_0,\nu_0})$ satisfy
\EQ{ 
 \ro^{\mu_0,\nu_0}_l(TF) \le B(l)\ro^{\mu_1,\nu_1}_l(F),}
for some non-deacreasing $B:[0,\I)\to[0,\I)$, all $F\in\X^{\mu_1,\nu_1}$ and all $l\in(0,\I)$. Then for any bounded $\ka\in C^\ndc$, any $F\in\X^{\mu_1,\ka\nu_1}$ and $l\in(0,\I)$, we have
\EQ{ \label{size weight tran}
 \ro^{\mu_0,\ka\nu_0}_l(TF) \le B(l)\ro^{\mu_1,\ka\nu_1}_l(F).}
\end{lem}
\begin{proof}
For any $0<t\le l<\I$ and $F\in\sD'$, we have 
\EQ{
 \ro^{\mu(t)}_t(F) \le \ro^{\mu,\ka\nu}_l(F)\ka(t)\nu(t) \le \ka(l)\ro^{\mu,\ka\nu}_l(F)\nu(t),}
which implies \eqref{size weight mult}. 
Using \eqref{size weight mult}, we have 
\EQ{
 \ro^{\mu_0(t)}_t(TF) \pt\le \nu_0(t)\ro^{\mu_0,\nu_0}_t(TF) 
 \pn\le \nu_0(t)B(t)\ro^{\mu_1,\nu_1}_t(F)  
  \le \nu_0(t)B(l)\ka(t)\ro^{\mu_1,\ka\nu_1}_t(F),}
which implies \eqref{size weight tran}. 
The boundedness of $\ka$ was used merely to ensure $F\in\X^{\mu_1,\nu_1}$. 
\end{proof}

In the analysis of the space $\X$, a basic tool is the smooth cut-off $X_\de,X^\de$ in $\x_1$. 
Here we prepare a few estimates on them. 
For the factor coming from the Leibniz rule, we define $M_\de^\chi \in C^\ndc_{[1,\I)}$ for 
all $\de\in(0,1]$, first on $[0,\tf12]\cup\N$ by  
\EQ{
 M_\de^\chi(l)=\CAS{1+l, &(0\le l\le \tf12),\\
 \sum_{k=0}^{l} \tf{l!}{k!(l-k)!}\|\p^k\chi_\de\|_{L^\I(\R)}, &(l\in\N),} }
and then extend it to $[0,\I)$ as a continuous and increasing function, e.g., the piecewise linear function. 

\begin{lem} \label{lem:mult}
Let $\de,l\in(0,\I)$, $m\in[0,\I)$ and $F\in\sD'(\R^d)$. Then 
\EQ{
 \ro^m_l(X_\de F) \le M^\chi_\de(m)\ro^m_l(F), \pq \ro^m_l(X^\de F) \le M^\chi_\de(m)\ro^m_{\min(l,2\de)}(F).}
Let $\mu,\nu,\ti\mu,\ti\nu\in C^\ndc$ satisfying $\mu\le\ti\mu$. Let $0<L,H,l<\I$ and $F\in\sD'(\R^d)$. 
\begin{enumerate}
\item If $M^\chi_L(\mu)\nu \le \ti\nu$ on $[L,\I)$, 
then $\ro_l^{\ti\mu,\ti\nu}(X_L F) \le \ro_l^{\mu,\nu}(F)$. 
\item If $M^\chi_H(\mu)\nu \le \ti\nu$ on $(0,2H]$, 
then $\ro_l^{\ti\mu,\ti\nu}(X^H F) \le \ro_l^{\mu,\nu}(F)$. 
\item If $M^\chi_L(\mu)M^\chi_H(\mu) \nu \le \ti\nu$ on $[L,2H]$, then $\ro_l^{\ti\mu,\ti\nu}(X_LX^H F) \le \ro_l^{\mu,\nu}(F)$. 
\end{enumerate}
\end{lem}
\begin{proof}
By the Leibniz rule, we have for any $m\in[0,\I)$ and $\fy\in\sD$, 
\EQ{
 \max(\|X^\de\fy\|_m^+,\|X_\de\fy\|_m^+) \pt\le \sup_{|\al|\le m} \sum_{0\le k\le\al_1} \tf{\al_1!}{(\al_1-k)!k!}\|\chi_\de^{(k)}(\x_1)\|_0^+\|\p^{\al-ke_1}\fy\|_0^+ 
 \pr\le M^\chi_\de(m)\|\fy\|_m^+.}
Hence for any $\fy\in\sD_{<l}$ we have, using $\supp X^\de\fy\subset\R^d_{<2\de}$ as well, 
\EQ{
 \pt |X_\de F(\fy)| \le \ro^m_l(F)\|X_\de\fy\|_m^+ \le \ro^m_l(F) M^\chi_\de(m) \|\fy\|_m^+,
 \pr |X^\de F(\fy)| \le \ro^m_{\min(l,2\de)}(F)\|X^\de\fy\|_m^+ \le \ro^m_{\min(l,2\de)}(F) M^\chi_\de(m)\|\fy\|_m^+.}
from which the estimates in $\ro^m_l$ follow.
 
In the case of (1), we have $X_L F(\fy)=0$ for $l\le L$ and $\fy\in\sD_{<l}$, while for $l\ge L$,  
\EQ{
 \ro^{\ti\mu(l)}_l(X_L F) \le \ro^{\mu(l)}_l(X_L F) \le M^\chi_L(\mu(l))\ro_l^{\mu,\nu}(F)\nu(l) \le  \ti\nu(l)\ro_l^{\mu,\nu}(F).}
In the case of (2), we have similarly for all $l>0$ and $l':=\min(l,2H)$
\EQ{
 \ro^{\ti\mu(l)}_l(X^H F) \le \ro^{\mu(l')}_{l'}(X^H F) \pt\le M^\chi_H(\mu(l')) \ro^{\mu(l')}_{l'}(F)
 \le \ti\nu(l) \ro_l^{\mu,\nu}(F).}
In the case of (3), we have similarly for $l\ge L$, 
\EQ{
 \ro^{\ti\mu(l)}_l(X_LX^H F) \le M^\chi_L(\mu(l'))M^\chi_H(\mu(l')) \ro^{\mu(l')}_{l'}(F)
 \le \ti\nu(l) \ro_l^{\mu,\nu}(F).} 
Thus we obtain the estimates in $\X^{\ti\mu,\ti\nu}$ for all the cases (1)--(3). 
\end{proof}

\subsection{Characterization of $\X$}
Now we see that the definition of $\X$ in Section \ref{ss:GWPX} is the same as in this section. 
\begin{lem} \label{lem:defX}
For any $F\in\sD'(\R^d)$, $F\in\X$ is equivalent to \hyptag{X-i}--\hyptag{X-ii} with $\vv=e_1$. 
Moreover, for any $F\in\X$ and $\fy\in\sD(\R^d)$, we have 
\EQ{ \label{E X}
 F(\fy) = \lim_{l\to+0}F(X_l\fy).}
\end{lem}
The above identity \eqref{E X} allows us to ignore the test functions $\fy$ on $\R^d_{\le 0}$. More precisely, let $\sR:\sD'(\R^d)\to\sD'(\R^d_{>0})$ denote the restriction from $\R^d$ to $\R^d_{>0}$, 
and for any $G\in\sD'(\R^d_{>0})$ and $\fy\in\sD(\R^d)$, let 
\EQ{
 \sE G(\fy) := \lim_{l\to+0}G(X_l\fy).}
The limit does not converge for general $G\in\sD'(\R^d_{>0})$, but when it does for all $\fy\in\sD(\R^d)$, we have $\sE G\in\sD'(\R^d)$. \eqref{E X} may be rewritten as: $F=\sE\sR F$ for all $F\in\X$. 
\begin{proof}
For any $\de>0$, $\mu\in C^\ndc_\de$, $\nu\in C^\ndc_0$ and any $F\in\X^{\mu,\nu}$, 
Lemma \ref{lem:ext F} implies $\supp F\subset\R^d_{\ge 0}$
and \eqref{bd on Ffy} for any $l>0$ and $\fy\in\B^{\mu(l)}_{+<l}$. 
In particular, 
\EQ{
 |(X^\de F)(\fy)| \le \nu(l)\ro^{\mu,\nu}(F)\|\fy\|_0,}
for all $\fy\in\B^0$, so we have $X^\de F\in\M$ with the total variation satisfying
\EQ{
 |X^\de F|(\R^d_{<l}) \le \nu(l)\ro^{\mu,\nu}(F)\to 0 \pq(l\to+0).}
Hence by the dominated convergence, $X^\de F\in\M_+$, 
It also implies $F(X^l\fy)\to 0$ as $l\to+0$ for any $\fy\in\sD$, hence $F(\fy)=\sE \sR F(\fy)$.

For larger $\x_1$, it suffices to show for each $l>0$ that $X^lF\in W^{-m,1}(\R^d)$ for any $m\in\N$ satisfying $m-1\ge \mu(2l)$. 
Henceforth we ignore the dependence of constants on $F,l,m$. 
Define $\{c_\al\}_{|\al|\le m}\subset\Z$ by
\EQ{
 (1-\De)^m = \sum_{|\al|\le m}c_\al \p^{2\al},}
and let $\psi_n^{(\al)}:=(1-\De)^{-m}c_\al\p^\al\z_n\in\sS$ and $\psi_n^{(\al)\da}(\x):=\psi_n^{(\al)}(-\x)$. Then 
we have $\sum_{|\al|\le m}\p^\al\psi_n^{(\al)}=\z_n$. 
Let $f_n:=\z_n*X^lF$. Then $f_n\in C^\I$, and we have from \eqref{bd on Ffy}, 
\EQ{
 |\LR{f_n|\bar\fy}| = |X^lF(\z_n*\fy)| \lec \|\z_n*\fy\|_{m-1} \lec n^{m-1}\|\fy\|_0 }
for all $\fy\in\sD$, which implies $\|f_n\|_{L^1(\R^d)}\lec n^{m-1}$. Similarly we have, for $|\be|\le 1$, 
\EQ{
  |\LR{\p^\be\psi_n^{(\al)}*f_n|\bar\fy}| \pt= |X^lF(\p^\be \psi_n^{(\al)\da}*\z_n* \fy)| 
 \pr\lec \|\p^\be \psi_n^{(\al)\da}*\z_n*\fy\|_{m-1} \lec \|\z_n*\fy\|_0 \lec \|\fy\|_0.}
Hence $\psi_n^{(\al)}*f_n$ is bounded in $W^{1,1}(\R^d)$, and so convergent along a subsequence in $L^1_{loc}(\R^d)$ for $|\al|\le m$. 
Let $g_\al\in L^1(\R^d)$ be the limit. Then
\EQ{
 X^lF(\fy) = \lim_{n\to\I}\LR{\z_n*f_n|\bar\fy}\pt=\lim_{n\to\I} \sum_{|\al|\le m}\LR{\psi_n^{(\al)}*f_n|(-\p)^\al\bar\fy}
 \pr=\sum_{|\al|\le m}\LR{g_\al|(-\p)^\al \bar\fy}.}
Hence $X^lF=\sum_{|\al|\le m}\p^\al g_\al\in W^{-m,1}(\R^d)$. Thus \hyptag{X-i}--\hyptag{X-ii} are satisfied.

Conversely, suppose that $F\in\sD'(\R^d)$ satisfies \hyptag{X-i}--\hyptag{X-ii}. 
By \hyptag{X-i}, we have $|F(\fy)|\le 2|\fm|(\R^d_{<l})\|\fy\|_0^+$ for $\fy\in\sD_{<l}$ and small $l>0$, 
and by \hyptag{X-ii}, $|F(\fy)|\le\|F|_{\R^d_{<l}}\|_{W^{-m,1}}\|\fy\|_m^+$ for any $l>0$, $\fy\in\sD_{(0,l)}$ and some $m=m(l)\in\N$. 
Since $|\fm|(\R^d_{<l})\to|\fm|(\R^d_{\le 0})=0$ as $l\to +0$, 
we may find $\mu,\nu\in C^\ndc_0$ such that \eqref{bd on Ffy} holds for all $l>0$ and $\fy\in\sD_{<l}$. 
Then $F\in\X^{\mu,\nu}\subset\X$. 
\end{proof}

\subsection{Local basis in $\X$}
Here we characterize the topology of $\X$ by constructing a local basis at $0$. 
For any $m\in\N_0$ and $\nu\in C_+^\ndc$, let 
\EQ{
 \sB^m(\nu):=\{F\in\X \mid \sup_{l>0}\ro^m_l(F)/\nu(l) \le 1\}.}
For any $R_0,R_1\in(0,\I)$ and any $N=\{N_m\}_{m\in\N}\subset C_+^\ndc$, let 
\EQ{
 \pt \sB(N):=\Cu_{M\in\N}\sum_{1\le m\le M} \sB^m(N_m), 
 \pr \sB(R_0,R_1,N):=\bar B_{\M_+}(0,R_0) + \X_{\ge R_1} + \sB(N). }
Since $\sB(R_0,R_1,N)$ is the increasing limit of the sum of balls and a subspace, it is an absolutely convex subset of $\X$. Let $\sB_\X$ be the collection of all such $\sB(R_0,R_1,N)$. 
\begin{prop}
$\sB_\X$ is a local basis at $0$ of $\X$. 
In other words, a subset of $\X$ is a neighborhood of $0\in\X$ iff it contains a member of $\sB_\X$. 
\end{prop}
\begin{proof}
First we show that every $\sB(R_0,R_1,N)\in\sB_\X$ is a $0$-neighborhood. 
Since it is absolutely convex, by definition of the inductive limit $\X=\Cu_{\mu,\nu}\X^{\mu,\nu}$, 
it is equivalent that its intersection with $\X^{\mu,\nu}$ is a $0$-neighborhood of $\X^{\mu,\nu}$ for every $\mu,\nu\in C_0^\ndc$. 

Let $\de\in(0,R_1/4)$, $\mu\in C_\de^\ndc$, $\nu\in C_0^\ndc$ and $F\in\X^{\mu,\nu}$. 
Let $F_0:=X^\de F$, $F_2:=X_{R_1}F$ and $F_1:=X_\de X^{R_1}F$. Then 
$F_0\in\M_+$ with $\|F_0\|_\M \le \nu(2\de)\ro^{\mu,\nu}(F)$, $F_1\in \X_{(\de,2R_1)}$ and $F_2\in\X_{\ge R_1}$.  
By Lemma \ref{lem:mult}, $X_\de X^{R_1}$ is bounded $\X^{\mu,\nu}\to\X^{m,N_m}$ for  $m:=\mu(2R_1)$. Hence if $\ro^{\mu,\nu}(F)$ is small enough 
then $\|F_0\|_\M \le R_0$, $\ro^{m,N_m}(F_1) \le 1$, and $F\in\sB(R_0,R_1,N)$. So $\sB_\X$ consists of $0$-neighborhoods in $\X$.

Next let $U\subset\X$ be an absolutely convex $0$-neighborhood. 
If  $\bar B_{\M_+}(0,R_0)\not\subset U$ for all $R_0>0$, then there is a sequence $F_n\in\M_+\setminus U$ such that $\|F_n\|_\M<1/n^2$. Let $\nu_n(l):=\|X^{2l}F_n\|_\M \in C^\ndc_0$ 
and $\nu:=\sup_{n\in\N}n\nu_n$. 
Then $n\nu_n\le n\|F_n\|_\M<1/n$ implies that $\nu\in C^\ndc_0$, while 
$\ro^0_l(F_n)\le \nu_n(l) \le \tf1n\nu(l)$ for all $l>0$ 
implies $\ro^{0,\nu}(F_n)\le \tf 1n\to 0$, so $F_n\to 0$ in $\X^{0,\nu}\emb\X$, contradicting $F_n\not\in U$. Hence $\bar B_{\M_+}(0,R_0)\subset U$ for some $R_0>0$. 

Next, if there is no $R_1<\I$ such that $\X_{\ge R_1}\subset U$, then there are sequences $\mu_n,\nu_n\in C^\ndc_0$ and $F_n\in\X^{\mu_n,\nu_n}_{>n}\setminus U$. 
Define $\mu,\nu\in C^\ndc_0$ first at each $l\in\N$ by
\EQ{
 \mu(l)=\max_{1\le n\le l} \mu_n(l+1), \pq \nu(l)=\max_{1\le n\le l} n\ro^{\mu_n,\nu_n}(F_n)\nu_n(l+1),}
and then extend them to $[0,\I)$. 
Then we have $\mu_n(l)\le \mu(l)$ and $n\ro^{\mu_n,\nu_n}(F_n)\nu_n(l)\le \nu(l)$ for all $n\in\N$ and $l\ge n$, so $\ro^{\mu(l)}_l(F_n) \le \nu_n(l)\ro^{\mu_n,\nu_n}(F_n) \le \tf1n \nu(l)$, while 
$\supp F_n$ implies $\ro^{\mu(l)}_l(F_n)=0$ for $l\le n$.  
Hence $\ro^{\mu,\nu}(F_n)\le \tf 1n\to 0$, so $F_n\to 0$ in $\X^{\mu,\nu}\emb\X$, contradicting $F_n\not\in U$. Hence $\X_{\ge R_1}\subset U$ for some $R_1\in[1,\I)$.

It remains to consider the $\sB^m(N_m)$ part. 
Let $R_j:=2^{-j+1}R_1$ and $\nu_j(l):=\tf{l}{2R_j}\in C_+^\ndc$ for $j\in\N$.   
Since $U\subset\X$ is a $0$-neighborhood,  there is $\{\e_{m,j}\}_{m,j\in\N}\subset(0,\I)$ such that 
$\bar B_{\X^{m\nu_j,\nu_j}}(0,\e_{m,j})\subset U$. 
Take any $\{N_m\}_{m\in\N}\subset C_+^\ndc$ satisfying for all $m,j\in\N$
\EQ{
 N_m(2R_j) \le [M^\chi_{R_{j+1}}(m)M^\chi_{R_j}(m)]^{-1}2^{-j-3}\e_{m,j}. }
For any $m,j\in\N$ and $F\in \sB^m(N_m)$, let $F_j:=X_{R_{j+1}}X^{R_j}F$. Then Lemma \ref{lem:mult} implies
\EQ{
 \ro^{m\nu_j,\nu_j}(F_j)\le \nu_j(R_{j+1})^{-1}2^{-j-3}\e_{m,j}\ro^{m,N_m}(F) \le 2^{-j-1}\e_{m,j},}
so $F_j\in 2^{-j-1}U$, while $F\in\X$ implies $\|X^{R_J}F\|_\M<R_0/4$ for some $J\in\N$. Then 
\EQ{
 F=X^{R_J}F + X_{R_1}F + \sum_{1\le j<J}F_j  \in U/4 + U/4 + \sum_{1\le j<J}2^{-j-1}U  \subset U,}
by the absolute convexity of $U$, and similarly  
$\sB(R_0/4,R_1,\{2^{-m-1}N_m\}_{m\in\N}) \subset U$.   
Hence $\sB_\X$ is a local basis of $\X$. 
\end{proof}

The local basis $\sB_\X$ of $\X$ implies the following embeddings. For $\de>0$ and $m\in\N$, define the spaces of distributions 
\EQ{ \label{embedded in X}
 \pt \X_L := \{F \in \sD'_{\ge 0} \mid \{X^lF\}_{l>0} \subset L^1(\R^d)\}, 
 \pr \X_\M:= \{F\in \sD'_{\ge 0} \mid \{X^lF\}_{l>0} \subset \M_+\},  
 \pr W^{-m,1}_{\de}:=\{F\in\sD'_{\ge\de} \mid \{X^lF\}_{l>0} \subset W^{-m,1}(\R^d)\},}
equipped with the initial topology for the family of mappings $\{X^l\}_{l>0}$. Then 
\EQ{ \label{embs}
 L^1_{\ge 0}\emb \left\{\begin{aligned} &\M_+ \\ &\X_L \end{aligned}\right\} \emb \X_\M \emb \X, \pq W^{-m,1}_{\de} \emb \X.}
Indeed, $L^1(\R^d)\subset\M$ is a closed subspace with the same norm. 
If $F\in L^1(\R^d)$ satisfies $\supp F\subset\R^d_{\ge 0}$, then $F\in\M_+$ follows from the dominated convergence, hence $L^1_{\ge 0}\emb\M_+$ and $\X_L\emb\X_\M$.  
$L^1_{\ge 0}\emb \X_L$ and $\M_+\emb\X_\M$ are obvious. 
$\X_\M\emb\X$ because $\sB(R_0,R_1,N)\supset (X^{R_1})^{-1}B_{\M_+}(0,R_0)$, and  
$W^{-m,1}_{\de} \emb \X$ because
\EQ{
 \sB(R_0,R_1,N) \supset \{F\in\sD'_{\ge\de}\mid \|X^{R_1}F\|_{W^{-m,1}}<N_m(\de)\},}
which follows from the duality $|G(\fy)| \le \|G\|_{W^{-m,1}}\|\fy\|_m^+$ for $G\in\sD'_{\ge\de}$ and $\fy\in\sD$. 

To justify very weak solutions in $\F^{-1}\X$, the density of nice functions is crucial, for which the flexibility of weights $\mu,\nu$ is necessary. In fact, $\sD(\R^d)$ is not dense in any $\X^{\mu,\nu}$, but we may construct approximate identities $I_n:\X\to\sD_{>0}$ as follows. 
\begin{prop} \label{prop:dense X} 
Let $\{c_n\}_{n\in\N}\subset\R^d$ and $\{r_n\}_{n\in\N}\subset(0,\I)$ satisfy $|c_n|+r_n\to 0$ as $n\to\I$ and $c_n\in\R^d_{>r_n}$ for all $n\in\N$. 
Let $\fy_n(\x):=r_n^{-d}\z(r_n^{-1}(\x-c_n))$, $\psi_n(\x):=\chi(r_n|\x|)$ and $I_nF:=\fy_n*(\psi_n F)$. Then $I_n:\X\to\sD_{>0}$ is continuous for each $n\in\N$, and $I_nF\to F$ in $\X$ as $n\to\I$ for all $F\in\X$. 
For any $\cR\subset\R^d_{\ge 0}$ such that $\cR+\cR\subset\cR$ and the closure of the interior contains $0\in\R^d$, we may choose $c_n,r_n$ such that $I_n$ preserves $\supp F\subset\cR$.  
\end{prop}
Those $I_n$ are just the standard mollification moving the support to the right. 
\begin{proof}
By definition, $\fy_n,\psi_n\in\sD(\R^d)$ and $\supp\fy_n\subset B_{\R^d}(c_n,r_n)\subset\R^d_{>0}$. Hence $F\in\sD'_{\ge 0}$ implies $I_nF\in\sD_{>0}$. 
If the closure of the interior of $\cR$ contains $0$, then we may choose $c_n,r_n$ such that $B_{\R^d}(c_n,r_n)\subset\cR$. Then for $\supp F\subset\cR$, we have $\supp(I_nF)\subset\supp\fy_n+\supp F \subset\cR+\cR\subset\cR$. 
The continuity of $I_n$ follows from that of $I_n:\sD'(\R^d)\to\sD(\R^d)$, which is well-known. 

It remains to prove $I_nF\to F$ in $\X$. 
Take any $F\in\X$ and $\sB(R_0,R_1,N)\in\sB_\X$. 
By Lemma \ref{lem:defX}, there exist $\de\in(0,R_1/2)$ and $m\in\N$ such that 
\EQ{
 X^\de F\in\M_+, \pq \|X^\de F\|_\M<\tf18 R_0, \pq X^{R_1}F \in W^{-m,1}(\R^d).}
Let $F=F^0+F^1+F^2$ with $F^0:=X^\de F$ and $F^2:=X_{R_1}F$. Then for all $n\in\N$ we have
\EQ{
 I_nF^0\in\M_+, \pq  \|I_nF^0\|_\M \le \|F^0\|_\M<\tf18R_0, 
 \pq\supp I_nF^2 \subset \R^d_{>R_1},}
which implies $I_n(F^0+F^2)-(F^0+F^2)\in\tf18(U+U+U+U)\subset\tf12U$. 
On the other hand, $F^1\in W^{-m,1}(\R^d)$ and $\supp F^1\subset\R^d_{\ge\de}$ 
imply that $\supp I_nF^1\subset\R^d_{\ge\de}$ and $I_nF^1\to F^1$ in $W^{-m,1}(\R^d)$, 
hence $I_nF^1\to F^1$ in $W^{-m,1}_\de\emb\X$ by \eqref{embs}. 
In particular $I_nF^1-F^1\in\tf12 U$ for large $n$, then $I_nF-F\in U$. 
Hence $I_nF\to F$ in $\X$. 
\end{proof}

\subsection{Bounded sets in $\X$}
Here we characterize the bounded subsets of $\X$. 
First recall that the inductive limit $X=\Cu_{n\in\N}X_n$ of a sequence of normed spaces $X_1\emb X_2\emb\cdots$ is said to be \textit{regular} if every bounded set $B\subset X$ is a bounded set of some $X_n$. It is not always true, but \cite{Flo} gives some sufficient conditions. 

For any $m\in\N_0$, we define the negative Sobolev spaces of measure-type by 
\EQ{
 \pt \W^m :=\{ F\in \sD'(\R^d) : \|F\|_{\W^m}<\I\}, \pq \W^\I := \Cu_{m\in\N_0} \W^m  \emb \sD'(\R^d),
 \pr \|F\|_{\W^m}:=\sup\{\re F(\fy) : \fy\in\sD(\R^d),\ \|\fy\|_m\le 1\}.}
Then $\W^m\emb \sD'(\R^d)$ is a Banach space, whose closed unit ball $\bar B_{\W^m}(0,1)$ is closed in $\sD'(\R^d)$ because of the definition by test functions, and $\sD'(\R^d)$ is reflexive. Hence \cite[Corollary 1]{Flo} implies that the inductive limit $\W^\I$ is regular.

\begin{prop} \label{bdd sets}
For any subset $B\subset\X$, its boundedness is equivalent to existence of $\mu\in C_0^\ndc$ and $\nu\in C^\ndc$ such that $\ro^{\mu(l)}_l(F) \le \nu(l)$ for all $F\in B$ and $l>0$. 
\end{prop}
\begin{rem}
If $\nu\in C^\ndc_0$ (namely $\nu(0)=0$) then $B$ is bounded in $\X^{\mu,\nu}$, but we may not choose $\nu\in C^\ndc_0$ in general, which is easily observed already in the subspace $L^1_{\ge 0}\emb \M_+\emb\X$: For any non-zero $f\in\sD_{>0}$, 
the concentrating sequence of mass $\{nf(n\x_1,\x^\perp)\}_{n\in\N}$ 
is bounded in $L^1_{\ge 0}$, but in none of $\X^{\mu,\nu}$. It is convergent in $\sD'_{\ge 0}$ to a measure on the boundary $\x_1=0$, which is not in $\X$. 
In short, $\X=\Cu_{\mu,\nu}\X^{\mu,\nu}$ fails to be regular because of the decay condition  \hyptag{X-i} at $\x_1=0$. 
\end{rem}
\begin{proof}
Let $B\subset\X$ be bounded. For every $l\in\N$, $X^l:\X\to \W^\I$ is continuous, and $\W^\I$ is the regular inductive limit, so there exist some $m_l,n_l\in\N$ such that
\EQ{ \label{bdd in Wm}
 X^l(B)\subset B_{\W^{m_l}}(0,n_l).} 

Next we show that $X^lB$ is bounded in $\M_+$ for some $l>0$. 
Otherwise, there exist sequences $F_n\in B$, $(0,1)\ni l_n\to 0$ and $\fy_n\in\sD_{(0,l_n)}$ such that $\|\fy_n\|_0\le 2^{-n}$ and $F_n(\fy_n)>2$. 
Then for each $m\in\N$, there exists $N_m\in C_+^\ndc$ satisfying 
$N_m(l_n)\|\fy_n\|_m \le 2^{-m-n}$ 
for all $n\in\N$, which is possible because of $l_n\to 0$. 
For any $n\in\N$, $G_0\in\bar B_{\M_+}(0,1)$, $G_1\in \X_{\ge 1}$ and $G_2\in\sB(N)$, we have 
\EQ{
 \pt |G_0(\fy_n)| \le 2^{-n}, \pq G_1(\fy_n)=0,
 \pq |G_2(\fy_n)| \le \sum_{m\in\N}N_m(l_n)\|\fy_n\|_m \le 2^{-n}.}
Hence $|G(\fy_n)|\le 2^{1-n}$ for any $G\in\sB(1,1,N)$, so $F_n\not\in 2^n\sB(1,1,N)$, contradicting the boundedness of $\{F_n\}_{n\in\N}\subset B$ in $\X$. 
Hence there exist $\de,C_*\in(0,\I)$ such that $X^\de B\subset B_{\M_+}(0,C_*)$. 
For $l>0$ and $\fy\in\sD_{<l}$, we have, using \eqref{bdd in Wm} and Lemma \ref{lem:mult},  
\EQ{
 |F(\fy)| \le |F(X^\de\fy)|+|F(X_\de\fy)| \pt\le C_*\|\fy\|_0^+ +  n_{2l}M^\chi_\de(m_k)\|\fy\|_{m_{2l}}^+.}
There exist $\mu\in C_\de^\ndc$ and $\nu\in C^\ndc$ satisfying $\nu(0)\ge C_*$ and 
\EQ{
 \mu(l)\ge m_{2k}, \pq \nu(l) \ge C_*+M^\chi_\de(m_k)n_{2k}} 
for all $k\in\N_{\ge 2}$ and $\max(\de,k-1)<l\le k$. Then we have $\ro^{\mu(l)}_l(F)\le \nu(l)$ for all $F\in B$.

Conversely, suppose that such $\mu\in C_\de^\ndc$ and $\nu\in C^\ndc$ exist. 
Let $\de>0$ such that $\mu(2\de)<1$. 
There is $\ti\nu\in C^\ndc_0$ satisfying $M^\chi_\de(\mu(l))\nu(l)\le\ti\nu(l)$ for $l\ge\de$. 
Then Lemma \ref{lem:mult} implies that $X_\de B$ is bounded in $\X^{\mu,\ti\nu}\emb\X$, and 
that $X^\de B$ is bounded in $\M_+$, so it is absorbed by any $B_{\M_+}(0,R_0)\subset\sB(R_0,R_1,N)\in\sB_\X$, meaning the boundedness in $\X$. Hence $B=X^\de B+X_\de B$ is also bounded in $\X$.   
\end{proof}

\subsection{Compact sets in $\X$}
Here we characterize compact subsets of $\X$. 
First we recall variant notions of compactness. 
A topological space $X$ is sequentially compact iff every sequence has a convergent subsequence. 
$X$ is countably compact iff every sequence has an accumulation point. 
Countable compactness is weaker both than compactness and than sequential compactness, but there is no order in general between the latter two. 

\begin{lem} \label{count-seq cmp}
Let $\ta_1,\ta_2$ be topologies on a set $X$. 
Let $\ta_1$ be countably compact and stronger than $\ta_2$. 
Let $\ta_2$ be Hausdorff and sequentially compact.  
Then so is $\ta_1$. 
\end{lem}
In particular, in any topological vector space embedded into $\sD'$, countably compact subsets are sequentially compact. 
\begin{proof}
The Hausdorff property transfers to the stronger topology, so it suffices to show the sequential compactness of $\ta_1$. 
For any sequence $S_0$, there is a subsequence $S_1\subset S_0$ convergent to some $x_\I$ in $\ta_2$. 
If $S_1$ is not convergent to $x_\I$ in $\ta_1$, then there is a $\ta_1$-neighborhood $U$ of $x_\I$ such that $S_2:=S_1\setminus U$ is a subsequence, which has a $\ta_1$-accumulation point $x_*$. 
Since $S_2\cap U=\empt$, we have $x_*\not=x_\I$. 
Then $S_2\to x_\I$ and accumulating at $x_*\not=x_\I$ in $\ta_2$, 
contradicting the Hausdorff. Hence $S_1\to x_\I$ in $\ta_1$. 
\end{proof}

Next we consider the compactness in $\W^\I$. 
\begin{lem} \label{lem:cmp Winf}
$K\subset \W^\I$ is countably compact iff it is compact in some $\W^m$. 
\end{lem}
\begin{proof}
Since $\W^m\emb \W^\I$ by definition, it suffices to show the necessity. 
For the radial mollifiers $\z_n=n^d\z(n\x)$, we have
\EQ{
 \pt |\z_n*F(\fy)| \le \|F\|_{\W^m}\|\z_n*\fy\|_m \le \|F\|_{\W^m} n^{m}\|\z\|_{W^{1,m}}\|\fy\|_0,
 \pr |(F-\z_n*F)(\fy)| \le \|F\|_{\W^m}\|\fy-\z_n*\fy\|_m \le \|F\|_{\W^m} n^{-1}\|\fy\|_{m+1},}
for all $m\in\N_0$, $\fy\in\sD$, $F\in \W^m$ and $n\in\N$. In particular, $\z_n*:\W^m\to \W^0$ is bounded for all $m\in\N_0$, so is $\z_n*:\W^\I\to \W^0$ by definition of the inductive limit.

Let $K\subset \W^\I$ be countably compact. Since $\W^\I\emb\sD'$ and bounded subsets of $\sD'$ are precompact, the above lemma implies that $K$ is sequentially compact. Also, the regularity of $\W^\I$ yields some $m\in\N$ and $R<\I$ such that $K\subset B:=\bar B_{\W^m}(0,R)$. 

To show that $K$ is compact in $\W^{m+1}$, take any sequence $\{F_n\}_{n\in\N}\subset K$. 
Since $K$ is sequentially compact in $\W^\I$, we may replace it by its subsequence convergent to some $F_\I$ in $\W^\I$. Let $F_n':=F_n-F_\I$. 
Since $\fy_k*:\W^\I\to \W^0$ is continuous for any $k\in\N$, we have $\|\z_k*F_n'\|_{\W^{m+1}}\to 0$ as $n\to\I$, while by the above estimate, 
\EQ{
 \|F_n'-\z_k*F_n'\|_{\W^{m+1}} \le \|F_n'\|_{\W^m}k^{-1} \le 2Rk^{-1}\to 0}
as $k\to\I$ uniformly for $n\in\N$. Thus we obtain $\|F_n'\|_{\W^{m+1}}\to 0$ as $n\to\I$, and so $K$ is compact in $\W^{m+1}$. 
\end{proof}

It turns out that every compact subset of $\X$ is bounded in a single $\X^{\mu,\nu}$, but the compactness does not hold for that norm, losing regularity. 
By interpolation, the loss may be chosen arbitrarily small, but to estimate convolutions, it is more convenient to use the derivatives as follows. For any $\mu,\nu,\io\in C_0^\ndc$, let
\EQ{
 \pt \X^{\mu,\nu,\io}_\p := \{F_0 + \sum_{1\le k\le d}\p_k F_k \in\sD'(\R^d) \mid F_0\in\X^{\mu,\nu},\ F_1\etc F_d\in\X^{\mu,\io}\},
 \prQ \|G\|_{\X^{\mu,\nu,\io}_\p} := \inf\{\ro^{\mu,\nu}(F_0)+\sum_{1\le k\le d}\ro^{\mu,\io}(F_k) \mid G=F_0 + \sum_{1\le k\le d}\p_k F_k\}. }
As shown below, the weight $\io$ is arbitrary in characterizing the compactness, 
but its freedom is important for boundedness of multipliers with singularity at the boundary. 

For any $G\in\X^{\mu,\nu,\io}_\p$, $l\in(0,\I)$ and $\fy\in\sD_{<l}$, we have 
\EQ{
 |G(\fy)| \pt\le |F_0(\fy)|+\sum_{1\le k\le d}|F_k(\p_k\fy)| 
 \pr\le \max(\nu(l),\io(l))\|\fy\|_{\mu(l)+1}[\ro^{\mu,\nu}(F_0)+\sum_{1\le k\le d}\ro^{\mu,\io}(F_k)],}
where $F_0\etc F_d$ are as in the above definition. Hence 
\EQ{ \label{mu-p to mu+1}
 \ro^{\mu+1,\max(\nu,\io)}(G) \le \|G\|_{\X^{\mu,\nu,\io}_\p}}
and in particular $\X^{\mu,\nu,\io}_\p\emb\sD'$. 
Since $\X^{\mu,\nu}$ is Banach and the representation $F=F_0+\sum_{k=1}^d \p_k F_k$ is preserved when each $F_k$ converges in $\sD'$, we deduce that $\X^{\mu,\nu,\io}_\p$ is also Banach. 
By definition, we also have
\EQ{
 \|F\|_{\X^{\mu,\nu,\io}_\p} \le \ro^{\mu,\nu}(F), \pq \max_{1\le k\le d}\|\p_k F\|_{\X^{\mu,\nu,\io}_\p} \le \ro^{\mu,\io}(F).} 

\begin{prop} \label{prop:cmp X}
For any subset $K\subset\X$, the following conditions are equivalent. 
\begin{enumerate}
\item $K$ is compact in $\X$.
\item $K$ is sequentially compact in $\X$. 
\item $K$ is countably compact in $\X$.
\item $K$ is bounded in $\X^{\mu,\nu}$ and compact in $\X^{\mu,\nu,\io}_\p$ for some $\mu,\nu\in C_0^\ndc$ and all $\io\in C_+^\ndc$. 
\item $K$ is bounded in $\X^{\mu,\nu}$ and compact in the $\ro^{\mu+1,\nu}$ norm for some $\mu,\nu\in C_0^\ndc$. 
\end{enumerate}
\end{prop}
\begin{proof}
(3) follows both from (1) and from (2) in general. Moreover, since $\X\emb\sD'$, by Lemma \ref{count-seq cmp},  (3) implies (2). 
(4) implies (5) because of \eqref{mu-p to mu+1}.

Next we show that (5) implies (1). 
Let $R<\I$ such that $K\subset B:=\bar B_{\X^{\mu,\nu}}(0,R)$. Then it suffices to show that $B\subset \X$ is continuous for the $\ro^{\mu+1,\nu}$ metric on $B$. 
Take any convergent sequence $B\ni F_n\to F_\I\in B$ and any $U:=\sB(R_0,R_1,N)\in\sB_\X$. 
Let $F_n':=F_n-F_\I$. There exists $\de>0$ such that $\mu(2\de)<1$ and $\nu(2\de)<\tf{R_0}{4R}$. 
Then for all $n\in\N$, $\ro^{\mu,\nu}(F_n')<2R$ implies $X^\de F_n'\in B_{\M_+}(0,R_0/2)\subset U/2$. 
There exist $\mu_1,\nu_1\in C_0^\ndc$ such that $\mu(l)+1\le\mu_1(l)$ and $M^\chi_\de(\mu(l)+1)\nu(l)\le\nu_1(l)$ for $l\ge\de$. 
Then Lemma \ref{lem:mult} implies $\ro^{\mu_1,\nu_1}(X_\de F_n')\le \ro^{\mu+1,\nu}(F_n')\to 0$, 
so $X_\de F_n'\to 0$ in $\X$. Thus $X_\de F_n'\in U/2$ and $F_n'\in U$ for large $n$. 
Hence $B\subset \X$ is continuous for $\ro^{\mu+1,\nu}$, so $K$ is compact also in $\X$. 

It remains to show that (2) implies (4). Let $K\subset\X$ be sequentially compact. 
Proposition \ref{bdd sets} yields $\mu\in C_\de^\ndc$ with some $\de>0$ and $\nu_0\in C^\ndc$ such that $\ro^{\mu(l)}_l(F)\le\nu_0(l)$ for all $F\in K$ and $l>0$. 

Let $S(l):=\sup_{F\in K}\|X^lF\|_\M$ for $l>0$  
and suppose that $S(l)\not\to 0$ as $l\to+0$. 
Then there exist $\e>0$ and sequences $F_n\in K$, $(0,\de)\ni l_n\to 0$, and $\fy_n\in\sD_{<l_n}$ such that $\|\fy_n\|_0\le 1$ and $F_n(\fy_n)>\e$ for all $n\in\N$. 
By the sequential compactness of $K$, we may assume that $F_n$ converges to some $F_\I$ in $\X$. 
There exists $\{N_m\}_{m\in\N}\subset C_+^\ndc$ satisfying 
$N_m(l_n)\|\fy_n\|_m \le 2^{-m-3}\e$ 
for all $m,n\in\N$. For any $n\in\N$, $G_0\in \bar B_{\M_+}(0,\e/4)$, $G_1\in \X_{\ge 1}$ and $G_2\in\sB(N)$, we have 
\EQ{
 \pt |G_0(\fy_n)| \le \e/4, \pq |G_1(\fy_n)|=0,
 \pq |G_2(\fy_n)| \le \sum_{m\in\N}N_m(l_n)\|\fy_n\|_m \le \e/4,}
so $|G(\fy_n)|\le \e/2$ for any $G\in\sB(\e/4,1,N)$ and $n\in\N$. 
On the other hand, $F_\I\in\X$ and $l_n\to+0$ imply $F_\I(\fy_n)\to 0$, hence $|(F_n-F_\I)(\fy_n)|>\e/2$ for large $n\in\N$, then $F_n-F_\I\not\in\sB(\e/4,1,N)$, contradicting $F_n\to F_\I$ in $\X$.  
Hence $S(l)\to 0$ as $l\to+0$, so that we may find $\nu_1\in C_0^\ndc$ such that $K\subset B_{\X^{\mu,\nu_1}}(0,1)$. 

Let $\nu\in C_0^\ndc$ such that $\nu_1\le\nu$, $\nu_1(l)M^\chi_l(\mu(l))/\nu(l)\to 0$ as $l\to\I$,  and $\nu_1(l)/\nu(l)\to 0$ as $l\to +0$. 
Take any sequence $F_n\in K$ convergent to some $F_\I\in K$ in $\X$. 
It suffices to prove $F_n \to F_\I$ also in $\X^{\mu,\nu,\io}_\p$ for all $\io\in C_+^\ndc$. 
Since we have, using Lemma \ref{lem:mult}, 
\EQ{
 \pt \ro^{\mu,\nu}(X_HF)+\ro^{\mu,\nu}(X^LF) \le \ro^{\mu,\nu_1}(F)\BR{\sup_{l\ge H}\tf{\nu_1(l)M^\chi_l(\mu(l))}{\nu(l)}+\sup_{0<l\le 2L}\tf{\nu_1(l)}{\nu(l)}} \to 0
}
as $H\to\I$ and $\de\ge L\to+0$ uniformly for all $F\in K$, it suffices to show that 
$F_n':=X_LX^H(F_n-F_\I) \to 0$ in $\X^{\mu,\nu,\io}_\p$ for all $0<L<\de<1<H<\I$.

Since $X^H:\X\to \W^\I$ is continuous, $X^H(K)$ is also countably compact in $\W^\I$. 
Then by Lemma \ref{lem:cmp Winf}, there exists $m\in\N$ such that $X^H(K)$ is compact in $\W^m$. Since $F_n \to F_\I$ in $\sD'$, we deduce that 
$\|F_n'\|_{\W^m}\to 0$, while $\{F_n'\}_{n\in\N}$ is bounded in $\X^{\mu,1}$. 
Let $\{\fy_n\}_{n\in\N}$ be the mollifier in Proposition \ref{prop:dense X} and $\fy_n^\da(\x):=\fy_n(-\x)$. Then
\EQ{
 |\fy_k*F_n'(\fy)| \pt\le \|F_n'\|_{\W^m}\|\fy_k^\da*X_{L/2}\fy\|_m 
  \pr\le \|F_n'\|_{\W^m} r_k^{-m}\|\z\|_{W^{1,m}}M^\chi_{L/2}(\mu(L))\|\fy\|_{\mu(L)}^+,}
for any $k\in\N$ and $\fy\in\sD$, while $\fy_k*F_n'(\fy)=0$ if $\fy\in\sD_{<l}$ for some $l\le L$. 
Hence if we choose $k=k(n)\to\I$ sufficiently slowly compared with $\|F_n'\|_{\W^m}\to 0$ as $n\to\I$, then $\ro^{\mu,\nu}(\fy_k*F_n')\to 0$. 
On the other hand, we have 
\EQ{
 \pt F_n'-\fy_k*F_n' = \na\cdot F'_{k,n}, \pq F'_{k,n}:=\int_0^1\int_{|y-c_k|<r_k}y\fy_k(y)F_n'(x-\te y)dyd\te,}
and for any $l\in(0,\I)$, using that $\supp F_n'\subset\R^d_{>L}$, 
\EQ{ 
 \ro^{\mu(l)}_l(F'_{k,n}) \lec (|c_k|+r_k)\ro^{\mu(l)}_l(F_n') \lec \tf{|c_k|+r_k}{\io(L)} \io(l) \ro^{\mu,1}(F_n').}
Hence $\ro^{\mu,\io}(F'_{k,n})\to 0$ as $k(n)\to\I$ and $n\to\I$. 
Thus we obtain $\|F_n'\|_{X^{\mu,\nu,\io}_\p}\to 0$, hence $F_n\to F_\I$ in $\X^{\mu,\nu,\io}_\p$, which concludes the proof.
\end{proof}

\begin{rem}
Applying the above result to any sequence $\{F_n\}_{n\in\N\cup\{\I\}}\subset\X$, we obtain  characterization of convergence: $F_n\to F_\I$ in $\X$ iff
\EQ{
 \sup_{n\in\N} \X^{\mu,\nu}(F_n)<\I,
 \pq \lim_{n\to\I}\|F_n-F_\I\|_{\X^{\mu,\nu,\io}_\p}= 0,}
for some $\mu,\nu\in C_0^\ndc$ and all $\io\in C_+^\ndc$. 
The latter suffices to hold for one $\io\in C_+^\ndc$, and also the norm may be replaced with $\ro^{\mu+1,\nu}$. 

It is natural to look for simpler characterizations. On one hand, the convergence in any $\X^{\mu,\nu}$ is too strong. 
A simple counter-example is $F_n(\x):=\x_1e^{in\x_1}$. It is bounded in $\X^{0,\nu}$ with $\nu(l):=l$,  and $F_n=in^{-1}[(e^{in\x_1}-1)-\p_1 \x_1(e^{in\x_1}-1)]$ implies $\ro^{1,\nu}(F_n)\to 0$, hence $F_n\to 0$ in $\X$, whereas $\|F_n\|_{\M(\R^d_{<l})}=\|F_1\|_{\M(\R^d_{<l})}>0$ for all $l>0$ implies $\inf_n\|F_n\|_{\X^{\mu,\nu}} >0$ for any choice of $\mu,\nu\in C_0^\ndc$. 
On the other hand, the bounded weak-* convergence is too weak: $\sup_n\|F_n\|_{\X^{\mu,\nu}}<\I$ and $F_n\to F_\I$ in $\sD'(\R^d)$ do not imply $F_n\to F_\I$ in $\X$, if $d\ge 2$. A simple counter-example is $F_n(\x):=\fy(\x-(0,nv))$ for any non-zero $\fy\in\sD_{>0}$ and $v\in\R^{d-1}$. 
Another candidate is the weak convergence in a single $\X^{\mu,\nu}$, but it does not seem convenient to handle. 
\end{rem}

The multiplication factor $\ka$ is essential for the following multiplier estimate. Let 
\EQ{ \label{def Y}
 Y:= C^\I(\R^d_{>0};\C)\cap L^\I(\R^d_{(0,1)};\C) \cap \Ca_{m\ge 0} \Ca_{n\ge 2}W^{m,\I}(\R^d_{(1/n,n)};\C).}
It consists of smooth functions on $\R^d_{>0}$ that are bounded as $|\x^\perp|\to\I$, including all the derivatives, locally uniformly for $0<\x_1<\I$, and also bounded (but not derivatives) uniformly as $\x_1\to+0$. $Y$ is a Fr\'echet space with the topology defined by those seminorms of  $L^\I(\R^d_{(0,1)})$ and $W^{m,\I}(\R^d_{(1/n,n)})$. Thanks to \eqref{E X}, for any $F\in\X$ and $f\in Y$ we may define the product $fF\in\sD'(\R^d)$ by 
$(fF)(\fy):=\sE \sR F(f\fy)$ for all $\fy\in\sD(\R^d)$.  
It is indeed bounded on $\X$ as follows. 

\begin{lem} \label{lem:mult X}
Let $n_1,n_2\in\N$ and $A:\R^d_{>0}\to M(n_1,n_2)$ with all components in $Y$.  
Then for any $\de>0$ and $\mu\in C^\ndc_\de$, there exist $\be,\ga \in C^\ndc$ such that 
\begin{align} \label{bd psi-b}
 &\|A^\tran \fy\|_{\mu(l),\de} \le \be(l)\|\fy\|_{\mu(l)}^+, \\ 
  \label{bd psi-ka-b}
 & \sum_{k\in\N} \ga(2^{2-k}\de)\|(\na A^\tran)\fy\|_{L^\I(1<2^k\x_1/\de<4)} 
 + \ga(l)\|X_\de(\na A^\tran)\fy\|_{\mu(l)} \le \be(l)\|\fy\|_{\mu(l)}^+,
\end{align}
for all $l\in(0,\I)$ and $\fy\in(\sD_{<l})^{n_1}$. 
$(\be,\ga)$ may be chosen uniformly for a set of $A$ whose components are bounded in $Y$. 

Moreover, for any $\be\in C^\ndc$ satisfying \eqref{bd psi-b} and any $\nu\in C^\ndc$, the multiplication with $A$ is bounded $\X^{\mu,\nu}\to\X^{\mu,\be\nu}$. If in addition $\ga\in C^\ndc$ satisfies \eqref{bd psi-ka-b} then it is also bounded $\X^{\mu,\nu,\ga\nu}_\p\to\X^{\mu,\be\nu,\be\ga\nu}_\p$. The operator norm is bounded respectively as
\EQ{
 \|AF\|_{\X^{\mu,\be\nu}} \le \|F\|_{\X^{\mu,\nu}}, \pq \|AF\|_{\X^{\mu,\be\nu,\be\ga\nu}_\p} \le 2\|F\|_{\X^{\mu,\nu,\ga\nu}_\p}.}
\end{lem}
\begin{proof}
Using the Leibniz rule and Lemma \ref{lem:mult}, we have 
\EQ{
 \|A^\tran \fy\|_{\mu(l),\de} \pt\le \|A^\tran\|_{L^\I(\R^d_{(0,l)};M(n_2,n_1))}\|\fy\|_0^+
 \prq+ 2^{\mu(l)}\sup_{|\al|\le\mu(l)}\|\p^\al A^\tran\|_{L^\I(\R^d_{(\de,l)};M(n_2,n_1))}M^\chi_\de(\mu(l))\|\fy\|_{\mu(l)}^+, }
where the factors multiplied with the seminorms of $\fy$ are non-increasing and locally bounded in $l\in(0,\I)$, so that we may find $\be_0\in C^\ndc$ dominating the sum of them. Then we have
\EQ{
 \|A^\tran \fy\|_{\mu(l),\de} \le \be_0(l)\|\fy\|_{\mu(l)}^+.}
Similarly, there exist $\{a_k\}_{k\in\N}\subset[1,\I)$ and $\be_1\in C^\ndc$ such that 
\EQ{
 \|(\na A^\tran)\fy\|_{L^\I(1<2^k\x_1/\de<4)} \le a_k\|\fy\|_0^+,
 \pq \|X_\de(\na A^\tran)\fy\|_{\mu(l)} \le \be_1(l)\|\fy\|_{\mu(l)}^+.}
Then we may find $\ga\in C^\ndc$ satisfying 
$\ga(2^{2-k}\de)a_k \le 2^{-1-k}$ for all $k\in\N$ and $\ga(l)=\ga(2\de)\le 1/4$ for all $l\ge 2\de$. 
Choosing $\be:=\max(1,\be_0,\be_1)$ satisfies both \eqref{bd psi-b} and \eqref{bd psi-ka-b}. 

Next we prove the boundedness assuming \eqref{bd psi-b}. Combining it with \eqref{bd on Ffy-de} yields
\EQ{
 |AF(\fy)|=|F(A^\tran\fy)| \le \nu(l)\ro^{\mu,\nu}_l(F)\|A^\tran\fy\|_{\mu(l),\de} \le \nu(l)\ro^{\mu,\nu}_l(F)\be(l)\|\fy\|_{\mu(l)}^+,}
for all $F\in\X^{\mu,\nu}$, $l>0$ and $\fy\in\sD_{<l}$. 
Hence $\ro^{\mu,\be\nu}_l(AF)\le\ro^{\mu,\nu}_l(F)$. 

Suppose further \eqref{bd psi-ka-b}. For any $F\in\X^{\mu,\nu,\ga\nu}_\p$, let $F=F_0+\sum_k \p_kF_k$. 
For the first two terms in 
$A F = A F_0 + \sum_k \p_k(A F_k) -\sum_k (\p_k A)F_k$, 
we have, using the above estimate, 
\EQ{
 \ro^{\mu,\be\nu}(A F_0) \le \ro^{\mu,\nu}(F_0),
 \pq \ro^{\mu,\be\ga\nu}(A F_k) \le \ro^{\mu,\ga\nu}(F_k). }
For the remaining term, let $X_\de^{(k)}:=X_{2^{-k}\de}-X_{2^{1-k}\de}$ for $k\in\N$ and $X_\de^{(0)}:=X_\de$. 
Using \eqref{bd psi-ka-b}, we have for any $l>0$ and $\fy\in\sD_{<l}$, 
\EQ{
 \pt |(\p_k A) F_k(\fy)|  = |\sE \p_k A \sR F_k(\fy)| \le \sum_{k\in\N_0}|F_k(X_\de^{(k)}\p_kA^\tran \fy)| 
 \pr\le \ro^{\mu,\ga\nu}(F_k)\nu(l)\BR{\sum_{k\in\N} \ga(2^{2-k}\de)\|X_\de^{(k)}\p_kA^\tran \fy\|_0+\ga(l)\|X_\de\p_kA^\tran \fy\|_{\mu(l)}^+},
 \pr\le \ro^{\mu,\ga\nu}(F_k)\nu(l)\be(l)\|\fy\|_{\mu(l)}^+.}
Thus we obtain $\ro^{\mu,\be\nu}((\p_kA) F_k) \le \ro^{\mu,\ga\nu}(F_k)$.  
Summing the above three estimates yields 
\EQ{
 \|AF\|_{\X^{\mu,\be\nu,\be\ga\nu}_\p}
 \le \ro^{\mu,\nu}(F_0)+2\sum_{1\le k\le d}\ro^{\mu,\ga\nu}(F_k),}
and then the infimum over $F_0\etc F_d$ leads to the desired operator bound. 
\end{proof}

\section{Convolution for the distributions} \label{s:conv}
The purpose of the following analysis is to quantify the propagation of positive frequency by convolution as well as its decay near the zero frequency, in the distribution sense. 
For $F,G\in \X$, the convolution is formally defined by
\EQ{ \label{def conv}
 \sD(\R^d)\ni\forall\fy, \pq (F*G)(\fy) := F_\x(G_\y(\fy(\x+\y))),}
where the subscript $\y$ of $G_\y$ indicates the independent variable of the test function.
In general, for any $G\in\sD'(\R^d)$ and $\fy\in\sD(\R^d)$, we have $G_\y(\fy(\x+\y))\in C^\I(\R^d)$ and
\EQ{
 \p^\al_\x G_\y(\fy(\x+\y))=G_\y(\p^\al\fy(\x+\y)),}
but we need some condition for $F\in\sD'$ to act on it. 
\begin{lem}
Let $l>0$, $m\ge 0$, $G\in\sD'(\R^d)$, $\ro_l^m(G)<\I$ and $\fy\in\sD_{<l}$. 
Then we have $G_\y(\fy(\x+\y))\in \B^\I_{+<l}$ and 
\EQ{
 \forall n\in\N, \pq \|G_\y(\fy(\x+\y))\|_n^+ \le \ro_l^m(G)\|\fy\|_{m+n}^+. }
\end{lem}
\begin{proof}
Since $\supp\fy$ is compact, we have 
$l':=\max\{\x_1\mid \x\in\supp\fy\}<l$. 
Then for any $\x \in\R^d_{>l'}$, we have $\supp_\y\fy(\x+\y)\subset\R^d_{<0}$, so $\supp G\subset\R^d_{\ge 0}$ implies $f(\x):=G_\y(\fy(\x+\y))=0$. 
Hence $\supp f\subset\R^d_{\le l'}\subset\R^d_{<l}$. Moreover, $f\in C^\I$ and for any $\x\in\R^d_{>0}$ and $\al\in\N_0^d$ with $|\al|\le n$ we have 
\EQ{
 |\p^\al_\x f(\x)| = |G_\y(\p^\al\fy(\x+\y))| \le \ro_l^m(G)\|\p^\al\fy\|_m^+ \le  \ro_l^m(G)\|\fy\|_{m+n}^+.}
The supremum for $\x$ and $\al$ yields the desired estimate. 
\end{proof}
The above lemma makes sense of \eqref{def conv} for $F,G\in\X$, where $F$ is the extension by Lemma \ref{lem:ext F} to $\B_{+<l}^\I$ if $\fy\in\sD_{<l}$. 
This definition is compatible with the convolution for sign-definite locally integrable $F,G$. 
Moreover, we have the following estimate on $F*G$.

\begin{lem}
Let $\de\in(0,1]$, $\mu\in C^\ndc_\de$ and $\nu\in C^\ndc$. 
Then for any $F,G\in\X^{\mu,\nu}$, $l\in(0,\I)$, $\e\in(0,\de/2]$ and $\fy\in\sD_{<l}$, we have
\EQ{
 \tf{|F*G(\fy)|}{\ro^{\mu,\nu}_l(F)\ro^{\mu,\nu}_l(G)}
 \le 3[\nu(2\e)\nu(l) \pt+1_{l>2\e}\nu(\min(2\de,l/2))\nu(l-\e)] \|\fy\|_{\mu(l)}^+ 
 \pr + 1_{l>4\de}[M_\de^\chi(\mu(l-\de))\nu(l-\de)]^2 \|\fy\|_{2\mu(l-\de)}^+.}
\end{lem}
\begin{proof}
First we consider the case of $l\le 4\de$. For any $\fy\in\sD_{<l}$, there is $l'\in(0,l)$ such that $\supp\fy\subset\sD_{<l'}$. 
Let $\ti\chi\in C^\I(\R;[0,1])$ be a cut-off function satisfying 
\EQ{
 \supp\ti\chi\subset(-\I,l/2), \pq \ti\chi((-\I,l'/2])=\{1\},}
$\ti\chi^\e:=\chi^\e\ti\chi$, and $\ti\chi_\e:=\chi_\e\ti\chi$. Then we decompose
\EQ{ 
  (F*G)(\fy) \pt= F_\x G_\y (\ti\chi^\e(\x_1)\fy(\x+\y)) + G_\y F_\x (\ti\chi^\e(\y_1)\fy(\x+\y))
 \prq - F_\x G_\y (\ti\chi^\e(\x_1)\ti\chi^\e(\y_1)\fy(\x+\y))
 \prq + F_\x G_\y (\ti\chi_\e(\x_1)\fy(\x+\y)) + G_\y F_\x (\ti\chi_\e(\y_1)\fy(\x+\y))
 \prq - F_\x G_\y (\ti\chi_\e(\x_1)\ti\chi_\e(\y_1)\fy(\x+\y)), }
using that $(1-\ti\chi(\x_1))(1-\ti\chi(\y_1))\fy(\x+\y)=0$ by the support condition. 
Since $2\e\le\de$ and $\mu\in C^\ndc_\de$ imply $\mu(2\e)<1$, the first term is bounded by
\EQ{
 |F_\x G_\y (\ti\chi^\e(\x_1)\fy(\x+\y))|
  \pt\le \ro_l^{\mu,\nu}(F)\nu(2\e)\|G_\y (\fy(\x+\y))\|_0^+
 \pr\le \ro_l^{\mu,\nu}(F)\nu(2\e)\ro_l^{\mu,\nu}(G) \nu(l) \|\fy\|_{\mu(l)}^+.}
By symmetry, the second term is bounded exactly in the same way, 
while the third term is easier. 
The fourth term is estimated similarly by
\EQ{
 |F_\x G_\y (\ti\chi_\e(\x_1)\fy(\x+\y))|
 \pt\le \ro_l^{\mu,\nu}(F)\nu(l/2)\sup_{\x_1>\e}|G_\y(\fy(\x+\y))|
 \pr\le \ro_l^{\mu,\nu}(F)\nu(l/2)\ro_l^{\mu,\nu}(G)\nu(l-\e)\|\fy\|_{\mu(l)}^+.}
The fifth term is the same by symmetry, while the final term is bounded by 
\EQ{
 |F_\x G_\y (\ti\chi_\e(\x_1)\ti\chi_\e(\y_1)\fy(\x+\y))|
 \pt\le \ro_l^{\mu,\nu}(F)\nu(l/2)\sup_{\x_1>\e}|G_\y(\ti\chi_\e(\y_1)\fy(\x+y))|
 \pr\le \ro_l^{\mu,\nu}(F)\nu(l/2)\ro_l^{\mu,\nu}(G)\nu(l-\e)\|\fy\|_0^+,}
since $l/2\le2\de$ implies $\mu(l/2)<1$. Thus we obtain 
\EQ{
 \tf{|F*G(\fy)|}{\ro_l^{\mu,\nu}(F)\ro_l^{\mu,\nu}(G)}
 \pt\le 3[\nu(2\e)\nu(l)+1_{l>2\e}\nu(l/2)\nu(l-\e)] \|\fy\|_{\mu(l)}^+,}
where the restriction by $1_{l>2\e}$ comes from $\ti\chi_\e=0$ for $l\le 2\e$. 

For the other case $4\de<l$, we replace the above $\ti\chi_\e$ with $\chi_\e^\de:=\chi_\e\chi^\de$ and decompose 
\EQ{ 
  (F*G)(\fy) \pt= F_\x G_\y (\chi^\e(\x_1)\fy(\x+\y)) + G_\y F_\x (\chi^\e(\y_1)\fy(\x+\y))
 \prq - F_\x G_\y (\chi^\e(\x_1)\chi^\e(\y_1)\fy(\x+\y))
 \prq + F_\x G_\y (\chi_\e^\de(\x_1)\fy(\x+\y)) + G_\y F_\x (\chi_\e^\de(\y_1)\fy(\x+\y))
 \prq - F_\x G_\y (\chi_\e^\de(\x_1)\chi_\e^\de(\y_1)\fy(\x+\y))
 \pn + F_\x G_\y (\chi_\de(\x_1)\chi_\de(\y_1)\fy(\x+\y)). }
The first six terms are estimated in the same way as in the previous case, using $\supp\chi_\e^\de\subset\R^d_{(\e,2\de)}$.  
For the last remaining term, we have
\EQ{
 l':=l-\de \implies \supp \chi_\de(\x_1)\chi_\de(\y_1)\fy(\x+\y) \subset \{\de<\x_1,\y_1<l'\},}
and so
\EQ{ \label{conv hh}
 \pt |F_\x G_\y (\chi_\de(\x_1)\chi_\de(\y_1)\fy(\x+\y))|
 \pn\le \ro_l^{\mu,\nu}(F) \nu(l')\|\chi_\de(\x_1)G_\y(\chi_\de(\y_1)\fy(\x+\y))\|_{\mu(l') (\x)}^+.}
The last norm is bounded by 
\EQ{ \label{exp FG1}
 \pt\max_{|\al|\le\mu(l')}\sup_{\de<\x_1<l'}\sum_{0\le k\le\al_1}\tf{\al_1!}{(\al_1-k)!k!}|\chi_\de^{(k)}(\x_1)G_\y(\chi_\de(\y_1)\p^{\al-ke_1}\fy(\x+\y))| 
 \pr\le M_\de^\chi(\mu(l')) \ro_l^{\mu,\nu}(G) \nu(l')\sup_{\x_1>\de}\max_{|\al|\le\mu(l')}\|\chi_\de(\y_1)\p^\al\fy(\x+\y)\|_{\mu(l')(\y)}^+,}
where the last norm of \eqref{exp FG1} is bounded similarly by 
$M_\de^\chi(\mu(l'))\|\fy\|_{2\mu(l')}^+$, 
and thus we obtain the desired estimate.
\end{proof}

As an immediate consequence of the above lemma, we have 
\begin{lem} \label{lem:conv}
Let $\de\in(0,1]$, $\mu\in C^\ndc_\de$, $\nu,\ti\mu,\ti\nu\in C^\ndc$. Suppose 
\begin{enumerate}
\item $\mu\le\ti\mu$ and $2\mu(l-\de)\le\ti\mu(l)$ for $l>4\de$. 
\item For each $l>0$ there is $\e\in(0,\de/2]$ such that
$9\nu(2\e)\nu(l) \le \ti\nu(l)$, and, if $2\e<l$ then $9\nu(l-\e)^2 \le \ti\nu(l)$. 
\item For $l>4\de$, 
$3[M_\de^\chi(\mu(l-\de))\nu(l-\de)]^2 \le \ti\nu(l)$. 
\end{enumerate}
Then we have for any $F,G\in\X^{\mu,\nu}$ and $l\in(0,\I)$
\EQ{
 \ro_l^{\ti\mu,\ti\nu}(F*G) \le \ro_l^{\mu,\nu}(F)\ro_l^{\mu,\nu}(G).}
\end{lem}
Thus we have reduced the convolution estimate to the weight conditions. 
Note that the condition on $\mu$ is independent of $\nu,\ti\nu$. 
The point is that the left side for $\ti\nu$ has either a small factor or a smaller argument $l-\e$ or $l-\de$. 
Hence, if $\mu,\nu$ are accelerating sufficiently fast, then $\X^{\mu,\nu}$ becomes a convolution algebra, and the convolution may even gain arbitrary small factor. 
More precisely, we have the following.
\begin{lem} \label{lem:gen weight}
For each $\de\in(0,1]$, we have two mappings 
\EQ{
 \pt \fC_\de^1:C^\ndc_\de \to C_\de^\ndc,
 \pq \fC_\de^2:C^\ndc_\de \times C^\ndc_0 \times C^\nin_{(0,1]} \to C_0^\ndc}
satisfying the following. Define $\fC_\de:C^\ndc_\de \times C^\ndc_0 \times C^\nin_{(0,1]} \to C^\ndc_\de \times C_0^\ndc$ by
\EQ{
 \fC_\de(\mu_0,\nu_0,\ka):=(\fC_\de^1(\mu_0),\fC_\de^2(\fC_\de^1(\mu_0),\nu_0,\ka)).}
Let 
$\mu_1:=\fC_\de^1(\mu_0)$ and $\nu_1:=\fC_\de^2(\mu_0,\nu_0,\ka)$ for any $\mu_0,\nu_0,\ka$ in the domain. Then
\begin{enumerate}
\item $\mu_0\le \mu_1$, $\nu_0\le\ka\nu_1\in C^\ndc_0$.
\item $\ti\mu:=\mu:=\mu_1$ satisfies the condition (1) of Lemma \ref{lem:conv}.
\item $\mu:=\mu_0$, $\nu:=\nu_1$ and $\ti\nu:=\ka\nu_1$ satisfy the conditions (2)-(3) of Lemma \ref{lem:conv}. In particular, for $(\mu_1,\nu_1):=\fC_\de(\mu_0,\nu_0,\ka)$, $F,G\in\X^{\mu_1,\nu_1}$ and $l\in(0,\I)$ we have
\EQ{ \label{conv est}
 \ro_l^{\mu_1,\ka\nu_1}(F*G) \le \ro_l^{\mu_1,\nu_1}(F) \ro_l^{\mu_1,\nu_1}(G).}
\item $\fC_\de^1,\fC_\de^2$ are non-decreasing for the pointwise order. 
\item $\fC_\de^1,\fC_\de^2$ are continuous in the topology of locally uniform convergence. 
\item Let $b\in C^\ndc_{[1,\I)}$ satisfy $b(l)=b(0)$ as long as $\nu_0(l)\le\ka$. Then $\fC_\de^2(\mu_0,b\nu_0,\ka)\ge b\fC_\de^2(\mu_0,\nu_0,\ka)$. 
\end{enumerate}
\end{lem}
As a special case, for any $\mu_0\in C^\ndc_\de$ and $\nu_0\in C^\ndc_0$, we obtain the Banach algebra for convolution: $\X^{\mu_1,\nu_1}\supset \X^{\mu_0,\nu_0}$ by $(\mu_1,\nu_1):=\fC_\de(\mu_0,\nu_0,1)$. 
Hence their union $\X$ is also a convolution algebra. 
In other words, the pointwise multiplication is continuous on $\F^{-1}\X$ (after continuous extension). 
A crucial point of the lemma for solving PDEs is that we may include arbitrary small factors $\ka$. 
\begin{proof}
For any $\mu_0\in C^\ndc_\de$, $\nu_0\in C^\ndc_0$ and $\ka\in C^\nin_{(0,1]}$, 
we will define $\mu_1=\fC_\de^1(\mu_0)$ and $\nu_1=\fC_\de^2(\mu_0,\nu_0,\ka)$. 
Henceforth, monotonicity and continuity for (vector) functions are for the pointwise ordering and for the local uniform convergence, respectively. 

First we define $\mu_1$ by
\EQ{
 \mu_1(l) := 2^{\max(l/\de-3,0)}\mu_0(l).}
Then $\mu_0\le\mu_1\in C_\de^\ndc$ and for $l\ge 4\de$
\EQ{
 2\mu_1(l-\de) = 2^{l/\de-3}\mu_0(l-\de) \le 2^{l/\de-3}\mu_0(l) = \mu_1(l).}
Moreover, $\mu_0\mapsto\mu_1$ is non-decreasing and continuous. 
Thus we have obtained $\fC_\de^1$. 

Next we define $\nu_1$. Let $\hat\nu:=\ka^{-1}\nu_0\in C^\ndc_0$. 
The mapping $(\nu_0,-\ka)\mapsto\hat\nu$ is non-decreasing and continuous. 
For each $l>0$, let
\EQ{ \label{def hate}
 \hat\e(l):=\sup\{\e>0 \mid 9\hat\nu(2\e) \le \ka(l)(1-2\e/\de)\}. }
Note that the set is non-empty because $\hat\nu(0)=0<\ka(l)$. 
The last factor $2\e/\de$ ensures $\hat\e(l)\le\de/2$, as well as the uniform decay in $\e$ of the right hand side.    
Hence $\hat\e:[0,\I)\mapsto(0,\de/2]$  is non-increasing and continuous, so is the mapping $(\nu_0,-\ka)\mapsto\hat\e$. 
Let 
\EQ{
 l_0:=\sup\{l>0 \mid l \le 2\hat\e(l)\}.}
Then $0<l_0\le\de$, $2\hat\e(l_0)=l_0$, $t\le 2\hat\e(t)$ for $0<t\le l_0$, and $0<2\hat\e(t)\le l_0$ for $t>l_0$, and
the mapping $\hat\e\mapsto l_0$ is non-decreasing and continuous. 
Let  
\EQ{
 \ga(l) := l-\hat\e(l).}
Then $\ga:[l_0,\I)\to[l_0/2,\I)$ is an increasing homeomorphism, 
so is the inverse $\la:=\ga^{-1}:[l_0/2,\I)\to[l_0,\I)$. 
The mapping $\hat\e\mapsto \la$ is non-decreasing and continuous.

To satisfy the condition (2) of Lemma \ref{lem:conv} for $\nu=\nu_1$ and $\ti\nu=\ka\nu_1$, 
we choose 
\EQ{
 \e:=\min(\tf{l_0}2,\hat\e(l)).}
Define $\nu_1$ first on $[0,l_0]$ by $\nu_1(t):=\hat\nu(t)$. 
Since $2\e\le l_0$ and $\e\le\hat\e(l)$, we have
\EQ{
 9\nu(2\e)\nu(l) = 9\hat\nu(2\e)\nu_1(l) 
\le 9\hat\nu(2\hat\e(l))\nu_1(l) \le \ka(l)\nu_1(l),}
by the monotonicity of $\hat\nu$ and the definition of $\hat\e$. 
Then it remains to satisfy
\EQ{ \label{cond conv}
 \pt l>l_0 \implies 9\nu_1(l-\hat\e(l))^2 \le \ka(l)\nu_1(l),
 \pr l>4\de \implies 3 [M_\de^\chi(\mu_0(l-\de))\nu_1(l-\de)]^2 \le \ka(l)\nu_1(l),}
as well as monotonicity of $\ka\nu_1$, to apply Lemma \ref{lem:conv}. 

We define $\nu_1$ to satisfy them iteratively starting from $\nu_1=\hat\nu$ on $[0,l_0]$ which is already fixed above. Define the sequence $l_0<l_1<\cdots$ by 
\EQ{
 l_n := \la(l_{n-1}) \iff l_{n-1}=\ga(l_n)=l_n-\hat\e(l_n).}
Since $\hat\e$ is positive and continuous, this sequence diverges $l_n\nearrow\I$ as $n\to\I$ (otherwise $\hat\e(l_n)\to 0$ would contradict). 
Then we define $\nu_1$ on $(l_{n-1},l_n]$ inductively by 
\EQ{ \label{def nu1}
 \pt \nu_1(l):=\sM(\hat\nu(l),\max(\nu_L(l),\nu_H(l))), \pq \sM(\al,\be):=\max(\al,\be,\al\be),
 \prq \nu_L(l):=\tf{9\nu_1(\ga(l))^2}{\ka(l)},
 \pq \nu_H(l):=1_{l>3\de}\tf{3[M_\de^\chi(\mu_0(l-\de))\nu_1(l-\de)]^2}{\ka(l)}.}
Since $l_n-l_{n-1}=\hat\e(l_n)<\de$, the right hand side of \eqref{def nu1} is well-defined if $\nu_1$ is already defined on $[0,l_{n-1}]$. 
Thus $\nu_1:[0,\I)\to[0,\I)$ is well-defined. 
Since $\sM(\al,\be)\ge\max(\al,\be)$, we have $\nu_1\ge\max(\hat\nu,\nu_L,\nu_H)$, 
which implies \eqref{cond conv} and $\ka\nu_1\ge\nu_0$. 
In fact, we could define $\nu_1=\max(\hat\nu,\nu_L,\nu_H)$ to satisfy the properties (1)--(5), 
but the factor $\al\be$ of $\sM$ is needed for the superlinear property (6). 
The monotonicity of $\nu_1,\ka\nu_1$ is also immediate by induction. 

Similarly for the continuity, it suffices to check it at $l=l_0$ and $3\de$. 
Around $l=l_0<3\de$, we have $\nu_H=0$ and $\nu_L<\hat\nu$, because $\nu_1=\hat\nu$ for $l\le l_0$ and 
\EQ{
 9\hat\nu(\ga(l_0))^2 \le 9\hat\nu(l_0)^2 \le \ka(l_0)(1-l_0/\de)\hat\nu(l_0).}
Hence $\nu_1=\hat\nu$ around $l=l_0$, which is continuous.
Around $l=3\de$, we have $\nu_H<\nu_L$, because 
\EQ{
 3[M_\de^\chi(\mu_0(2\de))\nu_1(2\de)]^2 < 9\nu_1(2\de)^2 < 9\nu_1(\ga(3\de))^2,}
where we used $\mu_0(2\de)\le1/2$, $M_\de^\chi(1/2)^2<3$ 
and $\hat\e(3\de)<\de$. Hence the definition of $\nu_1$ is the same as the case of $l<3\de$ around $l=3\de$, which is continuous. 
Thus we conclude $\nu_1,\ka\nu_1\in C^\ndc_0$. 
The monotonicity and continuity of the mapping $(\mu_0,\hat\nu,-\ka)\mapsto\nu_1$ follows also by induction. 

Thus we have defined the mapping $\fC_\de$ satisfying the desired properties (1)--(5). 

It remains to check the final (6), namely the superlinearity in $\nu$ for multiplication.
We follow the above procedure for the new input $(\mu_0,b\nu_0,\ka)$, distinguishing the modified output with subscript $\bullet$. 
We have $\hat\nu_\bullet=b\hat\nu \ge\hat\nu$. 
Then $\hat\e_\bullet\le\hat\e$, $l_{0\bullet}\le l_0$, $\ga_\bullet\ge\ga$ and $\la_\bullet\le\la$, by their mapping monotonicity. 
Hence by induction $l_{n\bullet}\le l_n$ for all $n\ge 0$. 

Let $l_*:=\sup\hat\nu^{-1}([0,1])$. 
Then by the assumption, $b\ge 1$ is a constant on $l<l_*$, where $\nu_{1\bullet}\ge b\nu_1$ follows by induction. 
For $l\ge l_*$, $\hat\nu_0\ge 1$ implies 
\EQ{
 \nu_1=\sM(\hat\nu,\max(\nu_L,\nu_H))=\hat\nu\max(1,\nu_L,\nu_H),}
hence 
\EQ{
 \nu_{1\bullet} = \hat\nu_\bullet \max(1,\nu_{L\bullet},\nu_{H\bullet})
 = b\hat\nu \max(1,\nu_{L\bullet},\nu_{H\bullet}),}
so $\nu_{1\bullet}\ge b\nu_1$ for $l\ge l_*$ follows from the monotonicity of $\nu_L,\nu_H$. 
\end{proof}

The convolution estimate in $\X$ is transferred to $\X^{\mu,\nu,\io}_\p$ by
\begin{lem} \label{lem:conv Xp}
Let $\mu_k,\nu_k\in C^\ndc$ for $k=0,1,2$ such that for all $F,G\in\X$ and $l\in(0,\I)$ we have 
$\ro_l^{\mu_0,\nu_0}(F*G) \le \ro_l^{\mu_1,\nu_1}(F) \ro_l^{\mu_2,\nu_2}(G)$.  
Then for any $\ka\in C^\ndc$, we have
\EQ{
 \|F*G\|_{\X^{\mu_0,\nu_0,\ka\nu_0}_\p} \le \|F\|_{\X^{\mu_1,\nu_1,\ka\nu_1}_\p}\|G\|_{\X^{\mu_2,\nu_2}}.}
\end{lem}
\begin{proof}
Let $F\in\X^{\mu_1,\nu_1,\ka\nu_1}_\p$ and $F=F_0+\p_1F_1+\cdots+\p_dF_d$ for some $F_0\in\X^{\mu_1,\nu_1}$ and $F_k\in\X^{\mu_1,\ka\nu_1}$ ($k\ge 1$). Then, we have $F*G = F_0*G+\sum_{1\le k\le d} \p_k(F_k*G)$ with
\EQ{
 \pt \|F_0*G\|_{\X^{\mu_0,\nu_0}} \le \|F_0\|_{\X^{\mu_1,\nu_1}}\|G\|_{\X^{\mu_2,\nu_2}},
 \pr \|F_k*G\|_{\X^{\mu_0,\ka\nu_0}} \le \|F_k\|_{\X^{\mu_1,\ka\nu_1}}\|G\|_{\X^{\mu_2,\nu_2}},}
by using Lemma \ref{lem:weight change}. The infimum for $\{F_k\}_{0\le k\le d}$ yields the conclusion.
\end{proof}

\section{Global wellposedness for distributions} \label{s:GWP}
Consider the Taylor expansion of $\sH=(\sH_1\etc\sH_{n_3})$ at $z=0\in\C^{n_2}$ 
\EQ{ \label{Taylor1}
 \sH_k(z)=\sum_{\al\in\N_0^{n_2},\ |\al|\ge 2} \tf{1}{\al!}\sH_k^{(\al)}(0) z^\al \pq(k=1\etc n_3),}
which is assumed to be absolutely convergent in a $0$-neighborhood in $\C^{n_2}$: There exists $R_\sH>0$ such that for all $r<R_\sH$, we have 
\EQ{\label{Taylor2}
 \ti \sH(r):=\sum_{\al\in\N_0^{n_2},|\al|\ge 2}\tf{1}{\al!}\|\sH^{(\al)}(0)\|_{\C^{n_3}}r^{|\al|-2}<\I.}
Note that $\sH(0)=0$ is crucial for the global wellposedness, since otherwise $\sH$ could produce $0$ frequency in the solution, leading to ODE blow-up in general (cf.~Section \ref{ss:opt HS}). 
On the other hand, $\sH'(0)=0$ is not an essential restriction, because the linear part may be included in $L$. 

The difference of $\sH$ between two points may be estimated using
\EQ{ \label{Taylor3}
 \ti \sH^{(1)}(r):=\sum_{1\le j\le n_2}\sum_{\al\in\N_0^{n_2},|\al|\ge 1}\tf{1}{\al!}\|\sH^{(\al+e_j)}(0)\|_{\C^{n_3}}r^{|\al|-1},}
which is also finite for $r<R_\sH$. 
The Fourier transform (in $x$) of equation \eqref{GEE} is written for $\hat u:=\F u$ as
\EQ{ \label{Feq}
 \pt \hat u_t = \hat L\hat u + \hat N \hat \sH(\hat M\hat u),
 \pq \hat \sH_k(v):=
 \sum_{\al\in\N_0^{n_2}} \tf{1}{\al!}\sH_k^{(\al)}(0) v^{*\al}, \pq \hat u(0)=\hat u_0,}
where $v^{*\al}$ denotes the multi-index power in convolution, defined by
\EQ{
 v^{*\al}=v_1^{*\al_1}*\cdots*v_{n_2}^{*\al_{n_2}},
 \pq v_1^{*(a+1)}=v_1^{*a}*v_1, \pq v_1^{*0}=\de,}
and the Duhamel form is
\EQ{ \label{FDuh}
 \hat u(t) = \D(\hat u_0,\hat u):=e^{t\hat L(\x)}\hat u_0 + \int_0^t e^{(t-s)\hat L(\x)}\hat N(\x) \hat \sH(\hat M(\x)\hat u(s))ds.}
Note that the right side is closed on $C(\R;\X)$.
The idea is to solve it by the fixed point theorem using some norm on $C([0,\I);\X^{\mu,\nu})$ for appropriate choice of $\mu,\nu\in C^\ndc_0$, where $\nu$ depends on $t\ge 0$.

Fix $\de\in(0,1]$ and $\mu\in C^\ndc_\de$ satisfying $2\mu(l-\de)\le\mu(l)$ for $l>4\de$. 
If $\mu_0\in C^\ndc_\de$ is given, then we may choose $\mu:=\fC_\de^1(\mu_0)\ge\mu_0$ by Lemma \ref{lem:gen weight}. 
For the coefficients $\hat L,\hat M,\hat N$, 
assume that they are smooth on $\R^d_{>0}$ and that 
there exist $\te\in[0,1)$ and $c,\ti L,\ti M\in C^\ndc_{[1,\I)}$ satisfying for all $t,l>0$, $\fy\in(\sD_{<l})^{n_1}$, and $\psi\in(\sD_{<l})^{n_2}$,
\EQ{ \label{coeff bd}
 \pt \|e^{t\hat L^\tran(\x)}\fy\|_{\mu(l),\de} + \|t^\te\hat N^\tran(\x)e^{t\hat L^\tran(\x)}\fy\|_{\mu(l),\de} \le c(l)e^{t\ti L(l)}\|\fy\|_{\mu(l)}^+,
 \pr \|\hat M^\tran(\x)\psi\|_{\mu(l),\de} \le \ti M(l)\|\psi\|_{\mu(l)}^+. }
The factor $t^\te$ is to exploit smoothing effect of $e^{t\hat L}$ when it is available. 
Lemma \ref{lem:mult X} implies that \hyptag{C-s} with $\vv=e_1$ is a sufficient condition for the above to hold for all $\mu\in C^\ndc_0$. 
In particular, it is satisfied with $\te=0$ if all components of $\hat L,\hat M, \hat N$ are in $Y$ defined in \eqref{def Y}. For $d=1$, it is enough that $\hat L,\hat M, \hat N$ are smooth on $\x_1>0$ and bounded as $\x_1\to+0$. 

Now fix $\nu_0\in C^\ndc_0$, and let $l_*:=\sup \nu_0^{-1}([0,1])$. 
Without losing generality, we may assume $l_*<\I$ (otherwise we may multiply $\nu_0$ with a big constant without changing the space $\X^{\mu,\nu_0}$). 
Then we may replace $\ti L$ in \eqref{coeff bd} by a larger function if necessary such that 
$\ti L(l)$ is a constant for $0<l<l_*$. 

Fix $b_0>0$ and the initial data $\hat u_0\in\X^{\mu,\nu_0}$ such that $\ro^{\mu,\nu_0}(\hat u_0)\le b_0$.
For any $\nu_1\in C([0,\I);C^\ndc_0)$, $\hat u\in C([0,\I);\X^{n_1})$ and $l>0$, we obtain from \eqref{bd on Ffy-de} and \eqref{coeff bd}
\EQ{ \label{est0 Duh}
 \ro^{\mu(l)}_l(\D(\hat u_0,\hat u)) \pt\le c(l)e^{t\ti L(l)}\nu_0(l) \ro^{\mu,\nu_0}_l(\hat u_0)
 \prq+ \int_0^t c(l)(t-s)^{-\te}e^{(t-s)\ti L(l)}\nu_1(s,l)\ro^{\mu,\nu_1(s)}_l(\hat \sH(\hat M(\x)\hat u(s)))ds.}
Let $b_1:=2b_0$ and fix $\e>0$ small enough so that 
\EQ{ \label{choice e}
 b_1\e < R_\sH, \pq b_1\e \ti \sH(b_1\e) \Ga(1-\te) \le 1/2, \pq b_1\e \ti \sH^{(1)}(b_1\e) \Ga(1-\te) \le 1/2.}
Define $\nu_*,\nu_1,\nu_2,\nu:[0,\I)\to C^\ndc_0$ and $\ka\in C^\nin_{(0,1]}$ by using Lemma \ref{lem:gen weight} 
\EQ{
 \pt \nu_*(t,l):=\e^{-1} \nu_0(l)e^{t(1+\ti L(l))},
 \pq \ka:=1/(c \ti M), 
 \pr \nu_2(t):=\fC^2_\de(\mu,\nu_*(t),\ka), \pq \nu_1:=\ka\nu_2, 
 \pq \nu(t):=\e \nu_2(t)/\ti M=\e c \nu_1(t).}
At any fixed $t\ge 0$, let $v\in\X^{n_1}$ and $V:=\max_{1\le j\le n_2}\ro^{\mu,\nu_2}(v_j) \le \ro^{\mu,\nu_2}(v)$. If $V<R_\sH$ then repeated use of \eqref{conv est} to the Taylor expansion of $\hat \sH$ yields 
\EQ{ \label{est Hv}
 \ro^{\mu,\nu_1}(\hat \sH(v)) \le \sum_{|\al|\ge 2}\tf{1}{\al!}\|\sH^{(\al)}(0)\|_{\C^{n_3}}V^{|\al|} \le \ti \sH(V)V^2.}
Note that every summand contains convolution as $|\al|\ge 2$.
The difference for any $v^0,v^1\in\X^{n_1}$ is estimated in the same way as above, 
\EQ{
 \pt \ti V:=\max_{k=0,1} \max_{j=1 \etc n_2}\ro^{\mu,\nu_2}(v^k_j) < R_\sH
 \pr\implies \ro^{\mu,\nu_1}(\hat \sH(v^0)-\hat \sH(v^1)) \le \ti \sH^{(1)}(\ti V)\ti V\ro^{\mu,\nu_2}(v^0-v^1).}
Since $\ti M\nu\le \e\nu_2$, we have, again using \eqref{bd on Ffy-de} and \eqref{coeff bd}
\EQ{
 \ro^{\mu,\nu_2}(\hat M \hat u) \le \e\ro^{\mu,\nu}(\hat u),}
and, injecting it into \eqref{est Hv},
\EQ{
 \ro^{\mu,\nu_1}(\hat \sH(\hat M \hat u)) \le \ti \sH(\e\ro^{\mu,\nu}(\hat u))|\e\ro^{\mu,\nu}(\hat u)|^2, }
for all $t\ge 0$, provided that $\e\ro^{\mu,\nu}(\hat u)<R_\sH$.
Injecting this into \eqref{est0 Duh}, we obtain
\EQ{ \label{est1 Duh}
 \pt \ro^{\mu(l)}_l(\D(\hat u_0,\hat u)) \le c(l)e^{t\ti L(l)}\nu_0(l) b_0
 \prQ+ \int_0^t c(l)(t-s)^{-\te} e^{(t-s)\ti L(l)}\nu_1(s,l)\ti \sH(\e\ro^{\mu,\nu(s)}(\hat u(s)))|\e\ro^{\mu,\nu(s)}(\hat u(s))|^2ds.}
Suppose that $\ro^{\mu,\nu(t)}(\hat u(t))\le b_1$ for all $t\ge 0$. 
Since $\nu=\e c\nu_1\ge \e c\nu_*\ge c\nu_0e^{t\ti L}$, we obtain 
\EQ{
 \ro^{\mu,\nu(t)}(\D(\hat u_0,\hat u(t)))
 \le b_0+ \e \ti \sH(\e b_1)b_1^2 \sup_{l>0}\int_0^t (t-s)^{-\te}e^{s-t}
\tf{\ti\nu_1(s,l)}{\ti\nu_1(t,l)}ds, }
where
\EQ{
 \ti\nu_1(t,l):=e^{-t(1+\ti L(l))}\nu_1(t,l).}
Now the superlinear property (6) of $\fC^2_\de$ in Lemma \ref{lem:gen weight} implies 
\EQ{
 \tf{\ti\nu_1(s)}{\ti\nu_1(t)}=e^{(t-s)(1+\ti L)}\tf{\fC^2_\de(\mu,\nu_*(s),\kappa)}{\fC^2_\de(\mu,\nu_*(t),\kappa)} 
 \le  \tf{\fC^2_\de(\mu,e^{(t-s)(1+\ti L)}\nu_*(s),\kappa)}{\fC^2_\de(\mu,\nu_*(t),\kappa)} = 1,}
where we used that $\ti L$ is constant on $[0,l^*]$ and $\nu_*(s,l^*)\ge\nu_0(l^*)=1$. 
Thus we obtain
\EQ{
 \sup_{t\ge 0}\ro^{\mu,\nu(t)}(\hat u(t))\le b_1 \implies \ro^{\mu,\nu}(\D(\hat u_0,\hat u)) \le b_0+\e \ti \sH(\e b_1)b_1^2\Ga(1-\te),}
and also
\EQ{ \label{cont diff est}
 \pt\max_{k=0,1}\sup_{t\ge 0}\ro^{\mu,\nu}(\hat u^k)\le b_1 
 \implies \prq\sup_{t\ge 0}\ro^{\mu,\nu}(\D(\hat u_0,\hat u^0)-\D(\hat u_0,\hat u^1)) \le \e b_1\ti \sH^{(1)}(\e b_1)\Ga(1-\te)\sup_{t\ge 0}\ro^{\mu,\nu}(\hat u^0-\hat u^1).}
Hence by the choice of $\e$ in \eqref{choice e}, $\hat u\mapsto \D(\hat u_0,\hat u)$ is a contraction map on the complete metric space
\EQ{ \label{contraction set}
 \{\hat u \in C([0,\I);\X^{n_1}) \mid \sup_{t\ge 0}\ro^{\mu,\nu(t)}(\hat u(t))\le b_1=2b_0\} }
for the metric $\sup_{t\ge 0}\ro^{\mu,\nu}$, whose fixed point is a unique global solution of \eqref{FDuh}. For any $T>0$, the integral in \eqref{FDuh} may be defined as the Bochner integral in $\X^{\mu,\nu_T}$ for $t<T$, for $\nu_T:=\min_{0\le t\le T}\nu(t)\in C^\ndc_0$, as well as the Riemann integral in $\X$. 
In this sense, \eqref{FDuh} is equivalent to \eqref{Feq}. 
Thus we obtain a unique global solution $u$ of \eqref{GEE} in the inverse Fourier transform of the above space \eqref{contraction set}. 

\begin{thm} \label{thm:gex unif}
Let $\de\in(0,1]$ and $\mu\in C^\ndc_\de$ satisfy $2\mu(l-\de)\le\mu(l)$ for $l>4\de$. 
Suppose that there exist $\te\in[0,1)$ and $c,\ti L,\ti M\in C^\ndc_{[1,\I)}$ satisfying \eqref{coeff bd} for all $t,l>0$, $\fy\in(\sD_{<l})^{n_1}$ and $\psi\in(\sD_{<l})^{n_2}$.
Then for any $\nu_0\in C^\ndc_0$ and any $b_0\in(0,\I)$, there exists $\nu\in C([0,\I);C^\ndc_0)$ with $\nu(t)\ge\nu_0$ for all $t\ge 0$ such that for any $\hat u_0\in(\sD'(\R^d))^{n_1}$ satisfying $\ro^{\mu,\nu_0}(\hat u_0)\le b_0$, there is a unique global solution $u \in C([0,\I);(\F^{-1}\X)^{n_1})$ of \eqref{Feq} satisfying $\ro^{\mu,\nu(t)}(\hat u(t)) \in L^\I_t(0,\I)$.
\end{thm}

Note that the growth condition \eqref{coeff bd} on the coefficients $\hat L,\hat M,\hat N$ depends on $\mu$ or the growth of initial condition in $\x_1$. 
If we do not care about the dependence, then it may be simplified to the wellposedness in $\F^{-1}\X$. If the data is supported on a frequency set closed under addition, then we may also restrict the coefficients onto that set. 
In this paper, the global wellposedness in a topological space $X$ for $t\ge 0$ is defined by
\begin{enumerate}
\item For any $u(0)\in X$, there is a unique global solution $u\in C([0,\I);X)$.
\item For any sequence $\{u_n(0)\}_{n\in\N}\subset X$ convergent to some $u_\I(0)\in X$, let $u_*\in C([0,\I);X)$ be the corresponding global solutions. Then for any $T\in(0,\I)$, $u_n(t)\to u_\I(t)$ in $X$ uniformly for $0\le t\le T$ as $n\to\I$.
\end{enumerate}
The global wellposedness for $t\in\R$ is similarly defined, as well as the local one. 
The above definition is indeed the standard one in the Hadamard sense when $X$ is a metric space. 
However, if $X$ is not a sequential space, then (2) implies only the sequential continuity for the solution map from the initial data. 
Our space $\F^{-1}\X$ is not apparently sequential, in which the topological continuity can be a challenging problem, because the nonlinearity does not preserve the local convexity. 

\begin{cor} \label{GWP-XR}
Let $\cR\subset\R^d_{\ge 0}$ be a closed subset satisfying $\cR+\cR\subset\cR$. 
Suppose that $\hat L,\hat M,\hat N$ are smooth on a neighborhood of $\cR\cap \R^d_{>0}$, satisfying \hyptag{C-s} with $\vv=e_1$. 
Then the Cauchy problem \eqref{Feq} is globally wellposed in $(\F^{-1}\X_\cR)^{n_1}$ for $t\ge 0$.
\end{cor}
The simplest choice is $\cR=\R^d_{\ge 0}$, then $\X_\cR=\X$. 
As mentioned already, a sufficient condition is that $\hat L,\hat M,\hat N\in Y$ (cf.~\eqref{def Y}), and then the global wellposedness holds for $t\in\R$, because the last condition is invariant for the time inversion $t\mapsto-t$. 
\begin{proof}
Since the support property is closed in $\sD'(\R^d)$ and $\X\emb\sD'$, $\X_\cR\subset\X$ is also closed, and it is the inductive limit of $\X_\cR^{\mu,\nu}$. 
The closedness in convolution is immediate from $\cR+\cR\subset\cR$, 
while the closedness is obvious for multiplication by the coefficients. Hence all the operations in the construction of solutions, including the iteration argument, are closed in $\X_\cR$, 
so it suffices that the coefficients $\hat L,\hat M,\hat N$ are defined around $\cR$, and we do not need their smoothness near $\x_1=0$.  
Moreover, \hyptag{C-s} with $\vv=e_1$ implies that \eqref{coeff bd} is satisfied for all $\mu\in C^\ndc_0$ and $t,l>0$.

For any $\de\in(0,1]$ and $\mu_0\in C^\ndc_\de$, let $\mu:=\fC^1_\de(\mu_0)$. 
Then for any $\nu_0\in C^\ndc_0$ and $b_0\in(0,\I)$, by the above result, there exists $\nu\in C([0,\I);C^\ndc_0)$ such that for any $\hat u_0\in(\X_\cR)^{n_1}$ satisfying $\ro^{\mu,\nu_0}(\hat u_0)\le b_0$, there is a unique global solution $u\in C([0,\I);(\F^{-1}\X_\cR)^{n_1})$ of \eqref{FDuh} satisfying $\ro^{\mu,\nu(t)}(\hat u(t))\in L^\I(0,\I)$. 
Thus we obtain a global solution for any initial data $u_0\in(\F^{-1}\X_\cR)^{n_1}$. 

If we have two solutions $u_1,u_2\in C([0,\I);(\F^{-1}\X_\cR)^{n_1})$ for the same initial data $u_0$, 
then for any $T\in(0,\I)$, $\hat u_1([0,T])\cup \hat u_2([0,T])\subset(\X_\cR)^{n_1}$ is compact, so Proposition \ref{prop:cmp X} yields some $\mu_1,\nu_1\in C^\ndc_0$ and $b_1\in(0,\I)$ such that $\ro^{\mu_1,\nu_1}(\hat u_j(t))\le b_1$ for $j=1,2$ and $t\in[0,T]$. 
Then we may apply the uniqueness by contraction for some weight $(\mu_2,\nu_2)\ge(\mu_1,\nu_1)$ and the size $b_1$ to deduce that $u_1=u_2$ on $t\in[0,T]$. Since $T\in(0,\I)$ is arbitrary, we obtain the uniqueness in $C([0,\I);(\F^{-1}\X_\cR)^{n_1})$. 

For the sequential continuity of the solution map, take any sequence of initial data $\hat u_{0,n}\in(\X_\cR)^{n_1}$ convergent to some $\hat u_{0,\I}$ in $\X_\cR$, and let $u_n$ be the corresponding sequence of solutions. 
Since $\{\hat u_{0,n}\}_{n\in\N\cup\{\I\}}\subset(\X_\cR)^{n_1}$ is compact, 
Proposition \ref{prop:cmp X} yields some $\mu_0,\nu_0\in C^\ndc_0$ and $b_0\in(0,\I)$ such that 
$\{\hat u_{0,n}\}\subset B_{\X^{\mu_0,\nu_0}}(0,b_0/4)$, and that $\hat u_{0,n}\to\hat u_{0,\I}$ in $\X^{\mu_0,\nu_0,\io}_\p$ for any $\io\in C^\ndc_+$. 
After replacing $\mu_0$ with $\mu:=\fC^1_\de(\mu_0)$ as before, we use Lemma \ref{lem:mult X} to obtain the weight $\ga$ so that the coefficient multipliers are bounded also on $\X^{\mu,\nu,\ga\nu}_\p$. For the convolutions in the nonlinearity, we use Lemma \ref{lem:conv Xp}. 
Thus we obtain, in the same way as \eqref{cont diff est}, for $\hat u^0,\hat u^1$ satisfying $\sup_{t\ge 0}\ro^{\mu,\nu}(\hat u^k)\le b_1$, 
\EQ{ \label{diff est Xp}
 \pt\|\D(\hat u_0^0,\hat u^0)-\D(\hat u_0^1,\hat u^1)\|_{L^\I_t \X^{\mu,\nu,\ga\nu}_\p} 
   \prq\le \ro^{\mu,\nu_0}(\hat u_0^0-\hat u_0^1)
 \pn+4\e b_1\ti \sH^{(1)}(\e b_1)\Ga(1-\te)\|\hat u^0-\hat u^1\|_{L^\I_t \X^{\mu,\nu,\ga\nu}_\p},}
where the extra factor $4$ comes from the factor $2$ in Lemma \ref{lem:mult X} used twice. 
Since $\ro^{\mu,\nu_0}(\hat u_{0,n})<b_0/4$, 
the contraction on \eqref{contraction set} implies that $\ro^{\mu,\nu}(\hat u_n)\le b_1/4=b_0/2$ for all $n\in\N\cup\{\I\}$ and $t\ge 0$. 
Then \eqref{diff est Xp} with the smallness of $b_1\e$ in \eqref{choice e} implies 
\EQ{
 \|\hat u_n-\hat u_\I\|_{L^\I_t \X^{\mu,\nu,\ga\nu}_\p} \le 2\ro^{\mu,\nu_0}(\hat u_{0,n}-\hat u_{0,\I})
 \to 0,}
as $n\to\I$. 
Combining it with the uniform bound $\ro^{\mu,\nu}(\hat u_n)\le b_0/2$, we obtain by Proposition \ref{prop:cmp X} that $\hat u_n(t)\to\hat u_\I(t)$ in $\X$ locally uniformly for all $t\ge 0$. 
\end{proof}

A notable special case of $\X_\cR$ is when $\cR$ is a lattice:
\EQ{
 \exists L\in(0,\I)^d, \pq \cR \subset \prod_{1\le j\le d} \tf{2\pi}{L_j}\Z.}
In this case, $\cR$ should avoid the boundary points on $\x_1=0$, because such distributions do not satisfy the decay condition at $\x_1=0$, \hyptag{X-i}. 
Under the above constraint, the solutions in $\F^{-1}\X_\cR$ are periodic in space: For any $F\in\F^{-1}\X_\cR$ and $k\in\Z^d$, we have 
\EQ{
 F(x_1+k_1L_1\etc x_d+k_dL_d)=F(x)}
in the distribution ($\F^{-1}\hat\sD'$) sense. 

On the other hand, if $F\in \F^{-1}L^1_{loc}(\R^d)$, then $F$ is decaying at the spatial infinity in the average sense: By Riemann-Lebesgue,
\EQ{ \label{dec L1loc}
 \F\sD(\R^d)\ni\forall\fy, \pq \lim_{|x|\to\I}(\fy*F)(x)=0.}
For example, the Littlewood-Paley decomposition of $F$ satisfies the decay. 
Thus the space $\F^{-1}\X$ contains both the periodic solutions and the decaying solutions, which 
is in a sharp contrast with the usual settings for the Cauchy problem, where the boundary condition at infinity has to be distinguished by the solution space. 
 
The condition \hyptag{C-s} requires uniform boundedness of the coefficients $\hat L,\hat M,\hat N$ in the orthogonal directions.  
For much more general equations, it may be satisfied by restricting $\supp\hat u_0$ in the orthogonal directions, which is the case if $\cR\cap\R^d_{<l}$ is bounded for each $l\in(0,\I)$, then  
\hyptag{C-0} implies \hyptag{C-s} with $\te=0$ for all $t\in\R$. 
Thus Corollary \ref{GWP-XR} implies Theorem \ref{thm:gex supp cond}. 
Note that in this case \eqref{Feq} makes sense in $\X_\cR$, with the classical time derivative in the topology of $\X$.  
Moreover, $\X_\cR$ is a bit simplified:  
\begin{lem} \label{lem:simp X}
If $\cR\subset\R^d$ is closed such that $\cR\cap\R^d_{<l}$ is bounded for each $l>0$, then
\EQ{ \label{Supp cond}
 \X_\cR = \{F\in \sD'(\R^d) \mid \supp F\subset \cR,\ X^lF\in\M_+\text{ for some $l>0$} \}. }
If in addition $\cR\cap\R^d_{<l}=\empt$ for some $l>0$, then $\X_\cR=\sD'_\cR$. 
\end{lem}
In particular in the periodic case, $\cR$ can not contain any point on the boundary $\x_1=0$ as mentioned above, so the last statement holds. 
\eqref{Supp cond} applies to $\X$ itself if $d=1$. 
\begin{proof}
In view of Lemma \ref{lem:defX}, it suffices to derive the condition \hyptag{X-ii} from the other conditions. 
For any $l\in(0,\I)$, the boundedness of $\cR\cap\R^d_{<l}$ with the definition of $\sD'$ implies that $F$ is of finite order on $\R^d_{<l}$. Hence there exists $m(l)\in\N$ such that $|F(\fy)|\le m(l)\|\fy\|_{m(l)}=m(l)\|\fy\|_{m(l)}^+$ for all $\fy\in\sD_{(0,l)}$, which is enough for the condition \hyptag{X-ii}, by the proof of Lemma \ref{lem:defX}. 
The last sentence of lemma is obvious. 
\end{proof}

Examples of $\cR$ satisfying the local boundedness 
are conical regions in the form
\EQ{
 \cR:=\{\x\in\R^d : |\x^\perp|\le R(\x_1)\},}
for some $R\in C^+_0$ that is super-additive ($R(a)+R(b)\le R(a+b)$ for all $a,b\ge 0$), 
which implies $\cR+\cR\subset\cR$. 
The left side $|\x^\perp|$ may be replaced with any semi-norm on $\R^d$ that bounds $\x^\perp$. 
We may choose $R$ to grow arbitrarily fast by
\begin{lem} \label{lem:super-additive}
For any $d\in\N$, any non-empty $A\subset\R^d_{\ge 0}$ and any locally bounded $R_0:A\to[0,\I)$ satisfying 
\EQ{
 \limsup_{\x\in A,\ |\x|\to+0}|\x|^{-1}R_0(\x)<\I,} 
there exists a super-additive $R\in C^\ndc_0$ satisfying $R_0(\x)\le R(|\x|)$ for all $\x\in A$. 
\end{lem}
\begin{proof}
Let $R_1(l):=\sup\{|\y|^{-1}R_0(\y)\mid \y\in A,\ 0<|\y|\le l\}$. 
Then $R_1:(0,\I)\to[0,\I)$ by the assumptions, and it is non-decreasing. Then $R:[0,\I)\to[0,\I)$ defined by 
$R(l):=\int_l^{2l} R_1(t)dt$ 
belongs to $C^\ndc_0$, satisfying  
\EQ{
 R(a+b)\ge \int_{a+b}^{2a+b} R_1(t-b)dt + \int_{2a+b}^{2a+2b}R_1(t-2a)dt = R(a)+R(b)}
and $R(|\x|)\ge |\x|R_1(|\x|)\ge R_0(\x)$ for all $\x\in A$.
\end{proof}

\begin{rem}
In addition to the translation invariance, the space $\F^{-1}\X$ is invariant for the exponential Fourier multipliers in the form $e^{z\p_1}:=\F^{-1}e^{iz\x_1}\F$ for $z\in\C$, 
which is the holomorphic extension for the $x_1$ variable. 
Indeed, the multiplication with $e^{iz\x_1}$ is bounded $\X^{\mu,\nu_0}\to\X^{\mu,\nu_1}$ if $e^{|\im z|l}\nu_0(l)\le\nu_1(l)$ for all $l>0$, so it is continuous on $\X$. 
Moreover, the convolution is invariant: for any $F,G\in\X$ and $z\in\C$, we have 
\EQ{
 e^{iz\x_1}(F*G)=(e^{iz\x_1}F)*(e^{iz\x_1}G),}
which is immediate for $F,G\in\sD(\R^d)$ and then extended by the density of $\sD_{>0}\emb\X$, Proposition \ref{prop:dense X}. 
Hence the set of solutions has the same invariance: If $u\in C(I;(\F^{-1}\X)^{n_1})$ is a solution of \eqref{GEE} on an interval $I$, then $e^{z\p_1}u\in C(I;(\F^{-1}\X)^{n_1})$ is also a solution for any $z\in\C$. 
\end{rem}


\section{Comparison with other solutions} \label{s:comp}
Although our definition of solutions is quite natural in view of the Fourier transform of the equation, it is non-trivial to compare those solutions with others defined and constructed in different ways. 
The bottleneck to use the uniqueness in $\X$ is the Fourier support condition and decay at $\x_1=0$. 

Nevertheless, the following is a natural procedure to compare our solutions in $\X$ with some others, when it is easy to obtain the latter. 
Suppose 
\begin{enumerate}
\item Global wellposedness in $\F^{-1}\X_\cR$ holds (Corollary \ref{GWP-XR}) for some $\cR\subset\R^d$ such that the closure of the interior of $\cR$ contains $0$. 
\item Local wellposedness holds in another space $X$ such that $\sD_{>0}\emb \F X \emb \sD'(\R^d)$.
\item $I_n$ in Proposition \ref{prop:dense X} is an approximating sequence also on $X$. 
\end{enumerate}
Then we may choose the approximating sequence $I_n$ of Proposition \ref{prop:dense X} such that $I_n:\X_\cR\to\sD_\cR$. 
For any initial data $u_0\in(\F^{-1}\X_\cR)^{n_1}$, we have an approximating sequence $u_{0,n}:=\F^{-1}I_n\hat u_0\in(\F^{-1}\sD_\cR)^{n_1} \emb X^{n_1}$. 
Let $u_n$ be the solution for the initial data $u_{0,n}$ for each $n\in\N$. 
The global wellposedness in $\F^{-1}\X_\cR$ implies that $u_n \to u$ in $C([0,\I);(\F^{-1}\X_\cR)^{n_1})$. 
The local wellposedness in $X$ implies that $u_n\in C([0,T_n);X^{n_1})$ for the maximal existence time $T_n\in(0,\I]$. 
If $u_0\in X^{n_1}$, then we have $u_{0,n}\to u_0$ by the approximation property in $X$, so the local wellposedness in $X$ implies $u_n\to u$ in $C([0,T);X^{n_1})$, where $T\in(0,\I]$ is the maximal time of $u$. 
Here the limit $u$ is the same, as long as the convergence holds, because both of $\X$ and $\F X$ are continuously embedded into $\sD'(\R^d)$. 

Therefore, the solutions are the same between $\F^{-1}\X_\cR$ and $X$, for all the initial data in their intersection as long as the solution exists in $X$. 
In this sense, we may regard the solutions in $\F^{-1}\X_\cR$ as extension from $X\cap\F^{-1}\X_\cR$. 
However, unless the global wellposedness holds also in $X$, the maximal existence time $T_n$ may well go to $0$ as $n\to\I$ when $u_0\not\in X^{n_1}$. 

\subsection{Solutions beyond blow-up} \label{ss:beyond bup}
Blow-up may well happen in finite time in $X$, while the solution is extended globally in $\F^{-1}\X_\cR$. 
A typical example is the ODE
\EQ{ \label{simple ODE}
 \dot u = u^2,}
whose Cauchy problem is explicitly solved with the Taylor expansion: 
\EQ{ \label{ODE formula}
 u(t) = \tf{u_0}{1-tu_0} = \sum_{n\in\N_0} t^n u_0^{n+1},}
whose radius of convergence is $|tu_0|<1$ with the simple pole at $tu_0=1$. 
The iteration sequence $\{v_n\}_{n\ge 0}$ defined by 
\EQ{ \label{def itera}
 \dot v_n=v_{n-1}^2, \pq v_0(0)=u_0,}
is in the form (which may be shown by induction)
\EQ{ \label{iteration seq}
 v_n(t) = \sum_{0\le k \le 2^n-1} c_{n,k}t^k u_0^{k+1}, \pq \CAS{c_{n,k}=1 &(0\le k\le n)\\ 0<c_{n,k}<1 &(n<k\le 2^n-1).} }
Using $\|f*g\|_{\M}\le\|f\|_{\M}\|g\|_{\M}$ for the Fourier transform of the equation, 
it is easy to see that this equation is locally wellposed in $\F^{-1}\M$, 
and the iteration converges to $u$ for $|t|\|u_0\|_{\F^{-1}\M}<1$, where $\|\fy\|_{\F^{-1}\M}:=\|\F\fy\|_{\M}$. 
Moreover, if $\hat u_0$ is a non-negative measure, 
then we have $\|\hat u_0*\hat u_0\|_{\M}=\|\hat u_0\|_{\M}^2$ 
and similarly for iterations, so 
\EQ{
 \|u(t)\|_{\F^{-1}\M}=\sum_{n\in\N_0} t^n\|u_0\|_{\F^{-1}\M}^{n+1}=\tf{\|u_0\|_{\F^{-1}\M}}{1-t\|u_0\|_{\F^{-1}\M}} }
blows up exactly as $t\|u_0\|_{\F^{-1}\M}\to 1-0$. 
Since the non-negative Fourier transform implies 
$\|u(t)\|_{L^\I}=u(t,0)=\|u(t)\|_{\F^{-1}\M}$, 
it also blows up in $L^\I_x$ and at $x=0$, as well as in smaller Banach algebras, such as $H^s(\R^d)\emb \F^{-1}L^1 \emb \F^{-1}\M$ for $s>d/2$. 
On the other hand, $u(t)=\sum_{n=0}^\I t^n u_0^{n+1}$ is convergent and smooth for all $t\in\R$ in $\X$. 
All the above remains valid when $\M$ is replaced with the closed subspace $L^1(\R^d)$. 

A more interesting case is the Burgers equation 
\EQ{ \label{Burgers}
 u_t = u_{xx} + 2uu_x}
for $u(t,x):\R^{1+1}\to\C$, which is well-known to be explicitly solvable by the Cole-Hopf transform, thereby the blow-up has been intensively studied, cf.~\cite{PoSv} for the complex-valued blow-up solutions on $\R$. 
A simple example within our function space is for the monochromatic initial data 
\EQ{ \label{monochro}
 u_0(x) = -iae^{ix},}
with a parameter $a\in\C$, for which \cite{KhWuYuZh} proved blow-up in $L^2(\R)$. 
Here we investigate the unique global solution $u\in C(\R;\F^{-1}\X)$ (given by Corollary \ref{GWP-XR}), using the explicit formula in the Fourier space as in \cite{KhWuYuZh}:
\EQ{ \label{Burgers in F}
 \pt u(t) = \sum_{n\in\N} u^n(t)e^{inx}, \pq u^1(t)=-iae^{-t},
 \pr u^n(t) = in\int_0^t e^{-n^2(t-s)}\sum_{1\le k\le n-1}u^k(s)u^{n-k}(s)ds \pq(n\ge 2).}
Changing the variable from $u^n$ to $U_n$ defined by 
\EQ{
 u^n(t)=-i(ae^{-t})^nU_n(t),}
the equation may be rewritten as 
\EQ{
 U_1(t)=1, \pq U_n(t)=\int_0^t e^{-n(n-1)(t-s)} n\sum_{1\le k\le n-1} U_k(s)U_{n-k}(s)ds.}
Then we easily observe by induction on $n\in\N$ that 
\EQ{
 U_2(t)^{n-1} = (1-e^{-2t})^{n-1} \le U_n(t) \le 1.}
Indeed, it is obviously satisfied for $n=1,2$, while for $n\ge 3$, the upper bound is immediate by  injecting the induction hypothesis for $k=1\etc n-1$ into the equation. 
For the lower bound, the same injection yields 
\EQ{
 U_n(t) \ge e^{-n(n-1)t}\int_0^t n(n-1)e^{n(n-1)s}U_2(s)^{n-2}ds,}
and the conclusion follows by replacing the integrand with 
\EQ{
 \p_s[e^{n(n-1)s}U_2(s)^{n-1}] \pt=n(n-1)e^{n(n-1)s}U_2(s)^{n-2}(1-\tf{n-2}{n}e^{-2s}) 
 \pr\le n(n-1)e^{n(n-1)s}U_2(s)^{n-2}. }
For the original Fourier coefficients, we thus obtain 
\EQ{ \label{est Un}
 |a|e^{-t}\cdot |ae^{-t}(1-e^{-2t})|^{n-1} \le |u^n(t)|=|ae^{-t}|^nU_n(t) \le |ae^{-t}|^n. }
Note that $e^{-t}(1-e^{-2t})$ attains its maximum $\tf2{3\sqrt{3}}$ at $t=\tf12\log 3$. 

On the other hand, the equation is locally wellposed in $\F^{-1}\M$, where the above solution $u$ satisfies 
\EQ{ \label{FM norm}
  \sum_{n\in\N}|u^n(t)| = \|u(t)\|_{\F^{-1}\M} = iu(t,x_a) = \|u(t)\|_{L^\I_x},}
for all $x_a\in\R$ satisfying $ae^{ix}=|a|$. 
The classical blow-up occurs when the above diverges. 
Hence the threshold on $|a|$ for blow-up is given by
\EQ{ \label{def a*}
 \pt \fB:=\{(t,a)\in[0,\I)^2 \mid \sum_{n\in\N}(ae^{-t})^nU_n(t)=\I\}, 
 \pr a_*(t):=\inf\{a\mid (t,a)\in \fB\}, \pq 
 a_{**}:=\inf_{t\ge 0}a_*(t).}
\eqref{est Un} implies
\EQ{ \label{est a*}
 e^t \le a_*(t) \le \tf{e^t}{1-e^{-2t}}, \pq 1\le a_{**} \le \tf{3}{2}\sqrt{3}=\inf_{t\ge 0}\tf{e^t}{1-e^{-2t}},}
while the local wellposedness in $\F^{-1}\M$ implies $\lim_{t\to+0}a_*(t)=\I$. 
The classical blow-up time is given by 
\EQ{
 T_*(a):=T_*(|a|)=\inf\{t>0 \mid a_*(t)\le|a|\}>0 \pq(\inf\empt=\I). }
Indeed, for $0\le t<T_*(a)$, $|a|<a_*(t)$ implies that $u^n(t)$ is exponentially decaying as $n\to\I$, so $u(t)$ is analytic. 
If $u(T_*(a))\in\F^{-1}\M$, then the local wellposedness in $\F^{-1}\M$ implies that $u\in C([0,T];\F^{-1}\M)$ for some $T>T_*(a)$, and also for slightly bigger $|a|$ by the continuous dependence on $u(0)$, contradicting the definition of $T_*(a)$. 
Hence, if $T_*(a)<\I$, then $(T_*(a),|a|)\in\fB$ and $T_*(a)=\min\{t\mid (t,|a|)\in\fB\}$, so $a_*(T_*(a))\le |a|$, while $T_*(a_*(t))\le t$ is obvious by definition.   
Then \eqref{est Un} implies
\EQ{ \label{bup time est}
 e^{T_*(a)} \le |a| \le \tf{e^{T_*(a)}}{1-e^{-2T_*(a)}}. }
In particular $T_*(a)\le\log|a|$. It is obvious by definition that $T_*(a)=\I$ for $|a|<a_{**}$ and $T_*(a)<\I$ for $|a|>a_{**}$. 
Let $T_{**}:=T_*(a_{**})$. 
If $T_{**}=\I$ then by the local wellposedness as above, the solution with slightly bigger $|a|$ exists for $t<\log a_{**}+1$, contradicting $T_*(a)\le\log|a|$ for $|a|>a_{**}$. 
Hence $(T_{**},a_{**})\in\fB$, $a_*(T_{**})=a_{**}=\min a_*$, and 
$u$ blows up iff $|a|\ge a_{**}$, and at $t=T_*(a)\le\log|a|$. 
Since $T_*$ is non-increasing by definition, the blow-up time is bounded for all $a\in\C$ by 
\EQ{
 \max\{T_*(a) \mid a\in\C,\ T_*(a)<\I\}=T_*(a_{**}) \le \log|a_{**}|. }
If $a_*(t)<|a|$ (which happens iff $|a|>a_{**}$), $u^n(t)$ is exponentially divergent as $n\to\I$ (along a subsequence) so $u(t)\not\in\sS'(\R)$. 
There can be some time gap after $T_*(a)$ if $a_*(t)\ge|a|$ on some time interval starting from $T_*(a)$, but such intervals are mutually disjoint, so those exceptional $|a|$ are at most countable. Thus we obtain
\begin{prop}
There exist unique $a_{**}\in(1,\tf32\sqrt{3}]$, $T_{**}\in(0,\log \tf32\sqrt{3})$ and $a_*:(0,\I)\to[a_{**},\I)$, satisfying $\lim_{t\to+0}a_*(t)=\I$, \eqref{est a*} and $\min a_*=a_{**}=a_*(T_{**})$, such that 
the unique solution $u\in C(\R;\F^{-1}\X)$ of \eqref{Burgers}--\eqref{monochro} blows up in the  classical sense if and only if $|a|\ge a_{**}$ and at $T_*(a):=\min a_*^{-1}((0,|a|])\le\min(T_{**},\log|a|)$. 
For those $t>0$ where $|a|<a_*(t)$, $u(t)$ is analytic. For $|a|>a_*(t)$, $u^n(t)$ is exponentially divergent as $n\to\I$ and so $u(t)\not\in\sS'(\R)$. 
Excepting a countable set of $|a|$, $u(t)$ gets out of $\sS'(\R)$ immediately after $T_*(a)$, while  $u(t)$ is analytic for $t<T_*(a)$ and for 
$t>T^*(a):=\sup a_*^{-1}((0,|a|]) \le \log|a|$. 
\end{prop}

\begin{rem} 
The above result may look contradictory with \cite[Theorem 2.1]{KhWuYuZh}, which reads blow-up at arbitrarily large time, but what is actually proven in \cite{KhWuYuZh} is that  
$\|u(t)\|_{L^2_x(\R/2\pi\Z)}\sim \|u^n(t)\|_{\ell^2_n(\N)}=\I$ for $T_0:=\sum_{k\ge 2}\tf{1}{k^2}\log\tf{3k-3}{2k-3}\le t\le\log\tf{|a|}{3}$, as well as exponential divergence of $u^n(t)$ for $T_0\le t<\log\tf{|a|}{3}$. 
The above result shows that the (first) blow-up time is earlier (which is also obvious from the argument in \cite{KhWuYuZh}), and uniformly bounded by $T_{**}$. Indeed, $3e^t>\tf{e^t}{1-e^{-2t}}$ for $t>T_1:=\tf12\log\tf32$ and \eqref{est a*} imply $|a|>a_*(t)$  for $T_1<t\le\log\tf{|a|}{3}$, while $T_1<\tf14\log 3<T_0$. 
Since the inverse function of $\ta:(\tf12 \log 3,\I)\ni t\mapsto \tf{e^t}{1-e^{-2t}}$ satisfies $\ta^{-1}(a)=\log a-a^{-2}+O(a^{-4})$ as $a\to\I$, 
the time interval $[T_0,\log\tf{|a|}{3}]$ considered in \cite{KhWuYuZh} is inside the set where $u(t)\not\in\sS'(\R)$ with a uniform distance from the boundary.  
\end{rem}

\begin{rem} \label{rem:Cole-Hopf}
The Cole-Hopf transform yields another solution $\ti u$ by the formula
\EQ{
 v:=\sum_{n\in\N_0}\tf{a^n}{n!}e^{inx-n^2t},
 \pq \ti u:=-v_x/v.}
The blow-up occurs at $v(t,x)=0$, before which $\ti u$ is analytic and the same as the solution in $\F^{-1}\X$, namely $\ti u|_{0\le t<T_*(a)}= u|_{0\le t<T_*(a)}$.  

To see what happens after that, let $V(t,x):=\sum_{n\in\N_0}\tf{1}{n!}e^{inx-n^2t}$, which is analytic on $H:=(0,\I)\times\C$ satisfying the heat equation and $v(t,x)=V(t,x-i\log a)$. 
Let $Z:=\{(t,x)\in H\mid V(t,x)=0\}$ and $D:=\{(t,x)\in Z\mid V_x(t,x)=0\}$. 
Then 
\cite[Theorem 4.1]{PKS} implies that $D$ is discrete in $H$. 
For any $(t_0,x_0)\in Z\setminus D$, the Weierstrass preparation theorem yields analytic functions $z(t),A(t,x)$ defined around $(t,x)=(t_0,x_0)$ such that 
\EQ{ \label{zero curve}
 V(t,x)=A(t,x)(x-x_0-z(t)), \pq z(t_0)=0, \pq A(t_0,x_0)\not=0.}
Let $E$ be the set of all $(t_0,x_0)\in Z\setminus D$ such that $\im z(t_0-\e)\not=0\le \im z(t_0-\e)\im z(t_0+\e)$ for all small $\e>0$. Then the analyticity of $z$ implies existence of a positive even integer $m$ such that 
$\im\p_t^kz(t_0)=0\not=\im\p_t^mz(t_0)$ for $k<m$. 
Since $Z=\{(t,x_0+z(t))\}$ in a neighborhood of $(t_0,x_0)$, 
we deduce that $E$ is also discrete in $H$. 
On the other hand, $V(0,x)=e^{e^{ix}}$ has no zero for $x\in\C$, $V(t,x)\to 1$ as $t\to\I$ uniformly for bounded $\im x$, and $|V(t,x)-1|\le e^{e^{-\im x}}-1$. 
Hence for any $\la\in\R$, there is $\de>0$ such that for any $(t,x)\in Z$ with $\im x>\la$ we have $\de\le t \le 1/\de$ and $\im x\le 1/\de$. 
Therefore, 
\EQ{
 S_\la:=\{(t,x)\in D\cup E \mid 0\le \re x\le 2\pi,\ \im x>\la\} } 
is finite for any $\la\in\R$. 
Hence $A:=\{e^{-\im x} \mid (t,x)\in S_{-\I}\}\subset(0,\I)$ is locally finite. 

Now let $a\in\C$ satisfy $|a|\ge a_{**}$ and $|a|\not\in A$. 
Then $t_0:=T_*(a)\in(0,T_{**}]$ is the first $t>0$ such that $0\in v(t,\R)$. 
Let $\B(a):=\{x\in\R \mid v(t_0,x)=0\}$ be the blow-up set, $x_1\in\B(a)$ and $x_0:=x_1-i\log a$. 
Then $(t_0,x_0)\in Z\setminus(D\cup E)$. 
Let $z(t)$ as in \eqref{zero curve}. 
Since $0\not\in v(t,\R)$ for any $t\in(0,t_0)$, we have $\im z(t)\not=0$ for $t<t_0$. 
Since $(t_0,x_0)\not\in E$, the minimal integer $m$ such that $\im\p_t^m z(t_0)\not=0$ is odd. 
In other words, the blow-up for $|a|\not\in A$ occurs when some zeros of $v$ crosses $\R$ in $\C$. It also implies that $\B(a)$ is discrete in $\R$. 
Let $\B(a)\cap[0,2\pi)=\{x_1\etc x_J\}$ such that all $x_j$ are distinct.  
Then we have an open interval $I\ni t_0$ and analytic functions $z_j(t)$ defined on $I$ for $j=1\etc J$ such that $v(t,x)=A_j(t,x)(x-x_j-z_j(t))$ around $(t_0,x_j)$, for some analytic $A_j$ satisfying $A_j(t_0,x_j)\not=0$. Then there exists $\{c_1\etc c_J\}\in\C\setminus\{0\}$ such that 
\EQ{
 R(t,x):=\ti u(t,x) - \sum_{j=1}^J c_j[e^{-i(x-x_j-z_j(t))}-1]^{-1}  }
is analytic for $(t,x)\in I\times\R$, and that each $\im z_j(t)$ changes its sign at $t=t_0$. 
The last summand may be expanded by using 
\EQ{
 [e^{-i(x-z)}-1]^{-1} = \CAS{ \sum_{n\ge 1}e^{in(x-z)} &(\im z<0), \\
 -\sum_{n\le 0}e^{in(x-z)} &(\im z>0).}  }
Since $\ti u$ is analytic for $t<t_0$ and $\supp\F\ti u\subset[1,\I)$, 
we deduce that $\im z_j(t)<0$ for $t<t_0$ and $\im z_j(t)>0$ for $t>t_0$, for all $j$. Therefore
\EQ{
 \ti u(t,x)=\CAS{\sum_{j=1}^J c_j \sum_{n\ge 1}e^{in(x-x_j-z_j(t))}+R(t,x) &(t_0-\e<t<t_0)\\
 \sum_{j=1}^J c_j \sum_{n\le 0}e^{in(x-x_j-z_j(t))}+R(t,x) &(t_0<t<t_0+\e).} }
In particular, $\ti u(t)$ is discontinuous even in $\F^{-1}\sD'(\R)$ at the blow-up time $t=t_0=T_*(a)$. 
\end{rem}

The above construction of blow-up works partly for the equation
\EQ{
 u_t = \De u + 2u\cdot \na u, \pq u(t,x):\R^{1+d}\to\C^d,}
with the general dimension $d\in\N$ and the initial data in the form 
\EQ{
  u_0 = -a\na \F^{-1}\mu, \pq a>0,\pq 0\le \mu\in\M(\R^d), \pq 
 \supp\mu \subset B_{\R^d}(\vv,R), }
for some $\vv\in\R^d\setminus\{0\}$ and $R\in(0,|\vv|/\sqrt{2})$. 
Then the iteration 
\EQ{
 \pt u(t)=-\na\sum_{n\in\N}a^n\F^{-1}U_n(t), \pq U_1(t,\x)=e^{-t|\x|^2}\mu(\x),
 \pr U_n(t,\x)=\sum_{1\le k\le n-1}\int_0^t e^{-(t-s)|\x|^2}\int_{\R^d} (\x-\y)\cdot\y U_k(s,\x-\y)U_{n-k}(s,\y) d\y ds,}
gives the unique global solution $u\in C(\R;\F^{-1}\X(\vv))$ with 
$\supp U_n\subset nB_{\R^d}(\vv,R)$. 
Let 
\EQ{
 \pt \U{b}:=\inf_{\x\in\supp\mu}|\x| \ge |\vv|-R>0, \pq \bar{b}:=\sup_{\x\in\supp\mu}|\x|\in[\U{b}, |\vv|+R], 
 \pr \U{c}:=\inf_{\x,\y\in\supp\mu}\x\cdot\y
  \ge \min(1-\tf{R^2}{|\vv|^2},\tf12)(|\vv|^2-2R^2)>0,
 \pr \bar{c}:=\sup_{\x,\y\in\supp\mu}\x\cdot\y \in[\bar{c}, \bar{b}^2].}
Then the equation of $U_n$ is bounded from below and above by 
\EQ{
 \pt \tf 1n e^{-\bar{b}^2nt} \U{f}_n(t)\mu^{*n} \le U_n(t) \le \tf1n e^{-\U{b}^2nt} \bar{f}_n(t)\mu^{*n}, \pq \U{f}_1(t):= \bar{f}_1(t):= 1,
 \prq \U{f}_n(t) := \U{c} \int_0^t ne^{-\bar{b}^2n(n-1)(t-s)} \sum_{1\le k\le n-1} \U{f}_k(s)\U{f}_{n-k}(s) ds,
 \prq \bar{f}_n(t) := \bar{c} \int_0^t ne^{-\U{b}^2n(n-1)(t-s)} \sum_{1\le k\le n-1}\bar{f}_k(s)\bar{f}_{n-k}(s)ds,}
by induction on $n$. Then in the same way as before, we obtain inductively 
\EQ{
 \pt (\bar{b}^{-2}\U{c}(1-e^{-2\bar{b}^2t}))^{n-1} \le \U{f}_n(t) \le \bar{f}_n(t) \le (\U{b}^{-2}\bar{c})^{n-1}. }
The lower bound implies exponential growth of $\|a^nU_n(t)\|_\M$ as $n\to\I$ if
\EQ{
 a >  \tf{3\sqrt{3}}2\bar{b}^2\U{c}^{-1},}
at some $t\in(0,\I)$, where $(\tf{3\sqrt{3}}2)^{-1}=\max_{t>0}e^{-\bar{b}^2t}(1-e^{-2\bar{b}^2t})$ as before. By the Fourier support and positivity, it implies that $u(t)\not\in\sS'(\R^d)$ for those $t>0$. 
On the other hand, the upper bound implies that $\|a^nU_n(t)\|_\M$ decays exponentially as $n\to\I$ if 
\EQ{
 0\le a< \U{b}^2\bar{c}^{-1},}
uniformly in $t\ge 0$, and so $u$ is an analytic global solution on $t\ge 0$.

The above examples show that the global wellposedness without size restriction as in this paper can not hold within the tempered distributions $\sS'$, 
even with a smoothing linear part, a very mild nonlinearity, and analytic initial data.  

\subsection{Solutions beyond the natural boundary} \label{ss:beyond nb}
For the ODE $\dot u=u^2$, the solution formula \eqref{ODE formula} makes sense and smooth except at most the single time $t=1/u_0$ for each $u_0\in\C$. 
However, the global wellposedness in $\F^{-1}\X$ applies to more exotic equations as well, for example 
\EQ{
 \dot u = \sH(u) := \sum_{n=0}^\I u^{2^n}.}
This $\sH(u)$ is a well-known lacunary series with the natural boundary $|u|=1$, 
so it may not be extended anywhere of $|u|\ge 1$ as a holomorphic function. 
Nevertheless, the global wellposedness holds in $\F^{-1}\X$, which includes initial data such as 
\EQ{
 u_0(x) = a e^{i c\cdot x},}
for any $a\in\C$ and $c\in\R^d_{>0}$. 
If we choose $|a|\ge 1$, then $\sH(u_0(x))$ does not make sense for the holomorphic function $\sH$ at any $x\in\R^d$. 
On the other hand, the local wellposedness holds for $\|u_0\|_{\F^{-1}\M}<1$ and obtained by the iteration estimate as above, where the solutions are the same as those in $\F^{-1}\X$. 
In this sense, the global solutions in $\F^{-1}\X$ may be regarded as extension beyond the analytic continuation on $\C$. 
Actually, we could also consider everywhere divergent Taylor series if the initial Fourier support is uniformly away from the boundary $\x_1=0$, but we do not pursue it here. 

\section{Optimality of the conditions for the solution space} \label{s:opt}
In this section, we investigate the optimality of the function space $\X$ and the stronger support condition for the global wellposedness results, using the simple ODE \eqref{simple ODE} with the solution formula \eqref{ODE formula}.

\subsection{Frequency on the half space} \label{ss:opt HS}
The first example shows that 
$\hat u_0$ is not allowed to have point measures either in negative frequency $\x_1<0$ or on the boundary $\x_1=0$, for global existence in $\F^{-1}\sD'$. 
This means that the half space is optimal for the Fourier support condition, including the zero  boundary condition. 
 
Let $c\in\R^d$ and $u$ be the solution of \eqref{simple ODE} for the initial data 
\EQ{
 u(0,x)=\cos(c\cdot x), \pq \hat u(0,\x)=\tf12[\de(\x-c)+\de(\x+c)].}
Note that for a single point measure $\de(x-c)$ with $c\not=0$, we may choose the direction $\vv=c$ to satisfy $\supp\hat u_0\subset\R^d_{>0}(\vv)$ and so the global wellposedness follows.

Now we confirm blow-up of the above solution $u$ in $\F^{-1}\sD'(\R^d)$. 
Although it may look obvious from the solution formula, the previous section suggests that such very weak solutions might potentially survive beyond our intuition. 
The Cauchy problem is locally wellposed in $\F^{-1}\M$, where we have
\EQ{
 u(t,x)=\tf{\cos(c\cdot x)}{1-t\cos(c\cdot x)}=\sum_{n\ge 0} t^n\cos^{n+1}(c\cdot x),}
which is absolutely convergent in $\F^{-1}\M$ for $|t|<1$ and $u\in C((-1,1);\F^{-1}\M)$. 
If $c\not=0$ then the Fourier transform is in the form 
\EQ{
 \hat u(t,\x)=\sum_{j\in\Z}\sum_{n\ge 0}a_{j,n}t^n \de(\x-jc),}
for certain non-negative coefficients $a_{j,n}$ that may be computed explicitly. In particular at $\x=0$, the binomial expansion of $\cos^{n+1}(c\cdot x)=(\tf{e^{icx}+e^{-icx}}{2})^{n+1}$ yields
\EQ{
 \sum_{n\ge 0}a_{0,n}t^n = \sum_{k=1}^\I \mat{2k \\ k}\tf{t^{2k-1}}{2^{2k}} \ge \sum_{k=1}^\I \tf{t^{2k-1}}{2k} \to \I,}
as $0<t\to 1-0$. The above inequality follows from 
$\tf{a_{0,2k+1}}{a_{0,2k-1}}=\tf{2k+1}{2k+2} \ge \tf{k}{k+1}$ and $a_{0,1}=\tf{1}{2}$. 
Hence, coupling $\hat u(t)$ with a test function $\phi\in\sD(\R^d)$ with small support around $\x=0$, we deduce that $u(t)$ blows up in $\F^{-1}\sD'(\R^d)$ as $t\to 1-0$.  
In the case of $c=0$, it is more direct from the formula $\hat u(t,\x)=\tf{1}{1-t}\de(\x)$. 

\subsection{Decay to the boundary}
For finite measures on $\R^d_{\ge 0}$, the vanishing at the boundary $\x_1=0$ is equivalent to the decay as $\x_1\to+0$. 
It is the case for elements in $\X$, because they are required to be $\C$-valued measures near $\x_1=0$. 
Then it is natural to ask if the latter condition is necessary. 
In other words, can we weaken or replace the topology for the decay as $\x_1\to+0$, or can we include $\hat u_0$ that is not a measure near $\x_1=0$? 
The following shows that the measure space is optimal at least among the family of Sobolev or Besov spaces (in the Fourier side). Let 
\EQ{
 B^{-0}:=\Ca_{\s<0}(B^\s_{1,\I}\cap B^\s_{\I,\I})(\R^d) \emb \sD'(\R^d)}
be the Fr\'echet space endowed with a countable family of norms, e.g., $\|\fy\|_{B^{-1/n}_{1,\I}}+\|\fy\|_{B^{-1/n}_{\I,\I}}$ or $\|\LR{\na}^{-1/n}\fy\|_{L^1\cap L^\I}$ for $n\in\N$. It is also a convolution algebra because 
\EQ{
 \|f*g\|_{B^{2\s}_{p,\I}} \lec \|f\|_{B^\s_{1,1}}\|g\|_{B^\s_{p,1}} \lec \|f\|_{B^{\s/2}_{1,\I}}(\|g\|_{B^{\s/2}_{1,\I}}+\|g\|_{B^{\s/2}_{\I,\I}}), }
for any $\s<0$ and $p\in[1,\I]$, but we can not solve \eqref{simple ODE} in $\F^{-1}B^{-0}_{\ge 0}$. 
More precisely,
\begin{prop} \label{prop:ill B0}
For every $d\in\N$, there is a sequence $\{f_n\}_{n\in\N}\subset\sD_{>0}$ convergent in $B^{-0}$ such that the global solution $u_n\in C(\R;\F^{-1}\X)$ of \eqref{simple ODE} with $\hat u_n(0)=f_n$ is unbounded in $\F^{-1}\sD'(\R^d)$ as $n\to\I$ at any fixed $t\not=0$. 
The iteration sequence $v_n$ defined by \eqref{def itera} with $\hat u_0=\lim_{n\to\I}f_n$ is also unbounded at any fixed $t\not=0$. 
\end{prop}
Note that the convergence from $\sD_{>0}$ implies the decay of the limit $\hat u_0$ to the boundary $\x_1=0$ in the $B^{-0}$ norm. 
So we may not include any negative Sobolev or Besov space up to the boundary $\x_1=0$, for the local wellposedness, or to construct the solution by iteration.

\begin{proof}[Proof of Proposition \ref{prop:ill B0}]
For $k,n\in\N$, let 
\EQ{
 \pt r_k:=1/k!, \pq N_k:=k^{k^2}, \pq M_k:=2^{k^2},
 \pr g_k:=\chi(4r_{k+1}^{-1}|\x-r_ke_1|)\exp(-iN_k\x_1), \pq f_n:=\sum_{k=1}^n M_k g_k.}
Then $\supp g_k\subset\{|\x-r_k e_1|<r_{k+1}/2\}$ implies $g_k,f_n\in\sD_{>0}$, and we have
\EQ{
 \pt \ck g_k:=\F^{-1}g_k=[r_{k+1}^{d}\ck\chi(r_{k+1} x)e^{ir_k x_1}]|_{x\mapsto x-N_ke_1},}
where $\ck\chi:=\F^{-1}\chi(4|\x|)$, and 
\EQ{
 \|\LR{\na}^{-\e}g_k\|_{(L^1\cap L^\I)_\x}
 \lec \|(1-\De)^d[\LR{x}^{-\e}\ck g_k](x+N_ke_1)\|_{L^1_x} \lec N_k^{-\e},}
uniformly for $0\le\e\le1$ and $k\in\N$.  
Since $M_k/N_k^\e \to 0$ as $k\to\I$ for any $\e>0$, $f_n$ is convergent in $B^{-0}$. 

The global solution $u_n$ is given by the formula 
$\hat u_n(t)=\sum_{j=1}^\I t^{j-1}f_n^{*j}$ as before, 
and the support of $g_1\etc g_n$ implies that for any $\fy\in\sD_{<r_{n-1}-r_n/2}$ we have
\EQ{
 \hat u_n(t)(\fy) = \sum_{1\le j<s_n} t^{j-1}(M_ng_n)^{*j}(\fy)
 = \sum_{1\le j\le n-1} t^{j-1}(M_n\ck g_n)^j(\hat\fy),} 
where $s_n:=\tf{r_{n-1}-r_n/2}{r_n-r_{n+1}/2}=n-\tf{1}{2n+1}$. 
Since $(r_n+\tf12r_{n+1})(n-1)<r_{n-1}-\tf12r_n$, we may take $\fy=\fy_n:=(\bar g_n)^{*(n-1)}$, then
\EQ{
 (2\pi)^d \hat u_n(t)(\fy_n) \pt= \sum_{1\le j\le n-1} t^{j-1}M_n^j\LR{\ck g_n^j|\ck g_n^{n-1}}_{L^2(\R^d)},}
where the summand for $j=n-1$ is equal to 
\EQ{
 t^{n-2}M_n^{n-1}\|r_{n+1}^d\ck\chi(r_{n+1}x)\|_{L^{2(n-1)}}^{2(n-1)} = t^{n-2}M_n^{n-1}r_{n+1}^{d(2n-3)}\|\ck\chi\|_{L^{2n-2}}^{2n-2}, }
while for $j<n-1$ it is bounded by
\EQ{
 t^{j-1}M_n^j r_{n+1}^{d(n+j-2)}\|\ck\chi\|_{L^{n+j-1}}^{n+j-1}.}
Since $\ck\chi\in\sS(\R^d)$ with $\ck\chi(0)>0$, there are $c_0,c_1,c_2,c_3\in(0,\I)$ such that 
\EQ{
 c_0c_1^p \le \|\ck\chi\|_{L^p}^p \le c_2c_3^p}
for all $p\in[1,\I)$. Since $M_nr_{n+1}^d=2^{n^2}((n+1)!)^{-d}\gg A^n$ as $n\to\I$ for any $A\in(1,\I)$, 
we deduce that 
\EQ{
 |\hat u_n(t)(\fy_n)| \gec t^{n-2}M_n^{n-1}r_{n+1}^{d(2n-3)}c_1^{2n-2} \ge t^{-1}(c_1^2tM_n r_{n+1}^{2d})^{n-1} }
for sufficiently large $n$ (depending on $d,t$). 

On the other hand, since $\supp\fy_n\subset\{|\x|<2(n-1)r_n\}\subset\{|\x|\le 2\}$ for all $n\in\N$, if $\hat u_n(t)$ is bounded in $\sD'(\R^d)$, then there exists $m\in\N$ such that for large $n$
\EQ{
 |\hat u_n(t)(\fy_n)| \le m\|\fy_n\|_{\B^m} \le m\|g_n\|_{L^1}^{n-2}\|g_n\|_{\B^m}
 \le C(m)N_n^m.}
Comparison of those two estimates as $n\to\I$ leads to a contradiction, since 
\EQ{
 \pt \log (M_nr_{n+1}^{2d}) = n^2\log 2 - 2d\log(n+1)! \sim n^2,
 \pr \log (N_n^{m/(n-1)}) = \tf{m n^2}{n-1}\log n \sim mn\log n.}
Therefore, $\hat u_n(t)$ is not bounded in $\sD'(\R^d)$ at any $t\not=0$. 
The argument is exactly the same for the iteration $v_n$ because of the formula \eqref{iteration seq}. 
\end{proof}

\subsection{Integrability in the orthogonal directions}
Next we consider if $\X$ needs the integrability condition in $\x^\perp$, which is also quite strong in some sense. For example in the $L^2$-Sobolev spaces, we may easily observe that if $d\ge 2$ then  
\EQ{
 \{F\in H^s(\R^d) \mid \supp\hat F\subset\R^d_{\ge 0}\}\subset \F^{-1}\X \iff s>\tf{d-1}{2}.}
On the other hand, if $\cR\cap\R^d_{<l}$ is bounded for each $l>0$, then we have
$(\F H^s)_\cR\subset \X$ for all $s\in\R$.

First, the necessity of (lower) bounded regularity on $\R^d_{<l}$ for each $l>0$ is observed by the simple example
\EQ{
 \X \ni f_N:=\sum_{n=1}^Nc_n[\de(\x-le_1-ne_2)+(-\p_1)^n\de(\x-le_1+ne_2)] \to f_\I \IN{\sD'_{\ge 0}}, }
as $N\to\I$, where $l>0$ and $\{c_n\}_{n\in\N}\subset(0,\I)$ are arbitrarily fixed. 
For any $\fy\in\sD(\R^d)$ with $\supp\fy\subset B_{\R^d}(2le_1,1)$, we have 
\EQ{
 (f_N*f_N)(\fy) = 2\sum_{n=1}^N c_n^2(\p_1^n\fy)(2le_1),}
so $f_N*f_N$ is not bounded in $\sD'(\R^d)$ as $N\to\I$, no matter how fast $c_n\to 0$ as $n\to\I$. 
Since $\supp f_N^{*j}\subset \{\x_1=jl\}$, it is easy to deduce that the solution $u_N$ of \eqref{simple ODE} with $\hat u_N(0)=f_N$ is also unbounded in $\sD'(\R^d)$ at any $t\not=0$, 
as well as for the iteration sequence for the initial data $\hat u(0)=f_\I$.  

Second, the optimality of $L^1_\x$ may be observed from the example 
\EQ{
 \D_{>0}\ni g_N:=\chi(\x_1-3)\LR{\x^\perp}^{-\s} \chi(|\x|/N) \to g_\I:=\chi(\x_1-3)\LR{\x^\perp}^{-\s} \IN{\B^\I}}
as $N\to\I$, where $\s\in(0,d-1)$ is arbitrarily chosen.  
Note that $\supp g_N,\supp g_\I\subset \R^d_{(1,5)}$. 
For $n\in\N$ satisfying $\tf{n-1}{n} < \tf{\s}{d-1}$, we have 
\EQ{
 g_N^{*n} \to g_\I^{*n} = \chi^{*n}(\x_1-3n)(\LR{\x^\perp}^{-\s})^{*n} \IN{\sD'_{\ge 0}},\pq
 (\LR{\x^\perp}^{-\s})^{*n} \sim \LR{\x^\perp}^{(n-1)(d-1)-n\s}, }
but if $\tf{n-1}{n}>\tf{\s}{d-1}$ then $g_N^{*n}$ is unbounded as $N\to\I$ in $L^1_{loc}$ and also in $\sD'$. 

Thus in the same way as before, we observe the unboundedness of solutions for $\hat u(0)=g_N$ and for the iteration sequence of $\hat u(0)=g_\I$. 
This means that the condition $X^l\hat u_0\in W^{-m,1}(\R^d)$ may not be replaced with $W^{-m,p}(\R^d)$, for any $p>1$, in order to keep the local wellposedness of \eqref{simple ODE}. 
Note also that $\F^{-1}g_N\to \F^{-1}g_\I$ converges in $H^{s}(\R^d)$ for $s<\s-\tf{d-1}{2}$. 
Hence $u_0\in H^s(\R^d)$ is not enough for $s<\tf{d-1}{2}$, even if the Fourier support is in a bounded interval of $\x_1$ away from $0$. 
However, the above example does not preclude the endpoint case $s=\tf{d-1}{2}$, 
for which the iterated convolution grows only logarithmically for $|\x^\perp|\to\I$. 

\subsection{Cone support for unbounded coefficients}
For $d\ge 2$ and general smooth coefficients $\hat L$, the support restriction to a cone region is necessary, as the following example shows. Here we include a non-negative linear part, namely
\EQ{ \label{u2eq}
 u_t = Lu + u^2, \pq 0\le \hat L(\x)\in C^\I(\R^d).}
Suppose that $\hat u_0\in L^1(\R^d)$, supported and continuous on $\{|\x^\perp|\le R_0(\x_1)\}$, and positive on the interior, for some continuous $R_0:[0,\I)\to[0,\I)$ satisfying $R_0(0)=0$.

Suppose for contradiction that $R_0(a)/a\to+\I$ as $a\to+0$ but $u$ is global for $t\ge 0$.
Then the positivity of the data and the coefficients implies
\EQ{ \label{u2eq lbd}
 \CAS{\hat u(t) \ge \hat u_0,\\
 \hat u \ge \hat u_1 \implies \hat u(t) \ge \int_0^t \hat u_1(s)*\hat u_1(s)ds.} }
Iteration of those inequalities yields a lower bound
\EQ{
 t>0,\ \x_1>0 \implies \hat u(t,\x) \ge \hat u_*(t,\x)>0,}
for some continuous $\hat u_*$ independent of $\hat L$. 
$\hat u_*$ is positive everywhere on $\x_1>0$ because of $R_0(a)/a\to\I$ and $\hat u_0>0$ on $|\x|<R_0(\x_1)$. 
Then we may choose $\hat L(\x)$ such that
\EQ{
 1\le \x_1 \le 2 \implies e^{\hat L(\x)}\hat u_*(1,\x) \ge 1.}
Then $\hat u(2,\x)\ge 1$ on $\R^d_{[1,2]}$.
Injecting it into \eqref{u2eq lbd} leads to blow-up of $\hat u$ for $t>2$ because $1$ is not integrable for $\x_2\in\R$. 

So we need $\limsup_{a\to+0}R_0(a)/a<\I$ to avoid such blow-up. 
Then Lemma \ref{lem:super-additive} yields $\cR=\{|\x|\le R(\x_1)\}\supset\supp\hat u_0$ so that the global wellposedness holds in $\F^{-1}\X_\cR$.

\section{Global solutions with locally integrable Fourier transform} \label{s:L1}
In this section, we consider solutions in $\F^{-1} L^1_{loc}(\R^d)\subset \F^{-1}\sD'(\R^d)$, with the Fourier support in the half space, assuming that $\hat L$ decays on the boundary. 
This allows considerable simplification of both the function spaces and the construction of unique global solutions. 
Note that for any $\nu\in C_0^\ndc$, the closure of $\sD_{>0}$ in $\X^{0,\nu}$ is the weighted $L^1$ space 
\EQ{
 \X^{0,\nu}_0:=\{F\in L^1_{loc}(\R^d) \mid \supp F\subset\R^d_{\ge 0},\ \nu(\x_1)^{-1}F\in L^1(\R^d)\}.}
The solutions in this section belong to such spaces (upto rotation in $\x$). 
Note that $u_0\in\F^{-1}L^1_{loc}(\R^d)$ is decaying at spatial infinity in the average sense \eqref{dec L1loc}. 

Let $\cR\subset\R^d$ be a closed subset satisfying $\cR+\cR\subset\cR$. 
Let $\fp:\cR\to[0,\I)$ be continuous and super-additive, namely $\fp(\x+\y)\ge\fp(\x)+\fp(\y)$ for all $\x,\y\in\cR$. 
Let $\hat L,\hat M,\hat N: \cR\to M(n_1), M(n_2,n_1), M(n_1,n_3)$ be measurable, satisfying for some 
$0<c_0\le 1\le C_0<\I$ 
\EQ{ \label{LMN-p bd}
 \s_{\max}(\tf{\hat L+\hat L^*}{2}) \le (1-c_0)\fp, \pq \|\hat N\|_{M(n_1)}+\|\hat M\|_{M(n_2,n_1)} \le C_0\LR{\fp},}
almost everywhere on $\cR\subset\R^d$, 
where $\s_{\max}$ denotes the maximal eigenvalue of the self-adjoint matrix. 
The above condition of $\hat L$ implies $\|e^{t\hat L(\x)}\|_{M(n_1)} \le e^{t(1-c_0)\fp(\x)}$ 
a.e. on $\cR$. Hence, all of $e^{t\hat L},\hat N,\hat M$ are locally bounded on $\cR$. 

The continuity and super-additivity of $\fp$ may be implemented as follows. If \eqref{LMN-p bd} is satisfied by a locally bounded $\fp_0:\cR\to[0,\I)$ in place of $\fp$, such that  
\EQ{
  \sup\{|\x|^{-1}\fp_0(\x) \mid \x\in\cR,\ |\x|<\de\}<\I}
for some $\de>0$, then Lemma \ref{lem:super-additive} yields a continuous and super-additive $\fp:\cR\to[0,\I)$ satisfying $\fp_0\le\fp$, hence \eqref{LMN-p bd} holds for this $\fp$. 

A difference from the previous setting in $\X$ is that $\s_{\max}(\tf{\hat L+\hat L^*}{2})$ is required to decay as $\fp\to 0$, prohibiting uniform instability of $L$ near $\x_1=0$ for $\cR=\R^d_{\ge 0}$. 
It is actually possible to allow $\hat L$ to be merely bounded as $\fp\to0$ as before, but it will complicate the weight for the solution space (double-exponential in time). 
Since many physical examples satisfy the decay condition, we choose the simplicity in this section.

The nonlinearity $\sH:\C^{n_2}\to\C^{n_3}$ is assumed to be holomorphic around $0\in\C^{n_2}$ and vanishing to the order $\sH(z)=O(z^{k_0+1})$ for some $k_0\in\N$, namely 
\EQ{
 \sH(z) = \sum_{\al\in \N_0^{n_2},\ |\al|\ge k_0+1}\tf{1}{\al!}\sH^{(\al)}(0)z^\al}
is absolutely convergent in a $0$-neighborhood.  
As $k_0\ge 1$ increases, we may take the bigger space of initial data. 
In the previous setting in $\X$, it was fixed to $k_0=1$. 
Under those assumptions on $L,M,N,\sH$, we consider the equation as before, namely
\EQ{ \label{CP}
 u= e^{tL(D)}u_0 + \int_0^t e^{(t-\ta)L(D)}N(D)\sH(M(D)u(\ta))d\ta.}
For the initial data, we introduce a class of weighted $L^1$ spaces. 
For any $k_0>0>s$, $\cR\subset\R^d$ and $\fp:\cR\to[0,\I)$, let 
$\cZ_\fp^{k_0,s}(\cR)$ be the set of all Banach spaces $Z\emb L^1_{loc}(\R^d)$ satisfying for all $\fy\in Z$
\begin{enumerate}
\item $\supp\fy\subset \cR$ and $\||\fy|\|_Z=\|\fy\|_Z$. 
\item $\sup_{c\in\cR}\|\fy(\x-c)\|_Z\le\|\fy\|_Z$. \label{Z-trans}
\item $\|W_{s,k_0}(\fp)\fy\|_{L^1(\cR)}\le \|\fy\|_Z$, \label{Z-L1}
\end{enumerate}
where the decreasing function $W_{s,k_0}:(0,\I)\to(0,\I)$ is defined by
\EQ{
 W_{s,k_0}(\ro):=\max(\ro^{-1/k_0},\ro-1/s)e^{s\ro}.}
$W_{s,k_0}(\fp)$ is decreasing for $\fp,-s,k_0>0$.

A typical setting is the following. Let $P:[0,\I)\to[0,\I)$ be continuous and super-additive. 
Then a maximal element of $\cZ^{k_0,s}_{P(\x_1)}(\R^d_{\ge 0})$ is the weighted $L^1$ on the half space:
\EQ{ \label{L1-Z}
 \pt L^{k_0,s}_P:=\{\fy\in L^1_{loc}(\R^d) \mid \supp\fy\subset\R^d_{\ge 0},\ \|\fy\|_{L^{k_0,s}_P}<\I\}, 
 \prq \|\fy\|_{L^{k_0,s}_P}:=\|W_{s,k_0}(P(\x_1))\fy\|_{L^1(\R^d_{\ge 0})}, }
for which the property \eqref{Z-trans} follows from the monotonicity of $W_{s,k_0}$.   
Moreover, 
\EQ{ 
 1/W_{s,k_0}(P) \in C_0^\ndc, \pq  L^{k_0,s}_P=\X^{0,1/W_{s,k_0}(P)}} 
with equivalent norms. Conversely, for any $\nu\in C^\ndc_0$ satisfying
\EQ{
 \limsup_{t\to+0}t^{-1/k_0}\nu(t)<\I,}
let $P(t):=t\sup_{0<\ta<t}\nu(\ta)^{k_0}/\ta$.  
Then $P$ is continuous and super-additive, satisfying for any $s<0$, $\nu\le P^{1/k_0}\le C_{s,k_0}/W_{s,k_0}(P)$ with some constant $C_{s,k_0}\in(0,\I)$ depending only on $s<0<k_0$. 
Hence $\X^{0,\nu}_0\emb L^{k_0,s}_P$.

We may also consider anisotropic $L^p$ spaces. Let $d_0\in\{1\etc d\}$, $\vv\in\R^d\setminus\{0\}$, $\la>0$, 
\EQ{
 P_\vv(\x):=\vv\cdot\x, \pq \cR:=\{\x\in \R^d_{\ge 0}(\vv): |\bar\x| \le \la P_\vv(\x)\},}
where we denote $\x=(\bar\x,\ti\x)\in\R^{d_0}\times\R^{d-d_0}$ for any $\x\in\R^d$. 
Then 
\EQ{ \label{anis-Z}
 \pt Z:=\{\fy\in L^1_{loc}(\R^d) \mid \supp\fy\subset \cR,\ \|\fy\|_Z:=\|(P_\vv^{-\s}+1)\fy\|_{L^{r}_{\bar\x} L^{1}_{\ti\x}}/C<\I\},}
belongs to $\cZ^{k_0,s}_{P_\vv}(\cR)$ for any $s<0$ for some $C\in(0,\I)$, provided that $\tf{d_0}{r'}+\s> \tf{1}{k_0}$. 
Indeed, the property \eqref{Z-trans} follows from the translation invariance of $L^rL^1$ together with monotonicity of $P_\vv$ on $\R^d_{\ge 0}(\vv)$, while \eqref{Z-L1} follows by H\"older:
\EQ{
 \pt\|\max(P_\vv^{-1/k_0},P_\vv)e^{sP_\vv}\fy\|_{L^1_\x(\cR)} 
 \pn\lec \|1_{|\bar\x|<1}|\bar\x|^{\s-1/k_0}+|\bar\x|e^{s|\bar\x|}\|_{L^{r'}_{\bar\x}} \| (P_\vv^{-\s}+1)\fy\|_{L^r_{\bar\x} L^1_{\ti\x}}. }

The following is the unique global existence in those spaces. 
\begin{thm} \label{thm:global L1}
Let $d,n_1,n_2,n_3,k_0\in\N$ and let $\cR,\fp,L,M,N,\sH$ be as above. 
Then for any $s<0$,  $a\ge 2$ and $B>0$, there exists $b\in[1,\I)$ such that for any $Z\in\cZ_\fp^{k_0,s}(\cR)$ and any $u_0\in\F^{-1} L^1_{loc}(\R^d)^{n_1}$ satisfying $\supp\hat u_0\subset\cR$ and $\|e^{s\fp^a}\hat u_0\|_Z\le B$, 
there exists a unique global solution $u$ of \eqref{CP} satisfying 
\EQ{
 e^{b(s\fp^a-t\fp)}\hat u(t) \in [Z(L^\I_t(0,\I))\cap C([0,\I);Z)]^{n_1}.} 
\end{thm}
\begin{proof}
We show that the equation \eqref{CP} defines a contraction mapping on closed balls in the Banach space of measurable $\hat u:[0,\I)\times\cR\to\C^{n_1}$ with the norm 
\EQ{
 \|u\|_{Z^{b,s}_\fp} := \|e^{b(s\fp^a-t\fp)}\hat u(t)\|_{Z(L^\I_t(0,\I))},}
if $b\ge 1$ is chosen large enough. 
For the initial data part, $\s_{\max}(\tf12(\hat L+\hat L^*))\le \fp$ on $\cR\supset\supp\hat u_0$ implies
\EQ{
 \|\F e^{t L}u_0\|_{Z^{b,s}_\fp} \le \|e^{t \hat L}\hat u_0\|_{Z^{1,s}_\fp} \le\|e^{s\fp^a}\hat u_0\|_Z \le B.}
For the nonlinear Duhamel part, it suffices to estimate for $k\ge k_0$
\EQ{ \label{Duhamel int}
 U(t,\x) \pt:=e^{b(s\fp^a-t\fp)}\int_0^t |e^{(t-\ta)\hat L}\LR{\fp}(\LR{\fp}\hat u^0)*\cdots*(\LR{\fp}\hat u^{k})(\ta)|d\ta
  \pr\le \int_0^t e^{(\ta-t)c_0b\fp}\fp G(\x)d\ta = \tf{1-e^{-tc_0b\fp}}{c_0b}G(\x),}
in terms of $U_j(\x):=\|e^{b(s\fp^a-t\fp)}\hat u^j\|_{L^\I_t(0,\I)}$, where
\EQ{ \label{def G}
 \pt G(\x):= \int_{\x=\x^0+\cdots+\x^k} \tf{\LR{\fp}}{\fp}e^{bs\Om}\prod_{0\le j\le k}\LR{\fp(\x^j)}U_j(\x^j) d\x^1\cdots d\x^k, 
 \prq \Om:=\fp(\x)^a-\sum_{0\le j\le k} \fp(\x^j)^a.}
The above estimate by $G$ used the super-additivity of $\fp$ and the assumption on $\hat L$ in the form $\|e^{(t-\ta)\hat L}\|_{M(n_1)} \le e^{(t-\ta)(1-c_0)\fp}$ and $\fp(\x) \ge \fp(\x^0)+\cdots+\fp(\x^k)$. 

In estimating $G$, by symmetry, we may assume that $\fp(\x^0)=\max_{0\le j\le k} \fp(\x^j)$. 
We split this domain of $\x^0\etc\x^k$ into three regions, depending on a parameter $\te\in(0,1]$. 
The super-additivity of $\fp$, the binomial expansion and Young imply  
\EQ{
 \pt \Om \ge \fp(\x)^a-\fp(\x^0)^a-\fp(\x-\x^0)^a \ge a\fp(\x^0)^{a-1}\fp(\x-\x^0),
 \pr \fp(\x) \ge \fp(\x^1)+\cdots+\fp(\x^k) \ge k[\fp(\x^1)\cdots\fp(\x^k)]^{1/k}.}
If $\fp(\x^0)\le \te$ then $\fp(\x^j)\le \te$ for all $j$, so we just discard $\Om\ge 0$ to get 
\EQ{
 G|_{\fp(\x^0)\le \te} \le \LR{\te}^{k+2} U_0*V_1^0*\cdots*V_k^0, \pq  V_j^0:=\fp^{-1/k}U_j|_{\fp\le\te}.}
If $\fp(\x^0)\ge \te$ and $\fp(\x-\x^0)\ge \te$, then we have $\fp(\x)\ge\te$ and 
\EQ{
 \pt \Om \ge (a-1)\te\fp(\x^0)^{a-1}+\te^{a-1}\fp(\x-\x^0) \ge \te^{a-1}[\fp(\x^0)+\cdots+\fp(\x^k)],
 \prq G|_{\fp(\x^0),\fp(\x-\x^0)\ge\te} \lec \tf{\LR{\te}}{\te}\tf{\LR{\s}}{|\s|} U_0*V_1^1*\cdots*V_k^1, \pq V_j^1:=\LR{\fp}e^{\s\fp}U_j,}
where $\s:=b\te^{a-1}s<0$ and the $\x^0$ dependent weight is bounded by $\LR{\ro}e^{\s\ro}\lec |\s|^{-1}\LR{\s}$. 
If $\fp(\x^0)\ge \te\ge \fp(\x-\x^0)$, then $\fp(\x^j)\le \te\le\fp(\x)$ for $j=1\etc k$, and
\EQ{
 \pt \Om \ge a\fp(\x^0)^{a-1}k[\fp(\x^1)\cdots \fp(\x^k)]^{1/k},
 \pq e^{bs\Om} \le (b|s|\Om)^{-1},  
 \prq G|_{\fp(\x^0)\ge \te \ge \fp(\x-\x^0)}  \le \tf{\LR{\te}^2}{\te^2}\tf{\te}{ak|\s|} U_0*V_1^0*\cdots*V_k^0.}
Thus we obtain for $k\ge k_0$ 
\EQ{ \label{est G-conv}
\pt G|_{\fp(\x^0)=\max_j\fp(\x^j)} \lec \tf{\LR{\te}^{k+2}}{\te}\tf{\LR{\s}}{|\s|} U_0*V_1^2*\cdots*V_k^2, 
 \pq V_j^2:=V_j^0+V_j^1.}

Since all $U_j,V_j$ are supported on $\cR$, and $Z\in\cZ^{k_0,s}_\fp$, the property \eqref{Z-trans} of $\cZ^{k_0,s}_\fp$ implies 
\EQ{
 \|U_0*V_1^2*\cdots*V_k^2\|_Z \le \|U_0\|_Z \prod_{1\le j\le k}\|V_j^2\|_{L^1},}
while the property \eqref{Z-L1} implies for $\s\le s$ that $\|V_j^2\|_{L^1} \le \ka \|U_j\|_Z$ 
with  
\EQ{
 \ka:=\sup_{\ro>0}\tf{\ro^{-1/k}|_{\ro\le\te}+\LR{\ro}e^{\s\ro}}{\max(\ro^{-1/k_0},\ro-1/s)e^{s\ro}} \le \te^{-1/k+1/k_0}e^{-s\te}+\LR{s}e^{(\s-s)\te}. }
For any $k_0\in\N$, $s<0$ and $\e>0$, there are $\te\in(0,1)$ and $b\in[1,\I)$ such that $\ka\le\e$ for all $k\ge k_0+1$. For $k=k_0$, we choose $\te=1$. Then  
\EQ{
 \|U_0*V_1^2*\cdots*V_k^2\|_Z \le \CAS{(2e^{-s})^{k_0}\prod_{0\le j\le k_0}\|U_j\|_Z &(k=k_0)\\ 
  \e^k \prod_{0\le j\le k}\|U_j\|_Z &(k\ge k_0+1).} }
Injecting it into \eqref{Duhamel int}, \eqref{def G} and \eqref{est G-conv}, we may exploit the denominator $b$ in \eqref{Duhamel int}, taking $b\ge 1$ larger if necessary. Thus we obtain 
\begin{lem}
For any $-\I<s<0<c_0,\e<1$, $a\ge 2$ and $k_0\in\N$, there exists $b\in[1,\I)$ such that for any $\cR,\fp,L$ as above, any $Z\in\cZ^{k_0,s}_\fp(\cR)$, any $k\in\N$ with $k\ge k_0$ and any $u^0\etc u^k$ satisfying $\supp \hat u^j(t)\subset\cR$ for all $t\ge 0$, we have 
\EQ{
 \|U\|_{Z(L^\I_t(0,\I))} \le \e^k \prod_{0\le j\le k}\|\hat u^j\|_{Z(L^\I_t(0,\I))},}
for $U$ defined by \eqref{Duhamel int}. 
\end{lem}
The above lemma yields an estimate on the Duhamel nonlinear part 
\EQ{
  \|\int_0^t e^{(t-\ta)L} N(D)\sH(M(D)u)d\ta\|_{Z^{b,s}_\fp}
 \le \sum_{|\al|\ge k_0} C_0 \tf{\|\p^\al \sH(0)\|_{\C^{n_3}}}{\al!}(C_0\e\|u\|_{Z^{b,s}_\fp})^{|\al|}, }
and a similar estimate for the difference.  
Hence, if we choose $\e\in(0,1)$ small enough depending on $B,\sH,C_0$, then 
the map $u\mapsto e^{tL}u_0+\int_0^t e^{(t-\ta)L}\sH(u(\ta))d\ta$ is a contraction on 
the closed ball $\{u \in Z^{b,s}_\fp \mid \|u\|_{Z^{b,s}_\fp}\le B\}$, 
and the unique global solution of \eqref{CP} is obtained as its unique fixed point. 

Since the Duhamel formula is continuous in $t$ at each $\x$, 
$u\in C([0,\I);Z)$ follows if $Z$ satisfies the dominated convergence. 
Otherwise, it suffices to take $b$ slightly bigger. 
\end{proof}

\section{Examples of the equations} \label{s:Ex}
We apply the above abstract results to some typical nonlinear evolution equations which come from  physics, including the Navier-Stokes, the Euler, the nonlinear Klein-Gordon equations, as well as some nonlinear and nonlocal equations. 
The application is straightforward, leading to new global wellposedness and unique existence results for more wild initial data than the standard settings.

\subsection{Incompressible Navier-Stokes and Euler equations}
Consider the Cauchy problem for the  incompressible
Navier-Stokes and Euler equations:
\begin{align}
u_t - \varepsilon_1 \p_1^2   u -  \varepsilon_\perp \De_\perp u +u\cdot \nabla u +\nabla p=0, \ \  {\rm div}\, u=0,  \label{NS1}
\end{align}
where   $u(t,x):\R^{1+d}\to\R^d$ denotes the flow
velocity vector and $p:\R^{1+d}\to\R$ describes the scalar pressure, 
$ \Delta_\perp :=\partial^2_2 +\cdots+ \partial^2_d $, and $\varepsilon_1,\e_\perp \geq 0$ are some constants.   In the Ekman boundary layers for rotating fluids (cf.~\cite{DeGr2000}), it makes sense to consider the different viscosities of the type $\varepsilon_1\ge 0$ and $\varepsilon_\perp>0$.   If $\varepsilon_1= \varepsilon_\perp =0$, then \eqref{NS1} is the Euler equation.
As usual, \eqref{NS1} may be rewritten in the equivalent form: 
\begin{align}
  u_t  - \varepsilon_1 \p_1^2   u -  \varepsilon_\perp \De_\perp u+\mathbb{P}\ \textrm{div}(u\otimes u)=0,
 \label{NS mod}
\end{align}
where $\mathbb{P}:=I-{\rm div}\, \Delta^{-1}\nabla$ is the projection onto the divergence free vector fields,  $I\in M(d)$ is the identity matrix. 
We need to extend the equation \eqref{NS mod} to the complex values $u:\R^{1+d}\to\C^d$, since otherwise our solution space contains only the trivial $0$. 
We may also extend the coefficients to $\e_1,\e_\perp\in\C$. 

Since the symbols for $L:=\e_1\p_1^2+\e_\perp\De_\perp$ and $N:=\mathbb{P}\Div$ are given by 
\EQ{
 (e^{t\hat L}\hat N)_{jk}=e^{-t\e_1\x_1^2-t\e_\perp|\x^\perp|^2}i(\x_k-\x_j) \pq(j,k=1\etc d),}
the condition \hyptag{C-s} is satisfied with $\vv=e_1$ and $\te=1/2$ if $\e_\perp>0$. 
Hence, Corollary \ref{GWP-XR} implies that the Cauchy problem for \eqref{NS mod} is globally wellposed in $\F^{-1}\X^d$ for $t\ge 0$, for any $d\ge 2$, $\e_1\in\C$ and $\e_\perp>0$. 
If both $\e_1,\e_\perp>0$, then it does not depend on the direction of the half space: For any $\vv\in\R^d\setminus\{0\}$, it is globally wellposed in $\F^{-1}\X(\vv)^d$ on $t\ge 0$. 

For more general $\e_\perp\in\C$, let $R\in C^+_0$ be super-additive and $\cR:=\{\x\in\R^d; |\x|\le R(\vv\cdot \x)\}$. 
Then for any $d\ge 2$ and $\e_1,\e_\perp\in\C$, Theorem \ref{thm:gex supp cond} implies that the Cauchy problem for \eqref{NS mod} is globally wellposed in $\F^{-1}\X_\cR(\vv)^d$ for $t\in\R$. 
The simplest choice of $\cR$ is $\{|\x|\le\la\vv\cdot\x\}$ for some constant $\la>0$, but it may be much bigger as $\vv\cdot\x\to\I$. 

Let $\fp(\x):=C[(\vv\cdot\x)+(\vv\cdot\x)^2]$ for some constant $C>0$. 
If $C$ is large enough, depending only on $\e_1,\e_\perp,\vv$, then \eqref{LMN-p bd} is satisfied for some $0<c_0\le 1\le C_0<\I$. 
Then Theorem  \ref{thm:global L1} implies that for any $s<0$, $a\ge 2$, $Z\in\cZ^{1,s}_\fp(\cR)$, and any $e^{s\fp^a}\hat u_0 \in Z$ with $\supp\hat u_0 \subset \cR$, there exists some $b\ge 1$ such that \eqref{NS mod} 
has a unique global solution $u$ satisfying $e^{b(s\fp^a-t\fp)}\hat u(t) \in [Z(L^\I_t(0,\I))\cap C([0,\I);Z)]^d$. 
See \eqref{L1-Z} and \eqref{anis-Z} for examples of $Z$. 

For the 3D Navier-Stokes ($d=3$, $\e_1=\e_\perp>0$), Li and Sinai \cite{LiSi} constructed complex-valued blow-up solutions with Fourier support on a ball far from the origin and inside a cone. 
More precisely, they constructed a set of initial data with 10 parameters for which the energy and the enstrophy (namely the $L^2$ and $\dot H^1$ norms) of the solution blow up in finite time. 
Since their initial data satisfy our conditions, it is another and highly non-trivial example of the global wellposedness beyond the classical blow-up. 
Although the inductive construction of the solution works in the same way as for the simple example in Section \ref{ss:beyond bup}, the main difficulty in the Navier-Stokes equation comes from the projection $\mathbb{P}$, which eliminates the major part of the quadratic interaction under the  Fourier support restriction. 
Similar blow-up sets of initial data have been constructed also for the 2D Burgers equation and other hydrodynamical models \cite{LiSi2,LiSi3}. 

\subsection{Compressible Navier-Stokes and Euler equations}
Next we consider the compressible Navier-Stokes equations for the ideal gas (cf.~Lions \cite{Lion1996})
$$
\left\{
\begin{array}{l}
\partial_t \varrho + \Div(\varrho u) =0, \\
\partial_t (\varrho u)  +\Div(\varrho u \otimes u) + \nabla p = \cL_{\al,\be}u:= (\alpha+\beta) \nabla \Div u + \be\De u, \\
\partial_t (\varrho e)  + \Div( \varrho u  e)  + p \Div u  -\kappa \Delta T 
 = b_{\al,\be}(u):=\tf{\be}{2}\trace(\nabla u + (\nabla u)^{\top})^2 + \al(\Div u)^2,
\end{array}
\right.
$$
where $u:\R^{1+d}\to\R^d$ is the velocity, 
$\varrho:\R^{1+d}\to[0,\I)$ is the density, 
$T:\R^{1+d}\to[0,\I)$ is the temperature, $p=R\varrho T$ is the pressure, $e=c_vT$ is the internal energy, for some constants $\al,\be,\ka,R,c_v>0$. 
If $\al=\beta=\kappa =0$, then they are the compressible Euler equations. 
If $\varrho\not=0$, then the equations of momentum and energy may be rewritten as 
\EQ{
 \pt \p_t u + u\cdot\na u + R(\na T+T\tf{\na\varrho}{\varrho}) = \tf{1}{\varrho}\cL_{\al,\be}u,
 \pr \p_t T + u\cdot\na T + \tf{R}{c_v}T\Div u - \tf{\ka}{c_v\varrho}\De T = \tf{1}{c_v\varrho}b_{\al,\be}(u).}
Let $\varrho=\ro_*+\ro$ for some constant $\ro_*>0$. Then the system is rewritten for $(\ro,u,T)$ as 
\EQ{
 \CAS{\p_t\ro + \ro_*\Div u = -\Div(\ro u), \\
 \p_t u - \tf{1}{\ro_*} \cL_{\al,\be}u + R\na T = -u\cdot\na u - \tf{R}{\ro_*+\ro}T\na\ro - \tf{\ro}{\ro_*(\ro_*+\ro)}\cL_{\al,\be}u, \\
 \p_t T - \tf{\ka}{c_v\ro_*}\De T = -u\cdot\na T-\tf{R}{c_v}T\Div u+\tf{1}{c_v(\ro_*+\ro)}b_{\al,\be}(u)-\tf{\ka \ro}{c_v \ro_*(\ro_*+\ro)}\De T.} }
Because of the quasi-linear terms ($\cL_{\al,\be}u$ and $\De T$ on the right side), 
the condition \hyptag{C-s} is not satisfied as in the incompressible case, but 
we may still apply Theorems \ref{thm:gex supp cond} and \ref{thm:global L1} in the same way for global wellposedness in $\F^{-1}\X_\cR(\vv)^d$. 
Note that the singularity at $\ro=-\ro_*$ does not matter in our setting, without any size restriction (cf.~Sections \ref{ss:beyond bup}--\ref{ss:beyond nb}). 
We may also take more general constants $\al,\be,\ka,R,1/c_v\in\C$. 

The same argument works also for the equations of $(\ro,u,\ta)$ with $T=T_*+\ta$ for some constant $T_*>0$. 
We may similarly consider the equations for the isentropic gas, where $p=a\varrho^\ga$ for some constants $a,\ga>0$. In this case, the equations become for $(\ro,u)$
\EQ{
 \CAS{ \p_t\ro + \ro_*\Div u = -\Div(\ro u), \\
 \p_t u - \tf{1}{\ro_*}\cL_{\al,\be} u + \tf{a \ga}{\ro_*^{2-\ga}}\na\ro = -u\cdot\na u + [\tf{a\ga}{\ro_*^{2-\ga}}-\tf{a\ga}{(\ro_*+\ro)^{2-\ga}}]\na\ro- \tf{\ro}{\ro_*(\ro_*+\ro)}\cL_{\al,\be}u,} }
to which Theorems \ref{thm:gex supp cond} and \ref{thm:global L1} are applicable, with any constants $\al,\be,\ga,a\in\C$. 

\subsection{Nonlinear wave and Klein-Gordon equations}
Consider the nonlinear Klein-Gordon equation in the form
\begin{align}\label{Eq NLKG}
\left\{
\begin{array}{l}
u_{tt} + c^2u - \Delta  u   = \sH(\p_x^{\al^1}u \etc \p_x^{\al^m}u, \p_x^{\be^1}u_t\etc \p_x^{\be^n}u_t),   \\
u(0,x)= u_0(x), \ \ u_t(0,x)= u_1(x),
\end{array}
\right.
\end{align}
for $u(t,x):\R^{1+d}\to\C$, with a constant $c\ge 0$, $m,n\in\N$, $\al^1\etc \al^m,\be^1\etc \be^n\in\N_0^d$ and $\sH:\C^{m+n}\to\C$ defined and holomorphic in a neighborhood of $0$, satisfying $\sH(0)=0$. 
It is the nonlinear wave equation if $c=0$. 
The equation may be expanded into a first order system for $U=(U_0\etc U_{d+1}):=(u,\na u,u_t)$
\EQ{
 \CAS{ \p_t U_0 - U_{d+1} = 0, \\ 
 \p_t U_k - \p_k U_{d+1} = 0 \pq (1\le k\le d), \\
 \p_t U_{d+1} + c^2 U_0 - \sum_{1\le k\le d}\p_k U_k = \ti\sH(U), \\ 
 U(0)=(u_0,\na u_0,u_1),} }
where $\ti\sH(U)$ is an analytic function of derivatives of $U$ derived from $\sH(\p^{\al^1}_xu\etc \p^{\be^n}_xu_t)$, by replacing $u_t$ with $U_{d+1}$ and $\p^\al_xu$ with $\p^{\al-e_k}_xU_k$ (if $\al\ge e_k$) or with $U_0$ (if $\al=0$). 
The linear part of $\ti\sH$ should be included in the operator $L$ in applying our results. 

For any solution $U$ of the above system, the equations imply $U_{d+1}=\p_t U_0$ and $U_k=\na U_0+\fy$ for some $\fy(x)$ independent of $t$. Hence the initial constraint $U_k(0)=\na U_0(0)$ implies $U_k=\na U_0$ for all time. Then by definition of $\ti\sH$, $U_0$ satisfies the original equation \eqref{Eq NLKG}. The converse relation is obvious from the definition of the system. 

If the equation is semilinear in $\x^\perp$, namely $d=1$ or
\EQ{ 
 \max_{1\le j\le m} \sum_{2\le k\le d} \al^j_k \le 1, \pq \max_{1\le j\le n}\sum_{2\le k\le d}  \be^j_k = 0,}
then $\ti\sH$ has no derivative in $\x^\perp$, while $e^{t\hat L}$ is bounded on $\x\in\R^d$, including derivatives, locally in $t\in\R$.  
Hence the condition \hyptag{C-s} is satisfied with $\vv=e_1$ and $\te=0$, 
and Corollary \ref{GWP-XR} implies the global wellposedness of the system for $U\in\F^{-1}\X^{d+2}$ on $t\in\R$. 
If $|\al^j|\le 1$ and $\be^j=0$ for all $j$, then it is independent of the direction of the half space: the global wellposedness holds in $\F^{-1}\X(\vv)^{d+1}$ for any $\vv\in\R^d\setminus\{0\}$. 

In the general case of arbitrary derivatives $\{\al^j,\be^j\}$, we may apply Theorems \ref{thm:gex supp cond} and \ref{thm:global L1} to the above system, which imply global wellposedness in $\F^{-1}\X_\cR(\vv)^{d+2}$ and unique global existence in weighted $Z$ with frequencies in the conical region $\cR$ as before.

\subsection{Nonlinear Schr\"odinger equations}
Similarly to the Klein-Gordon equations, we may consider the nonlinear Schr\"odinger equations in the form
\EQ{
 iu_t + \De u = \sH(\p_x^{\al^1}u \etc \p_x^{\al^m}u),}
for $u(t,x):\R^{1+d}\to\C$ with the same type of nonlinearity as above.
In this case, the global wellposedness by Corollary \ref{GWP-XR} works if there is no derivative in the orthogonal directions, namely $\al_k^j=0$ for all $j$ and $k\ge 2$. 
Otherwise, we may still apply Theorems \ref{thm:gex supp cond} and \ref{thm:global L1} for the global wellposedness on the conical region.  

Note that the standard nonlinear Schr\"odinger equations are not in the above form, because of the complex conjugate. However, we may consider the equations in the form 
\EQ{
 iu_t + \De u = \sH(\p_x^{\al^1}u \etc \p_x^{\al^m}u, \p_x^{\be^1}u^\star \etc \p_x^{\be^n} u^\star),}
where $u^\star$ denotes the complex conjugate in the Fourier space, namely 
\EQ{
 u^\star(t,x) := \ba{u(t,-x)}, \pq \F u^\star(t,\x) = \ba{\F u(t,\x)}.}
See \cite{AbMu1,AbMu2} for the one-dimensional integrable models, with cubic terms containing one $u^\star$ and at most one derivative. 
Unique global solutions were constructed in \cite{ChLuWa} under the Fourier support restriction in  $(0,\I)$. 

The above equation may be expanded to the system of $U=(u,v):=(u,u^\star)$ with 
\EQ{
 \CAS{iu_t + \De u = \sH(\p_x^{\al^1}u \etc \p_x^{\al^m}u, \p_x^{\be^1}v \etc \p_x^{\be^n} v),\\
 iv_t - \De v = -\bar{\sH}((-\p_x)^{\al^1}v \etc (-\p_x)^{\al^m}v, (-\p_x)^{\be_1}u \etc (-\p_x)^{\be^n} u),}}
where $\bar{\sH}$ is the power series derived from $\sH$ by replacing all the coefficients with their complex conjugate. $v=u^\star$ is an invariant subspace of this system. 

We may apply Theorems \ref{thm:gex supp cond} and \ref{thm:global L1} to get global solutions to this system. 
If it has no derivative in the orthogonal directions, then Corollary \ref{GWP-XR} is also applicable for the global wellposedness on the half space. 
The invariance of $v=u^\star$ is immediate from the uniqueness in those results. 
In particular, the simplest model from \cite{AbMu1}: 
\EQ{ \label{NCNLS}
 i\dot u +  u_{xx} = \al u^2 u^\star,}
for $u(t,x):\R^2\to\C$ and a constant $\al\in\C$, 
is globally wellposed for $t\in\R$ in $\F^{-1}\X$, 
as well as in the smaller spaces $\F^{-1}\sD'_{\ge\e}$, for all $\e>0$. 

If $\al\not=0$, this equation has a family of stationary solutions 
\EQ{
 u(x)= \tf{c}{x+iz}=\F^{-1}[-ic1_{\x>0}e^{-z\x}], \pq c:=(2/\al)^{1/2}, \pq z\in(0,\I),}
contained in $\F^{-1}\X \cap H^\I(\R)$, as well as in \eqref{L1-Z} and \eqref{anis-Z} with $\vv=1$. 
The solutions $\tf{c}{x+iz}$ with $z<0$ are in those spaces with $\vv=-1$, 
while the spaces with $\vv=1$ contain more exotic solutions: $\tf{c}{x+i0}=c(p.v.\tf{1}{x}-i\pi\de)=\F^{-1}[-ic1_{\x>0}]\in\sS'(\R)$ and $\F^{-1}[-ic1_{\x>0}e^{z\x}]\not\in\sS'(\R)$ with $z>0$.

One may wonder if the same argument would work for the standard equations with $\bar{u}$ in the nonlinearity. In fact the global results apply to the extended system of $(u,v):=(u,\bar{u})$, 
but the intersection of the invariant subspace $v=\bar{u}$ and $\F^{-1}\X$ is only the trivial $0$ solution. 

\subsection{Completely integrable equations} 
For nonlinear dispersive equations, it is known that complete integrability may drastically lower the regularity needed for the Cauchy problem to be wellposed.  
In particular, for the KdV equation 
\EQ{
 u_t + u_{xxx} = 6uu_x}
for $u(t,x):\R^{1+1}\to\R$, Kappeler and Topalov \cite{KaTo} proved the global wellposedness in $H^{-1}(\R/\Z)$, and Killip and Vi\c{s}an \cite{KiVi} in $H^{-1}(\R)$. 
On the other hand, Molinet \cite{Mo1,Mo2} proved the illposedness in $H^s$ for $s<-1$ on both the domains. 

For mixed initial data, Tsugawa \cite{Ts} proved the local wellposedness in a Banach space of quasi-periodic functions, including initial data in the form 
\EQ{
 u_0 = \sum_{k=1}^N f_k, \pq f_k\in H^s(\R/\ell_k\Z), \pq s>-\tf{1}{2N},}
for any fixed $N\in\N$ and $\{\ell_k\}_{1\le k\le N}\subset(0,\I)$. For almost periodic data (which forms a closed subspace in $L^\I(\R)$), the unique global solution was constructed by Binder, Damanik, Goldstein and Lukic \cite{BiDaGoLu} under some spectral conditions. Laurens \cite{La} proved the global wellposedness in $V(t)+H^{-1}(\R)$, around any fixed solution $V\in C(\R;H^5(\R/\Z))$. 

For the complex-valued solutions, Birnir \cite{Bir} proved blow-up in $L^\I_x$ of the solutions given by Airault, McKean and Moser \cite{AMM}, for the initial data 
$u(0,x) = 6\wp(x+is)$, 
where $\wp$ is the Weierstrass elliptic function with the periods $1,i\om$ for some $\om>0$, and $s\in(-\om,\om)$. The Fourier support is not within $[0,\I)$. 
The solution is in the form 
\EQ{
 u(t,x)=2[\wp(x+2\re z(t)+is)+\wp(x-z(t)+is)+\wp(x-\bar z(t)+is)]} 
for some smooth $z:\R\to\C$ satisfying $z(0)=0$, 
and the blow-up occurs when $\im z(t)\in \pm s+\om\Z$ due to the poles of $\wp$. 
Similarly to Remark \ref{rem:Cole-Hopf}, such singularities are not continuous in $\F^{-1}\sD'(\R)$. 

On the other hand, Corollary \ref{GWP-XR} yields the global wellposedness for $t\in\R$ in $\F^{-1}\X$, where more wild initial data are allowed, under the Fourier support constraint. Indeed, $\F^{-1}\X$ includes all initial data in the form 
\EQ{ \label{u0 Hs many periods}
 u_0= \sum_{n\in\Z} f_n, \pq \CAS{f_0 \in H^{s_0}(\R),\ \supp \F f_0\subset [0,\I), &(n=0),\\ f_n \in H^{s_n}(\R/\ell_n\Z),\ \supp \F f_n\subset(0,\I) &(n\not=0),} }
for any sequences $\{s_n\}_{n\in\Z}\subset\R$ and $\{\ell_n\}_{n\in\Z}\subset(0,\I)$, satisfying 
\EQ{
  \lim_{n\to\I}\ell_n=0, \pq \inf_{\de>0}\sum_{n<0}\|\F f_n\|_{\M(\R^d_{<\de})}<\I.}
Theorem \ref{thm:global L1} also yields unique global solutions. 

Similarly to the previous equation \eqref{NCNLS}, the KdV has stationary solutions 
\EQ{
  u(x) = \tf{1}{(x+iz)^2}=-\F^{-1}[1_{\x>0}\x e^{-z\x}] \pq(z\in\C,\ \re z>0) }
contained in $\F^{-1}\X\cap H^\I(\R)$ and other spaces. 
The solutions $\tf{1}{(x+iz)^2}$ with $\re z<0$ are in the spaces with $\vv=-1$, while the spaces with $\vv=1$ contain also $\tf{1}{(x-a+i0)^2}\in\sS'(\R)$ for $a\in\R$ and $\F^{-1}[1_{\x>0}\x e^{-z\x}]\not\in\sS'(\R)$ for $\re z<0$. 

The same story holds for some other integrable equations, such as the modified KdV equation and the Benjamin-Ono equation. Since our wellposedness results require little for the equations, they 
are convenient for equations derived from algebraic viewpoints. 
For example, consider the deformed AKNS system on $\R^2$
\begin{align}\label{Eq AKNS}
\left\{
\begin{array}{l}
u_{t}  -   u_{xx}   =  2u^2 v + u^2 v(v u_{yy}+ 2u_y v_y) + 2u(vu_{xy} + v_x u_y),  \\
v_{t}  +   v_{xx}   = - 2v^2 u - v^2 u(u v_{yy}+ 2u_y v_y) - 2v (u v_{xy} + u_x v_y),  \\
u(0,x)= u_0(x), \ \ v(0,x)= v_0(x),
\end{array}
\right.
\end{align}
for $(u,v)(t,x,y):\R^{1+2}\to\C^2$, where the second equation is a backward parabolic equation. \eqref{Eq AKNS} is Lax integrable and symmetry integrable, cf. \cite{LoHaJi2022}. 
Theorems \ref{thm:gex supp cond} and \ref{thm:global L1} 
apply to the above system as well. 
Similarly, we may apply those results to the integrable $(3+1)$-dimensional KdV equation in \cite{HaLo2023}:
\EQ{
&u_t - u_{xxx} =  3b u_z (u_x + a_1 u u_y +a_2 u^2 u_z)^2  
\prQ + \p_x (  -2u^3 + \tf{3}{2}a^2 u^2 u_{yy} + 3a u u_{xy} + 3b u^2 u_{xz} + \tf{3}{2}b^2 u^4 u_{zz} + 3ab u^3 u_{zy} ) 
\prQ + a \p_y ( a^2 u^3 u_{yy} -u^4 + \tf{3}{2}a u^2 u_{xy} + \tf{9}{4}ab u^4 u_{zy} +  \tf{6}{5}b^2 u^5 u_{zz} + 2b u^3 u_{xz} ) 
\prQ + b \p_z ( b^2 u^6 u_{zz} - \tf{3}{5}a u^5  + \tf{3}{2}b u^4 u_{xz} +  \tf{3}{4}a^2 u^4 u_{yy} +  \tf{9}{5}ab u^5 u_{zy} + a u^3 u_{xy} ),}
where $u(t,x,y,z):\R^{1+3}\to\C$ and $a,b,a_1,a_2\in\C$ are some constants.

\subsection{Nonlinear hyperbolic conservation laws}
The general systems of nonlinear hyperbolic conservation laws on $\R^d$ 
\begin{align} \label{C laws}
\partial_t u_j + \sum^d_{k=1} \partial_k \sH_{jk}(u) =0\pq(j=1\etc n), \ \ 
\end{align}
for $u:\R^{1+d}\to\C^n$, may be also treated by our results, if each $\sH_{jk}:\C^n\to\C$ is defined and analytic in some neighborhood of $0\in\C^n$. 
Decomposing $\sH$ into the linear and super-linear parts, one may easily see that Theorems \ref{thm:gex supp cond} and \ref{thm:global L1} apply to the above system. 
In fact, it does not matter whether the nonlinearity is in the divergence form. 

\subsection{Nonlinear parabolic equations}
Constantin--Lax--Majda \cite{CoLaMa1985,Sch1986} proposed the following one-dimensional model of the vorticity equation: 
\begin{align}
v_t = \varepsilon v_{xx} + v \cH v, \label{CLMeq}
\end{align}
where $v:\R^{1+1}\to\R$ and $\cH=\F^{-1}(i\sign \x)\F$ is the Hilbert transform.
In the same way as for the KdV, Corollary \ref{GWP-XR} yields the global wellposedness in $\F^{-1}\X$, and Theorem \ref{thm:global L1} yields unique global solutions, both for $t\in\R$ and $\e\in\C$. 
It has also the stationary solutions in the same form, $v(x)=-c\F^{-1}[1_{\x>0}\x e^{-z\x}]$ with $c:=6i\e$.

More generally, we may consider the fractional heat equation
\begin{align}
u_t =   \e(-\Delta)^{s/2} u   +  \sH(\p_x^{\al^1}u\etc \p_x^{\al^m}u), \label{F heat}
\end{align}
for $u:\R^{1+d}\to\C$, where $\e,s\in\C$, $m\in\N$, $\al^1\etc \al^m\in\N_0^d$ are constants and $\sH:\C^m\to\C$ is holomorphic in a $0$-neighborhood, satisfying $\sH(0)=0$. 
Here the boundedness of $L=(-\De)^{s/2}$ as $\x_1\to 0$ requires $\re s\ge 0$. 
If $\re\e>0$ and 
\EQ{
 \bar\al:=\max_{1\le j\le m}\sum_{2 \le k \le d} \al_k^j <\re s,}
then the condition \hyptag{C-s} with $\te=\bar\al/\re s$ and $\vv=e_1$ is satisfied by the equation of $\LR{\na}^{\bar\al}u$, so Corollary \ref{GWP-XR} yields the global wellposedness for $u\in\LR{\na}^{-\bar\al}\F^{-1}\X$ on $t\ge 0$. 
If $\ti\al:=\max_{1\le j\le m}|\al^j|<\re s$, then it is independent of the direction: the global wellposedness holds for $u\in\LR{\na}^{-\ti\al}\F^{-1}\X(\vv)$ on $t\ge 0$ for any $\vv\in\R^d\setminus\{0\}$. 
Theorem \ref{thm:gex supp cond} is applicable without any restriction on $\{\al^j\}$ and $\e\in\C$ for $\re s\ge 0$. 
If $\re s\ge 1$, then the decay condition of $\hat L$ in \eqref{LMN-p bd} is satisfied, so that Theorem \ref{thm:global L1} can also be applied. 

We may also consider nonlinear diffusion, such as the porous medium equation:
\EQ{
 \pt u_t=\De(u^m),}
or the fast diffusion equation, and the $p$-Laplace heat equation:
\EQ{
 u_t= \Div(|\na u|^{p-2}\na u), \pq |\na u|^{p-2}:=[\sum_{j=1}^d(\p_ju)^2]^{p/2-1}.}
Here $u(t,x):\R^{1+d}\to\R$ and $m,p>0$ are given constants. 

In fact, if $m\ge 2$ is an even integer, or $p\ge 4$ is an even integer, then the right side is already in the nonlinear form of our setting, after extending $u:\R^{1+d}\to\C$, 
so that we may apply Theorems \ref{thm:gex supp cond} and \ref{thm:global L1}. 

In the remaining case of $m$ or $p$, we may expand the power around a fixed constant. 
For the porous medium, let $u=u_*+v$ for some constant $u_*\in\C\setminus\{0\}$. 
Then the equation is rewritten for $v$ as 
\EQ{
 v_t - mu_*^{m-1} \De v = \De \sH_m^{u_*}(v), \pq \sH_m^{u_*}(v):=\sum_{n\ge 2} \tf{\Ga(m+1)u_*^{m-n}}{\Ga(m-n+1)n!}v^n,}
which is in the form to which Theorems \ref{thm:gex supp cond} and \ref{thm:global L1} apply. 
For the $p$-Laplacian, let $u=u_*+\x_*\cdot x+v$ for some constants $u_*\in\C$ and $\x_*\in\C^d\setminus\{0\}$. Then the equation is rewritten for $v$
\EQ{
 \pt v_t - |\x_*|^{p-2}\De v - (p-2)|\x_*|^{p-4}\Div(\x_*\cdot\na v) = \Div \sH_{p,\x_*}(\na v),
 \prq \sH_{p,\x_*}(V):=\sH_{p/2-1}^{|\x_*|^2}(2\x_*\cdot V+|V|^2)(\x_*+V)+(p-2)|\x_*|^{p-4}(\x_*\cdot V)V,}
to which we may apply Theorems \ref{thm:gex supp cond} and \ref{thm:global L1}. 
Of course, we may consider various combinations of those nonlinear terms.

\end{document}